\documentclass[10pt]{article}
\usepackage{amsmath,amssymb}
\usepackage{color}
\usepackage{mathrsfs}
\usepackage{mathptmx}
\usepackage{verbatim}
\usepackage{epsfig}
\usepackage{CJK}
\usepackage{bm}
\usepackage{amsfonts}
\usepackage{amsthm}
\input{epsf.sty}
\usepackage{epsfig}
\usepackage{setspace}
\def\<{\langle}
\def\>{\rangle}
\def\p{\partial}
\def\W{\pmb{|} }
\def\zd{\mathbb Z^d}
\def\zn{\mathbb Z^N}
\def\e{\epsilon}

\def\m{\mathcal}
\def\ti{\tilde}
\def\be{\begin{equation}}
\def\ee{\end{equation}}
\def\ba{\begin{array}}
\def\ea{\end{array}}

\newtheorem{Lemma}{Lemma}[section]

\newtheorem{Corollary}{Corollary}

\newtheorem{thm}{Theorem}[section]

\newtheorem{lem}[thm]{Lemma}

\newtheorem{rem}{Remark}

\numberwithin{equation}{section}

\newcommand{\mi}{\mathbf{i}}

\def\e{\epsilon}

\def\be{\begin{equation}}
\def\ee{\end{equation}}
\def\bee{\begin{equation*}}
\def\eee{\end{equation*}}
\def\br{\begin{eqnarray}}
\def\er{\end{eqnarray}}
\def\p{\partial}
\def\l{\langle}
\def\r{\rangle}

\def\lf{\lfloor}
\def\rc{\rceil}

\begin{document}
\setlength{\parindent}{2em}
\newpage

\noindent

\title{\bf KAM theorem with normal frequencies of finite limit-points  for some shallow water equations\footnote{Supported by NNSF of China. (No.11790272 and No.11421061 ) }}
\author{{Xiaoping Yuan \footnote{Email:$\:$xpyuan@fudan.edu.cn; yuanxiaoping@hotmail.com}}\\
{\em\small School of Mathematical Sciences}\\
{\em\small  Fudan University}\\
{\em\small  Shanghai 200433,  China}\\
}
\date{}
\maketitle


\noindent {\bf Abstract.}  By constructing an infinite dimensional KAM theorem of  the normal frequencies being dense at finite-point, we show that some shallow water equations such as Benjamin-Bona-Mahony equation and the  generalized $d$-Dim.
Pochhammer-Chree equation subject to some boundary conditions possess many (a family of initial values of positive Lebesgue measure of finite dimension) smooth solutions which are quasi-periodic in time.



%


\tableofcontents

\section{Statement of the main result}

\ \

\ \
The general problem discussed here is the persistency of quasi-periodic solutions of linear or integrable partial
 differential equations (PDEs) after Hamiltonian perturbation, which is closely related to the well-known
  Kolmogorov-Arnold-Moser (KAM) theory of finite (lower) dimensional invariant tori in smooth infinite dimensional dynamical systems. In this setting the considered PDEs can be written as  an infinite
dimensional Hamiltonian in the (exterior or interior ) parameter dependent normal form $H_0$ subject to a Hamiltonian perturbation $R$
\be \label{2.1}H=H_0+R=(\omega^0(\xi),y)+ \sum_{j\in \Bbb Z^d}\lambda_j(\xi)z_j\bar z_j+\<B^0(\xi)z,\bar z\>+R(x,y,z,\bar z; \, \xi)
\ee  with the symplectic structure
\be\label{2.2} dy\wedge\;dx+{\bf i}\sum_{j\in\mathbb Z^d} d \bar z_j\wedge\;d z_j,\;\;{\bf i}^2=-1 \ee
where $\xi$ is in some parameter set $\mathcal O_0$ of positive Lebesgue measure, and $(x,y,z,\bar z)$ is in the domain $D_p(s_0,r_0)$ which is to be specified later.
The tangent
frequencies $\omega^0=(\omega_1^0,\cdots,\omega_N^0)$ and the
normal frequencies $\lambda_j$ ($\;j\in \zd$) and the linear operator $B^0$ depend on $ N$ dimensional
parameter vector $\xi\in \mathcal O_0\subset \Bbb R^{N}$. Let
$$\Lambda=\Lambda(\xi)=\text{diag}(\lambda_j(\xi):j\in\zd).$$
  Then the
Hamiltonian equations of motion of $H_0$ with symplectic structure $dy\wedge dx+\mi \,d\bar z\wedge dz $ are
\be\dot x=\partial_y H_0=\omega^0(\xi),\,  \dot y=-\partial_x H_0=0,\, \dot
z={\bf i}(\Lambda(\xi)+B^0(\xi))z,\, \dot {\bar
z}=-{\bf i}(\Lambda(\xi)+B^0(\xi))\bar z. \ee
 Hence, for each $\xi\in \mathcal O_0$, there is an invariant
$N$-dimensional torus \be\mathcal{T}^N_0=\mathbb {T}^N\times
\{y=0\}\times\{z=0\} \times \{\bar z=0\}\ee with rotational frequencies $\omega^0(\xi)$.
The main  aim of KAM theory is to prove the persistence of the torus $\mathcal T^N_0$, for
``most" (in the sense of Lebesgue measure) parameter vector
$\xi\in \mathcal O_0$, under small perturbation $R$ of the Hamiltonian
$H_0$.  Here we give a brief, but not complete at all, history for the infinite dimensional KAM theory to deal with lower dimensional invariant tori.
In that direction, Kuksin \cite{Kuk0}-\cite{Kuk3} and Wayne \cite{Way}  initiated  KAM theory to deal with some partial differential equations of spatial dimension $d=1$ such as 1 dimensional nonlinear Schr\"odinger equation (\cite{Kuk1},\cite{K-P})  and 1 dimensional nonlinear wave equation ( \cite{Posch2}, \cite{Way}). Bourgain\cite{Bour1}-\cite{Bour6} developed a new
method initiated by Craig-Wayne\cite{C-W} to deal with the KAM tori for the PDEs in
high spatial dimension, based on the Newton iteration, Fr\"ohlich-Spencer techniques,
harmonic analysis and semi-algebraic set theory (see \cite{Bour6}). This is called Craig-Wayne-Bourgain (C-W-B) method.
We also mention the work by Eliasson-Kuksin\cite{E-Kuk} and Eliasson-Gr\'ebert-Kuksin\cite{E-Kuk2} where the classical KAM theorem is extended
in the direction of \cite{Eliasson}, \cite{Kuk0}-\cite{Kuk3},  \cite{Posch1} and \cite{Way} to deal with higher spatial dimensional nonlinear
Schr\"odinger equation by introducing elegant analysis of T\"oplitz-Lipschitz operator. The obtained KAM tori by \cite{E-Kuk} is linear stable. In addition, the KAM theory is also developed to deal some
1 dimensional PDEs of unbounded perturbation. See, for example, \cite{BBM,Baldi-Berti-Montalto-2, BBM2,BBP,  Berti-Montalto, Feola-Procesi,Kapp-Posch,Kuk2,liu-Y1,liu-Y2} and \cite{Zhang-Gao-Yuan}, for the details. In all works mentioned as the above, a basic assumption is that the normal frequencies $\lambda_j$'s cluster to infinity, that is, for some $\kappa>0$,
\[\lambda_j\approx |j|^\kappa\to\infty,\;\text{as}\; |j|\to\infty.\]   We will construct a KAM theorem for
\[\lambda_j\approx \varpi+ |j|^{-\kappa}\to \varpi\neq \infty,\;\text{as}\; |j|\to\infty\]
where $\varpi$ is a finite real number.

 The present work will contain the first result on the persistency of quasi-periodic solutions and KAM tori for those PDEs with $\lambda_j$'s clustering to a finite limit-point.
   As for totally new applications of the KAM theorem, we will show that there are many (finite dimensional initial value set of positive Lebesgue measure) KAM tori and  quasi-periodic solutions for
some kinds of shallow water equations such as Benjamin-Bona-Mahony (BBM) equation (with $\varpi=0$):
\[ u_{t}+u_{x}+uu_{x}-u_{{xxt}}=0,\quad x\in \; \text{some compact domain in }\; \mathbb R\]
and the  generalized
Pochhammer-Chree  (gPC) equation (with $\varpi=1$)
\[u_{tt}-\Delta u_{tt}-\Delta u+\Delta \, u^3=0,\; x\in \; \text{some compact domain in}\; \mathbb R^d,\; d\ge 1.\]

\ \

  In order to state our results, let us introduce some notations. Denote by $d\ge 1$ the spatial dimension of those partial differential equations
to be considered.
%
For a $d$-dimensional integer vector $j=(j_1,...,j_d)\in \mathbb Z^d$, define $|j|=|j_1|+\cdots+|j_d|$.
For $x\in\mathbb R^N$ or $x\in\mathbb C^N$, we also by $|x|$ denote the Euclidean norm of $x$.
Given $p>d/2, \kappa>0,$ let $q=p+\kappa.$
 For  $\ti p\in\{p,q\}$, define
\[h_{\ti p}=\{z=(z_j\in\mathbb C:j\in\mathbb Z^d):\quad ||z||_{\ti p}^2=\sum_{j\in\mathbb Z^d}|z_j|^2|j|^{2\ti p}\},\] where we take $|j|=1$ as $j=0$ for convenience.
This is a Hilbert space  with a natural inner product corresponding to the norm $||\cdot||_{\ti p}$.
Given an integer $N>0,$ let $\mathbb T^N_{s}$ be the complexization of $\mathbb T^N=\mathbb{R}^{N}/(2\pi \mathbb{Z})^{N}$ with width $s>0$:
\[\mathbb T_{s}^N=\{x\in\mathbb C^N/(2\pi\mathbb Z^N):\quad |\Im{x}|\le s\}.\]
%
For given $s_0>0$ define
\[\mathcal {P}^{\tilde p}=\mathbb T^N_{s_0}\times\mathbb C^N\times h_{\ti p}\times h_{\ti p},\;\; \ti p\in\{p,q\}.\]
Take $\m P^p$ as phase spaces. Denote by $\mathcal L (h_{\ti p},h_{\ti q})$ the set consisting of all bounded linear operator from $h_{\ti p}\to h_{\ti q}$ where $\tilde p,\; \tilde q\in\{p,q\}$. Introduce a $D_{ p}(s_0,r_0)$-neighborhood of the torus $\mathbb T^N\times\{0\}\times\{0\}\times\{0\}$ in the phase space $\mathcal {P}^{p}$:
\[D_{p}(s_0,r_0)=\mathbb T^N_{s_0}\times\{y\in\mathbb C^N:|y|<r_0^2\}\times
\{z\in h_{p}:||z||_{p}<r_0\}\times \{\bar z\in h_{p}:||\bar z||_{p}<r_0\}.\]
 By $X_{\m G}$ denote the Hamiltonian vector field for a Hamiltonian function $\m G$ defined in $D_{p}(s_0,r_0)$ or some sub-domain of $D_{p}(s_0,r_0)$ The phase space $\m P^p$ endowed with
\eqref{2.2} is a symplectic space.
For two vectors
$b,c\in \Bbb C^\iota$ or $\Bbb R^\iota$, we write
$(b,c)=\sum_{j=1}^\iota b_jc_j$ if $\iota<\infty$. If the index $j\in\mathbb Z^d$, we write $\<b,c\>=\sum_{j\in\mathbb Z^d} b_jc_j$.

\noindent{\it\bf Assumption A:} {\it (Non-degeneracy.)  {\it  Assume that $\omega^0(\xi)=(\omega_1^0(\xi),...,\omega_N^0(\xi)):\; \m O_0\subset\mathbb R^N\to\mathbb R^N$ is real continuously differentiable in $\xi\in\mathcal O_0$ in the sense of Whitney.\footnote{In the following arguments, the differentiability with respect to the parameter $\xi$ is always in the sense of Whitney. We will not mention it again. An alternative way is replacing the Whitney smoothness by Lipschitzian continuity. } And assume there are two absolute constants $c_1,c_2>0$ such that
\be\label{2.6} \inf_{\xi\in\mathcal O_0}|\text{det }\p_{ \xi}\omega^0(\xi)|\ge c_1,\ee
\be \label{2.7}\sup_{\xi\in{\mathcal O}_0}|\p_{ \xi}\,\omega^0(\xi)|\le c_2.\ee
}

\begin{rem} Let $\omega^0(\xi)=\omega$. By  Assumption {\bf A}, we can regard $\omega$ as parameter, and $\xi=(\omega^0)^{-1}(\omega):\; \omega^0(\m O_0)\to \m O_0$. Or assuming $\omega(\xi)\equiv \xi$ without loss of generality.

\end{rem}

\noindent{\it\bf Assumption B} {\it (The normal frequencies clustering at the origin $0\in\mathbb{R}$, i.e. $\varpi=0$.)  Assume $\lambda_j=\lambda_j(\xi)$'s are  real and continuously differentiable in $\xi\in\mathcal O_0$.  Assume  that there exist constants $c_{11},c_{12}, c_{13}>0$ and $\kappa>0$ such that
\be\label{3-08-morning} c_{11}|j|^{-\kappa}\leq\lambda_j(\xi)\leq c_{12}|j|^{-\kappa},\quad \forall\; \xi\in\m O_0,\;\; j\in\mathbb Z^d \ee
and
\be\label{170512-1} \sup_{\xi\in{\mathcal O}_0}|\p_{ \xi}\,\lambda_j(\xi)|\le c_{13}\, |j|^{-\kappa}\, ,\;\;j\in\mathbb Z^d.\ee
 And assume
 for every $0\neq k\in\mathbb Z^n$  and $i,j\in\mathbb Z^d$,
\be\label{170510-2} \frac{d^*}{d\,\omega}\left(  (k,\omega)\pm\lambda_i\right)> 0,\;\; \frac{d^*}{d\,\omega}\left(  (k,\omega)\pm(\lambda_i+\lambda_j\right)> 0,\ee
where $\lambda_i=\lambda_i(\xi(\omega))=\lambda_i((\omega^0)^{-1}(\omega))$ and $\frac{d^*}{d\,\omega}$ is the directional derivative along the direction such that $\frac{d^*}{d\,\omega} (k,\omega)\ge 1$.

\begin{rem} Note that $\frac{d^*}{d\,\omega} (k,\omega)\ge 1$.  By Assumption {\bf A} and \eqref{170512-1}, there is $j_0>0$ such that, for all $i,j\in\mathbb{Z}^d $ with $|i|,|j|\ge j_0$, the inequality \eqref{170510-2} holds true automatically, and  there is a constant $c^*>0$ such that
\be\label{usefortwist} \frac{d^*}{d\,\omega}\left(  (k,\omega)\pm\lambda_i\right)> c^*,\;\; \frac{d^*}{d\,\omega}\left(  (k,\omega)\pm\lambda_i\pm \lambda_j\right)> c^*
. \ee In addition, if $c_{13}$ is small enough, the inequality \eqref{170510-2} holds true automatically, too.

\end{rem}

\begin{rem}\label{2018-2-28} By \eqref{3-08-morning} and $q=\kappa+p$, we have that for any linear operator $\m A:\; h_p\to h_q$
\[||\m A||_{h_p\to h_q}\le C ||\Lambda^{-1}\, \m A||_{h_p\to h_p},\; ||\Lambda^{-1}\, \m A||_{h_p\to h_p}\le C\, ||\m A||_{h_p\to h_q},\; ||\m A\, \Lambda^{-1}||_{h_q\to h_q}\le C\, ||\m A||_{h_p\to h_q}.\]
\end{rem}


\vspace{4pt}{\rm
We will consider a Hamiltonian perturbation $R^0$ of the integrable Hamiltonian $H_0$.
In order to define the size of the perturbation $R^0$, we need to introduce some norms. Arbitrarily take $s$ and $r$ with $s_0\ge s>0,r_0>r>0$, $\tilde p, \ti q\in\{p,q\}$, and
arbitrary take a subset $\m O\subset \m O_{0}$.
For a map $f(x,\xi):\mathbb{T}_{s}^{N}\times \m O \rightarrow \mathbb{C}^{N},$ define
$$|f|^{2}_{s,\mathcal{O}}=\sup_{\xi\in\m O}\sum_{k\in \mathbb{Z}^{N}}|\widehat{f}(k,\xi)|^{2}e^{2|k|s},$$ where $\widehat{f}(k,\xi)$ is the $k$-Fourier coefficient of $f(x,\xi)$ in $x\in \mathbb T^N$.
For a map $f:\mathbb{T}_{s}^{N}\times \m O \rightarrow h_{\tilde{p}},$ define
$$||f||^{2}_{\tilde{p}, s, \mathcal{O}}=\sup_{\xi\in\m O}\sum_{k\in \mathbb{Z}^{N}}||\widehat{f}(k, \xi)||^{2}_{\tilde{p}}\, e^{2|k|s},\;\;\tilde{p}\in \{p, q\}.$$
Consider a map $f:\; D_{\ti p}(s,r)\times\m O\to h_{\tilde q}$ with $\ti p,\, \ti q\in\{p,q\}$. For $(x,y,z,\bar z)\in D_{\ti {p}}(s,r)$, write $f$ into Fourier series in $x$:
\[f(x,y,z,\bar z;\xi)=\sum_{k\in\zn}\widehat{f}(k;y,z,\bar z;\xi)\; e^{\mi \; (k,x)}.\]
Define
 \be ||f(\cdot,y,z,\bar z;\xi)||_{\tilde q,s}^2:=\sum_{k\in\zn}e^{2|k|s}\;||\widehat{f}(k;y,z,\bar z;\xi)||^2_{{\tilde q}}. \ee
 The bound  $||f(\cdot,y,z,\bar z;\xi)||_{\tilde q,s}<\infty$ implies that $f(x,y,z,\bar z;\xi)$ as a function of $x$ with its range in $h_{\ti q}$ is analytic in $|\Im x|<s$ and continuous in $|\Im x|=s$. Let
\be ||f||_{\tilde p,\ti q,s,r,\m O}:=\sup_{\xi\in\m O,|y|< r^2,||z||_{\ti p}< r,||\bar z||_{\ti p}< r}||f(\cdot,y,z,\bar z;\xi)||_{{\tilde q},s}.\ee
For $f:\;D_{\ti p}(s,r)\times\m O\to\mathbb C^N$, define

\be  |f|_{\ti p,s,r,\m O}:=\sup_{\xi\in\m O,|y|< r^2,||z||_{\ti p}< r,||\bar z||_{\ti p}< r}\sqrt{\sum_{k\in\zn}|\widehat{f}(k;y,z,\bar z;\xi)|^2e^{ 2|k|s}}.\ee
For a map $$W=(X,Y,Z,\bar Z):\; D_{\ti p}(s,r)\times\mathcal O\subset\mathcal P^{\ti p}\times\m O\to\mathcal P^{\tilde q},$$ define
\be\label{20170729}\W W\W_{\tilde q,D_{\ti p}(s,r)\times\mathcal O}:=\sqrt{|X|_{\tilde p,s,r,\m O}^2+
|Y|_{\tilde p, s,r,\m O}^2+||Z||_{\tilde p, \tilde q,s,r,\m O}^2+||\overline{Z}||_{\tilde p, \tilde q,s,r,\m O}^2} .\ee
%
Following \cite{Bambusi-Grebert} with a minor modification, we introduce the modulus of a function (  vector, matrix or  operator). For a scalar complex value function defined on $D_ {\ti p}(s_0,r_0)\times \m O_0,$ $\ti {p}\in \{p, q\},$
\[f(x,y,z,\bar z;\xi)=\sum_{k\in\mathbb Z^N,\gamma\in\mathbb Z^N_+\alpha,\beta\in\mathbb Z_+^{\mathbb{Z}^d}}\, f_{k,\gamma,\alpha,\beta}(\xi)e^{\mi (k,x)}\,y^\gamma\,z^{\alpha}\, \bar z^{\beta},\]
we define
\[\lfloor f \rceil= \lfloor f(x,y,z,\bar z;\xi)\rceil=\sum_{k\in\mathbb Z^N,\gamma\in\mathbb Z^N_+\alpha,\beta\in\mathbb Z_+^{\mathbb{Z}^d}} \left|f_{k,\gamma,\alpha,\beta}(\xi)\right|\,e^{\mi ( k,x)}\,y^\gamma\,z^{\alpha}\, \bar z^{\beta},\]
where $\mathbb Z_+$ consists of all non-negative integers. For a vector $$f(x,y,z,\bar z;\xi)=(f_j(x,y,z,\bar z;\xi)\in\mathbb C:\, j\in \mathbb Z^d \; \text{or a subset of }\; \mathbb Z^d),$$ define
\[\lfloor f \rceil=(\lfloor f_j(x,y,z,\bar z;\xi)\rceil :\, j\in \mathbb Z^d \; \text{or a subset of }\; \mathbb Z^d)).\]
For an operator or matrix $$f(x,y,z,\bar z;\xi)=(f_{ij}(x,y,z,\bar z;\xi)\in\mathbb C:\, i,j\in \mathbb Z^d \; \text{or a subset of }\; \mathbb Z^d),$$ define
$$\lfloor f(x,y,z,\bar z;\xi) \rceil =(\lfloor f_{ij}(x,y,z,\bar z;\xi)\rceil :\, i,j\in \mathbb Z^d \; \text{or a subset of }\; \mathbb Z^d).$$
In the whole of this paper, we denote by $C$ a universal constant which may be different in different places and which is independent of the steps of KAM iterations.

\vspace{4pt}

\noindent{\it\bf Assumption C:} {\it (Regularity.) 
Assume the perturbation term $ R^0(x,y,z,\bar z;\xi)$ which
is defined on the domain $D_{p}(s_0,r_0)\times\mathcal O_0$ is analytic in
the space coordinates $(x,y,z,\bar z)$ and $C^1$-smooth in $\xi$ of the parameter vector $\xi\in\mathcal O_0$, and  for each $\xi\in \mathcal O_0$,
the modulus $\lfloor X_{R^0}\rceil$ of
 its Hamiltonian vector field $$X_{ R^{0}}:=(
R^{0}_y,- R^{0}_x, {\bf i}\;\p_{\bar z} R^{0},-{\bf i}\;\p_{ z}
R^{0}),$$ defines
 a analytic map
$$\lfloor X_{ R^{0}}\rceil :D_p(s_0,r_0)\subset\mathcal P^{p}\to\mathcal P^{q}
$$
satisfying
\be \label{2.10}\W \lfloor X_{R^0}\rceil\W_{q,D_p(s_0,r_0)\times\mathcal O_0} \le C\,\epsilon_0,\; \W \lfloor \p_\xi\, X_{R^0}\rceil\W_{q,D_p(s_0,r_0)\times\mathcal O_0} \le C\, \epsilon_0.\ee
}


\vspace{4pt}

\noindent{\it\bf Assumption D:} {\it (Reality.) The Hamiltonian functions $H_0(x,y,z,\bar z;\xi)$ and $R^0(x,y,z,\bar z;\xi)$ are real when  $x$ and $y$  are real and $\bar z$ is the complex conjugate of $z$.}

 \begin{rem}This condition implies that $\omega_0(\xi)$ and $\lambda_j(\xi)$'s  are real  and that $B^0(\xi)$ is self-adjoint in the square-summable space $\ell_2$ and
$$\overline{R^0(x,y,z,\bar z;\xi)}=R^0(x,y,\bar z, z;\xi).$$
\end{rem}


%
%
%
\vspace{4pt}

\noindent{\it\bf Assumption E:} {\it For any $\xi\in\mathcal{O}_0$, the modulus of the operator $B_0$ is relatively small with respect to $q=p+\kappa$ in the following sense:
$$\sup_{\xi\in\mathcal{O}_0}||\lfloor B^0(\xi)\rceil||_{h_{p}\rightarrow h_{q}}\leq C \,\epsilon_0,\;
\sup_{\xi\in\mathcal{O}_0}||\lfloor\p_\xi\,B^0(\xi)\rceil||_{h_{p}\rightarrow h_{q}}\leq C\, \epsilon_0.$$}
\begin{rem} If the modulus $\lf B\rc=(|B_{ij}|:\;i,j\in\mathbb Z^d)$ of a linear operator $B$ is bounded from $h_p\to h_q$, it implies that $B$ is absolutely bounded. See \cite{Halmos-Sunder} for the notation ``absolute boundedness".
\end{rem}

%

\begin{thm} \label{theorem2}  ($\varpi=0$). Suppose that the Hamiltonian $H=H_0+R^0$ obeys the assumptions $\bf{A, B, C, D, E.}$ Then
there is a sufficiently small $\epsilon^*=\epsilon^*(N,p,q,d,\mathcal{O}_0)>0$ such that for
any $0<\epsilon_0<\epsilon^*$ there is a subset $\mathcal{O}\subset\mathcal{O}_0$  and a function $\gamma=\gamma(\e_0) $ with
\begin{equation}\label{1}\mbox{Meas}\mathcal{O}\geq(\mbox{Meas}\mathcal{O}_0)(1-O(\gamma(\epsilon_0))),
\quad \lim_{\e_0\to 0}\gamma(\e_0)=0,\end{equation}
and there is a symplectic coordinate change
\begin{equation}\label{2}\Phi:D_p(\frac{s_0}{2},\frac{r_0}{2})\times\mathcal{O}
\subset\mathcal{P}^{ p}\rightarrow
D_{ p}(s_0,r_0)\times\mathcal{O}_0\subset\mathcal{P}^{ p},\end{equation}
such that $H=H_0+R^0$ is changed into
\begin{equation}\label{3}H^{\infty}:=H\circ\Phi=(\omega(\xi),y)+\sum_{j\in\mathbb{Z}^d}\lambda_j(\xi)z_j\bar{z}_j+\langle B^{\infty}(\xi)z,\bar{z}\rangle+R,\end{equation}
where \begin{equation}\label{4}R=O(|y|^2+|y|||z||_{p}+||z||_{p}^3),\end{equation}

 \begin{equation}\label{5}|\lf X_R\rc |_{q,D_p(s_0/2,r_0/2)\times\mathcal{O}}\leq C\epsilon_0,\; |\lf \p_\xi\,X_R\rc|_{q,D_p(s_0/2,r_0/2)\times\mathcal{O}}\leq C\epsilon_0,\end{equation}
and for any $\xi\in\mathcal{O}$, the operator $B^{\infty}(\xi)$  obeys
\begin{equation}
\label{7}||\lf B^{\infty}(\xi)-B^{0}(\xi)\rc||_{h_{p}\rightarrow h_{q}}\leq C\epsilon_0,\;
||\lf \p_\xi(B^{\infty}(\xi)-B^{0}(\xi))\rc||_{h_{p}\rightarrow h_{q}}\leq C\epsilon_0,\end{equation}
%
and $\omega:\mathcal{O}\rightarrow\mathbb{R}^N$ with
\begin{equation}\label{9}\sup_{\xi\in\mathcal{O}}|\omega-\omega^0|\leq C\epsilon_{0},\quad \sup_{\xi\in\mathcal{O}}|\partial_{\xi}(\omega-\omega^0)|\leq C\epsilon_{0}.\end{equation}

\end{thm}

\begin{rem} The function $\gamma=\gamma(\e_0)$ can be improved to $\gamma(\e_0)=\e_0^{1/3}$ when the measure of the set $\m O_0$ is independent of $\e_0$.
When $\m O_0$ depends on $\e_0$, we consider $\m O_0=[0,\sqrt{\e_0}]^N$, for example. Then $\gamma(\e_0)=1/|\log\, \e_0|$, and at the same time, \eqref{5} should be replaced by
\begin{equation}\label{5+}|\lf X_R\rc |_{q,D_p(s_0/2,r_0/2)\times\mathcal{O}}\leq C\epsilon_0,\; |\lf \p_\xi\,X_R\rc|_{q,D_p(s_0/2,r_0/2)\times\mathcal{O}}\leq C\sqrt{\epsilon_0},\end{equation}
 Thus we always have that $\lim_{\e_0\to 0} \, \text{Meas}\, \m O=\text{Meas}\, \m O_0$.

\end{rem}

\noindent{\it\bf Assumption ${\bf B}^\star$} {\it Let $\varpi\neq 0$ be a finite real number.
  Assume $\lambda_j=\lambda_j(\xi)$'s are  real and continuously differentiable in $\xi\in\mathcal O_0$.  Assume  that there exist constants $c_{11},c_{12}, c_{13}>0$ and $\kappa>0$ such that
\be\label{3-08-morning+} c_{11}|j|^{-\kappa}\leq|\lambda_j(\xi)-\varpi|\leq c_{12}|j|^{-\kappa},\quad \forall\; \xi\in\m O_0,\;\; j\in\mathbb Z^d \ee
and
\be\label{170512-1+} \sup_{\xi\in{\mathcal O}_0}|\p_{ \xi}\,\left(\lambda_j(\xi)-\varpi\right)|\le c_{13}\, |j|^{-\kappa}\, ,\;\;j\in\mathbb Z^d.\ee
 And assume
 for every $0\neq k\in\mathbb Z^n$  and $i,j\in\mathbb Z^d$, there is a constant $c^\star$ such that
\be\label{170510-2}\left| \frac{d^*}{d\,\omega}\left(  (k,\omega)\pm\lambda_i\right)\right|> c^\star,\;\; \left|\frac{d^*}{d\,\omega}\left(  (k,\omega)\pm(\lambda_i+\lambda_j\right)\right|> c^\star,\ee
where $\lambda_i=\lambda_i(\xi(\omega))=\lambda_i((\omega^0)^{-1}(\omega))$ and $\frac{d^*}{d\,\omega}$ is the directional derivative along the direction such that $\frac{d^*}{d\,\omega} (k,\omega)\ge 1$. }

\begin{thm} \label{theorem2+2} ($\varpi\neq 0$).  Suppose that the Hamiltonian $H=H_0+R^0$ obeys the assumptions $\bf{A,  C, D, E}$ and ${\bf B}^\star$.  Then the result of Theorem \ref{theorem2} holds true.
\end{thm}
\begin{Corollary}\label{Cor1}
For any $\xi\in\mathcal O$, $\Phi(\mathcal T_0)=\Phi(\mathbb T^N\times\{y=0\}\times
\{z=0\}\times\{\bar z=0\})$ is an invariant torus with rotational frequency $\omega$ for the original Hamiltonian $H=H_0+R^0$. The torus carries quasi-periodic solutions with frequency $\omega$ for $H$.
\end{Corollary}

\begin{Corollary}\label{Cor2}
The obtained KAM tori and quasi-periodic solutions are linearly stable.
\end{Corollary}
\begin{rem} Since we can by the symplectic transformation $\Phi$ eliminate
the quadratic terms (linear part in vector field)
$\langle R^{zz}(x,\xi)\, z,z\rangle+\langle R^{\bar z\bar z}(x,\xi)\,
\bar z,\bar z\rangle$ and reduce $\langle R^{z\bar z}(x,\xi)\,
z,\bar z\rangle$ to $\langle B^\infty(\xi)\, z,\bar z\rangle$ where $ B^\infty(\xi)$ is independent of angle variable  $x$,,
all the obtained KAM tori are linearly stable.
However, Not all the persisted tori by KAM technique are linearly stable for PDEs of high spatial
dimension $d\ge 2$ when $\lambda_j\to \infty$. Recently, Eliasson-Grebert-Kuksin\cite{E-Kuk2}
construct explicit examples of partially hyperbolic KAM tori for $d\ge 2$ dimensional beam equation where the nonlinear perturbation is of the very general form, especially depending on the spatial variable $x\in\mathbb T^d$. Those hyperbolic KAM tori are situated in the neighborhood of the origin and create around them some local instabilities. It seems possible that the  local instabilities lead global instabilities (chaotic motion) by searching for Smale horseshoes via the partially hyperbolic KAM tori.

\end{rem}
\ \


While constructing the (classical) lower dimensional KAM tori, one always needs the first Melnikov conditions
\[\Delta_{kj}:=\langle k,\omega\rangle+\lambda_j\neq 0,\quad \forall\; k\in\mathbb Z^N,\; \forall\; j\in\mathbb Z^d\]
and the second Melnikov conditions
\[\Delta_{kij}:=\langle k,\omega\rangle+\lambda_j-\lambda_i\neq 0,\quad \forall\; k\in\mathbb Z^N,\; \forall\; i\neq j\in\mathbb Z^d.\]
Refer to \cite{Mel} for the Melnikov conditions. Due to the analyticity of the perturbation, one can assume
\[|k|\le K=K_m\approx 2^m\]
where $m$ is the step number of the Newton iteration.
\begin{itemize}

\item For nonlinear Shr\"odinger equation, the normal frequencies $$\lambda_j=\lambda_j^{NLS}=|j|^2,\; j\in\mathbb Z^d,$$ for example. When $|j|>C K$ with $C\gg |\omega|$, we have
    \[\left| \Delta_{kj} \right|>\lambda_j^{NLS}-|\langle k,\omega\rangle|>C|j|^2-|\omega| \, K>1,\]
which is not small.Thus the number of small divisors $\Delta_{kj}$ is {\it finite} in the first Melnikov conditions. It is worth to point out that the first Melnikov conditions are unavoidable in both the classical KAM theory for the lower dimensional invariant tori and the KAM developed by Craig-Wayne-Bourgain.   On the other hand, for BBM and gPC equations, $\lambda_j=\lambda_j^{shallow}=\varpi+|j|^{-\kappa}$ with $\kappa>0$. At this time, the normal frequencies $\lambda_j^{shallow}$ cluster to a finite point $\varpi$.  We have
 \[\Delta_{kj} =\langle k,\omega\rangle+\lambda_j^{shallow}\to \langle k,\omega\rangle+\varpi,\; \text{as}\; |j|\to\infty.\] Note that zero is a limit point of the set $\{\langle k,\omega\rangle+\varpi:\; k\in\mathbb Z^N\}$. Thus, the number of small divisors $\Delta_{kj}$ is {\it infinite} in the first Melnikov conditions for BBM and gPC equations.

\item For NLS, the frequencies $\lambda_j^{NLS}$ has the gap property in sense that there is a constant $C$ such that $|\lambda_i^{NLS}-\lambda_j^{NLS}|>C$ when $\lambda_i^{NLS}\neq\lambda_j^{NLS}$. Clearly, for BBM and gPC, the gap property does not hold true, since
    \[\lambda_j^{shallow}-\lambda_i^{shallow}\to 0,\;\;\text{as}\;\; |i|,\, |j|\to\infty.\]
        Incidentally, the gap property does not hold true for nonlinear wave equation of spatial dimension
$d\ge 2$. Thus it is an open problem that whether or not there is classical KAM tori which are linearly stable for the  nonlinear wave equation with $d\ge 2$.
\end{itemize}
\section{Solution of linear equation for the first Melnikov conditions}

For a vector (or matrix) value function $f$ defined in $\mathbb T^N_s$ and a large number $K>0$, introduce a cut-off operator $\Gamma=\Gamma_K$ as follows:
\be \label{K}(\Gamma f)(x)=(\Gamma_Kf)(x):=\sum_{|k|\le K}\widehat{f}(k)e^{{\bf i}(k,x)},\ee
where $\widehat{f}(k)$ is the $k$-Fourier coefficient of $f(x)$.

\begin{Lemma} \label{lem4.1} Replacing $\m O_0$, $\omega^0$ and $B^0$ by $\m O$, $\omega$ and $B$. Assume  $\omega=\omega(\xi)$, $\lambda_j=\lambda_j(\xi)$ (here $j\in\mathbb Z^d$) defined in $\m O$ satisfy Assumptions {\bf A, B}, respectively.  And assume $B$ satisfies Assumptions {\bf E}.
In addition, assume
\be R(x,\xi):\;\mathbb T^N_s\times\m O\to h_q\ee
is analytic in $x\in\mathbb T^N_s$, $C^1$ in $\xi\in\m O,$ and the average
 \be\label{17-3-10-1}\int_{\mathbb T^N} R(x,\xi)\, d x=0,\quad \forall\;\; \xi\in\mathcal O.\ee
 Then there is a subset $\m O_1\subset\m O$ with
\be \text{Meas}\;\m O_1=(\text{Meas}\;\m O)(1-O(K^{-C}))\ee such that
for any $\xi\in\m O_1$, the homological equation
\be \label{4.9}\Gamma\left(({\bf i}\omega\cdot\p_x+\Lambda+ B(\xi))F(x,\xi) \right)=(\Gamma R)(x,\xi)\ee
has unique solution
\be F(x,\xi):\;\mathbb T^N_{s^\prime}\times\m O_1\to h_q\ee
 with $\Gamma F=F$ and
\be\label{2.9} ||\lfloor F\rceil||_{q,s^\prime,\m O_1}\le  K^C||\lfloor R\rceil||_{q,s,\m O}\ee and
\be \label{2.10} ||\lfloor \p_\xi F\rceil||_{q,s'',\m O_1}\le  K^C(||\lfloor R\rceil||_{q,s,\m O}+||\lfloor\p_\xi R\rceil||_{q,s,\m O}),\ee
where $0<s''<s'\le s$.

\end{Lemma}

\vspace{5pt}
\noindent{\bf Proof:}
$\;\;$By passing \eqref{4.9} to Fourier coefficients, we have
 \begin{equation}\label{61} (-( k,\omega)+\Lambda+{B})\widehat{{F}}(k)=\widehat{{R}}(k),\quad
 \forall \;k\in\mathbb{Z}^N,\;0<|k|\leq K,\end{equation}
 where $( k,\omega)=( k,\omega) E$ with $E$ being the identity from $h_{q}\rightarrow h_{q}$. In the following, we always by $1$ instead of $E$ denote the identity from $h_{\tilde p}$ to $h_{\tilde p}$ ($\tilde p\in\{p,q\}$) or from some finite dimensional space to itself, and write $E x=1 x=x$.

According to Assumption {\bf A}, assume $\omega(\xi)=\xi$ without loss of generality. In standard procedure, it can be easily proved that there is a subset $\mathcal{O}_1\subset\mathcal{O}$ of $\mbox{Meas}(\mathcal{O}_1)\leq K^{-c_{20}}$
  with constant $c_{20}>0$ such that
  \begin{equation}\label{62}|(k,\omega(\xi))|\geq K^{-c_{21}},\quad 0<|k|\leq K,
   \xi\in\mathcal{O}\setminus\mathcal{O}_1,\end{equation}%
  with constant $c_{21}=c_{20}+N$.
  Recall the Assumption {\bf B}: \begin{equation}\label{63}c_{11}|j|^{-\kappa}\leq\lambda_{j}\leq c_{12}|j|^{-\kappa}.\end{equation}
  So we choose a constant $c_{22}$ with $c_{22}\gg c_{21}$ such that
  \begin{equation}\label{64}|\lambda_j|\leq c_{12} K^{-\kappa\,c_{22}}\ll K^{-c_{21}}, \quad \mbox{for}\; |j|\geq K^{c_{22}}.\end{equation}
 Using $K^{c_{22}}$, we partition $\Lambda$ as follows: $$\Lambda=\Lambda^{(1)}\oplus\Lambda^{(2)}=\left(\begin{array}{ll}
\Lambda^{(1)}&0\\
0&\Lambda^{(2)}
\end{array}\right),$$ where $\Lambda^{(1)}=\mbox{diag}(\Lambda_j:|j|<K^{c_{22}})$, $\Lambda^{(2)}=\mbox{diag}(\Lambda_j:|j|\geq K^{c_{22}})$. In such way,
we partition ${{B}}:$
$${{B}}=\left(\begin{array}{ll}
B^{(11)}&B^{(12)}\\
B^{(21)}&B^{(22)}
\end{array}\right),$$
where $$B^{(11)}=(B_{ij}:|i|<K^{c_{22}},|j|<K^{c_{22}}),\quad B^{(21)}=(B_{ij}:|i|\geq K^{c_{22}},|j|<K^{c_{22}}),$$
$$B^{(12)}=(B_{ij}:|i|<K^{c_{22}},|j|\geq K^{c_{22}}),\quad B^{(22)}=(B_{ij}:|i|\geq K^{c_{22}},|j|\geq K^{c_{22}})$$
with $B_{ij}=B_{ij}(\xi)$ being the elements of matrix $B=B(\xi)$.
Again in such way, partition the $k$-Fourier coefficients $\widehat{F}(k)$ of $F(x)$ and $\widehat{{R}}(k)$ of $R(x)$: $$\widehat{{F}}(k)=\left(\begin{array}{ll}
F^{(1)}\\
F^{(2)}
\end{array}\right),\;\;\widehat{{R}}(k)=\left(\begin{array}{ll}
R^{(1)}\\
R^{(2)}
\end{array}\right),$$
where $F^{(1)}=(\widehat{{F}}_j(k):\, |j|<K^{c_{22}})$ and  $F^{(2)}=(\widehat{{F}}_j(k):\, |j|\ge K^{c_{22}})$ and so on.
And we split $h_p=h_p^{01}\oplus h_p^{02}$, $h_q=h_q^{01}\oplus h_q^{02}$ and $\ell_2=\ell_2^{01}\oplus \ell_2^{02}$ by advantage of $K^{c_{22}}$.
For example,
\[h_p^{01}=\{z=(z_j\in\mathbb C):j\in\mathbb Z^d, |j|< K^{c_{22}}\}\]
with
\[||z||_{h_p^{01}}^2=\sum_{j\in\mathbb Z^d, |j|< K^{c_{22}}} |z_j|^2|j|^{2p},\]
and
\[h_p^{02}=\{z=(z_j\in\mathbb C):j\in\mathbb Z^d, |j|\ge  K^{c_{22}}\}\]
with
\[||z||_{h_p^{02}}^2=\sum_{j\in\mathbb Z^d, |j|> K^{c_{22}}} |z_j|^2|j|^{2p}.\]
Therefore, by the partition as the above, we write
\begin{equation}\label{66}-(k,\omega)+(\Lambda+\widehat{{B}}(k))=\left(\begin{array}{ll}
-(k,\omega)+\Lambda^{(1)}+B^{(11)}&B^{(12)}\\
B^{(21)}&-( k,\omega)+\Lambda^{(2)}+B^{(22)}
\end{array}\right).\end{equation}
 Let \begin{equation}\label{79}\mathcal{B}_1:=B^{(11)}-B^{(12)}\mathcal{B}_2^{-1}B^{(21)},\end{equation}
\begin{equation}\label{80}\mathcal{B}_2:=-( k,\omega)+\Lambda^{(2)}+B^{(22)}.\end{equation}
Then by \eqref{61}
\begin{equation}\label{77}F^{(1)}=(-( k,\omega)+\Lambda^{(1)}+
\mathcal B_1)^{-1}(R^{(1)}-B^{(12)}\mathcal B_2^{-1}R^{(2)}),\end{equation}
 \begin{equation}\label{78}F^{(2)}=\mathcal B_2^{-1}R^{(2)}-\mathcal B_2^{-1}B^{(21)}F^{(1)}.\end{equation}
In view of Assumption {\bf E} by replacing $B^0$ by  $B$ and using Remark \ref{2018-2-28}, one has
\be\label{67} \sup_{\xi\in\mathcal{O}}||\Lambda^{-1}B||_{h_{p}\rightarrow h_{p}}\leq \sup_{\xi\in\mathcal{O}}||\Lambda^{-1}\lf B\rc||_{h_{p}\rightarrow h_{p}} \le C\,\epsilon_0.\ee
Similarly,
\be\label{67+1}\sup_{\xi\in\mathcal{O}}||
\Lambda^{-1}\,\partial_{\xi}\,B||_{h_{p}\rightarrow h_{p}}\leq \sup_{\xi\in\mathcal{O}}||
\Lambda^{-1}\,\lfloor\partial_{\xi}\,B\rceil||_{h_{p}\rightarrow h_{p}}\leq C\,\epsilon_0.\ee
 It follows by  Lemma \ref{splitlemma} in the Appendices,    that for $i,j\in\{1,2\},$
 \begin{equation}\label{68}\sup_{\xi\in\mathcal{O}}||(\Lambda^{(i)})^{-1}B^{(ij)}||_{h_{p}^{0i}\rightarrow h_{p}^{0j}}\leq C\,\epsilon_0,\end{equation}
 \begin{equation}\label{68+1}\sup_{\xi\in\mathcal{O}}||(\Lambda^{(i)})^{-1}\partial_{\xi}\, B^{(ij)}||_{h_{p}^{0i}
 \rightarrow h_{p}^{0j}}\leq C\,\epsilon_0.\end{equation}
 By \eqref{64} and \eqref{68}, we have
 \be\label{69}\begin{array}{lll}\big|\big| \Lambda^{(2)}+B^{(22)}\big|\big|_{h_p^{02}\to h_p^{02}}&=&
 \big|\big| \Lambda^{(2)}\left(1+(\Lambda^{(2)})^{-1}B^{(22)}\right)\big|\big|_{h_p^{02}\to h_p^{02}}
 \\& \le &  \big|\big| \Lambda^{(2)}\big|\big|_{h_p^{02}\to h_p^{02}} \left(1+ \big|\big|(\Lambda^{(2)})^{-1}B^{(22)} \big|\big|_{h_p^{02}\to h_p^{02}}\right)\\ &\le & C \, K^{-\kappa\, c_{22}}(1+\e_0).
 \end{array}\ee
 By $\kappa\, c_{22}\gg c_{21}$ and $|(k,\omega)|\ge K^{-c_{21}}$, it follows that there does exist the inverse
 of $\m B_2$:
 \be\label{201705-y3}\m B_2^{-1}=\left(-(k,\omega)+
 \Lambda^{(2)}+B^{(22)} \right)^{-1}=\frac{-1}{( k,\omega)}
 \left( 1-\frac{1}{(k,\omega)}\left(\Lambda^{(2)}+B^{(22)} \right)\right)^{-1}\ee
 and
 \be\label{170520-1} \big|\big| \m B_2^{-1}\big|\big|_{h_p^{02}\to h_p^{02}}\le C\frac{1}{|(k,\omega)|}\le C\, K^{c_{21}}. \ee
 Moreover, by noting $ || (\Lambda^{(2)})^{-1}\,B^{(22)}||_{h_p^{02}\to h_p^{02}}\le C\,\e_0$ and $ \big|\big| \Lambda^{(2)} \big|\big|_{h_p^{02}\to h_p^{02}}\le C\,K^{-\kappa\, c_{22}}\ll K^{-c_{21}}$,
 \be\label{201705-y3+1}\begin{array}{lll}||\big((\Lambda^{(2)})^{-1}\mathcal{B}^{-1}_2\Lambda^{(2)}\big)
 ||_{h_p^{02}\to h_p^{02}}&=&\big|\big|\frac{-1}{(k,\omega)}
 \left( 1-\frac{1}{(k,\omega)}\left(1+(\Lambda^{(2)})^{-1}\,B^{(22)} \right)\, \Lambda^{(2)} \right)^{-1}\big|\big|_{h_p^{02}\to h_p^{02}}\\ &\le  &C \, K^{c_{21}}.\end{array}\ee
 By \eqref{170512-1} in Assumption {\bf B},
 \be \label{170520-2} \big|\big| \p_{\xi}\, \Lambda^{(2)} \big|\big|_{h_p^{02}\to h_p^{02}} \le C \,
 \big|\big| \Lambda^{(2)} \big|\big|_{h_p^{02}\to h_p^{02}}\le K^{-\kappa\, c_{22}}\ll K^{-c_{21}}.\ee
 Note
 \[\left| \p_\xi\, (k,\omega)\right|\le K.\]
 By \eqref{68+1}, \eqref{201705-y3}, \eqref{170520-1} and \eqref{170520-2}, we have
 \be\label{201705-y4}   \big|\big| \p_{\xi}\m B_2^{-1} \big|\big|_{h_p^{02}\to h_p^{02}}=
 \big|\big| \m B_2^{-1}\left(\p_{\xi}\m B_2\right) \m B_2^{-1} \big|\big|_{h_p^{02}\to h_p^{02}}\le C \, K^{2\, c_{21}+1}.\ee
 In addition,
 \be\label{201705-y5}  \big|\big| B^{(12)} \big|\big|_{h_p^{02}\to h_p^{01}}\le C\, \e_0,\;  \big|\big| B^{(21)} \big|\big|_{h_p^{01}\to h_p^{02}}\le
  \big|\big|\Lambda^{(2)} \big|\big|_{h_p^{02}\to h_p^{02}}
   \big|\big|(\Lambda^{(2)})^{-1} B^{(21)} \big|\big|_{h_p^{01}\to h_p^{02}}\le C\,\e_0\, K^{-\kappa\, c_{22}}\ee
 and
 \be\label{201705-y6}   \big|\big|\p_\xi\, B^{(12)} \big|\big|_{h_p^{02}\to h_p^{01}}\le C\, \e_0,\;  \big|\big|\p_\xi\, B^{(21)} \big|\big|_{h_p^{01}\to h_p^{02}}\le   \big|\big|\Lambda^{(2)} \big|\big|_{h_p^{02}\to h_p^{02}}    \big|\big|(\Lambda^{(2)})^{-1} \p_\xi\,B^{(21)} \big|\big|_{h_p^{01}\to h_p^{02}}\le C\,\e_0\, K^{-\kappa\, c_{22}}.\ee
 Applying \eqref{68}, \eqref{68+1}, \eqref{201705-y3}, \eqref{201705-y5} and \eqref{201705-y6} to \eqref{79}, we have
 \be\label{201705-y7} \sup_{\xi\in\m O}\big|\big|\p_\xi^t\, \m B_1 \big|\big|_{h_p^{01}\to h_p^{01}} \le C\, \e_0,\;\; t=0,1.\ee
Since $B$ is self-adjoint  in $\ell_2$, it is easy to see that $\mathcal B_1$ is Hermitian. Thus, by \eqref{17-15-3} in Appendices,
\begin{equation}\label{83}||\mathcal{B}_1||_{\ell_2^{01}\rightarrow\ell_{2}^{01}}\le ||\mathcal B_1||_{h_p^{01}\to h_p^{01}}
\le C\, \epsilon_0,\end{equation}
and \begin{equation}\label{84}||\partial_{\xi}\mathcal{B}_1||_{\ell_{2}^{01}\rightarrow\ell_{2}^{01}}\le ||\partial_{\xi}\mathcal B_1||_{h_p^{01}\to h_p^{01}}\le C\,
\epsilon_0,\quad
\xi\in\mathcal{O}\setminus\mathcal{O}_1.\end{equation}
Choose $\hat\xi$ to be a direction such that the directional derivative $\p_{\hat \xi}\, ( k,\omega)=|k|$. (Recall that we have assumed that $\omega(\xi)\equiv \xi$).  Then by \eqref{usefortwist},
\begin{equation} -\partial_{\hat \xi}\left(-(k,\omega)+\Lambda^{(1)}+\mathcal B_1\right)\ge c^*-C\,\epsilon_0\ge c^*/2, \end{equation}
where $X\ge Y$ means that $X-Y$ is positive for Hermitian matrices $X$ and $Y$.
{By the variation principle of eigenvalue of matrix (See \cite{[Ka]}), any eigenvalue, say $\mu=\mu(\xi)$, of $(-\langle k,\omega\rangle+\Lambda^{(1)}+\mathcal B_1)$ satisfies
\[|\partial_{\hat\xi}\mu(\xi)|\ge c^*-C \epsilon_0\ge \, c^*/2,\] where and in the following $C$ is a general constant which might depends on $N,s_0,\tau$ and which may be different in different places.}
Therefore, there is a subset $\mathcal{O}_2\subset\mathcal{O}\setminus\mathcal{O}_1$ with $\mbox{meas}\mathcal{O}_2=O(K^{-C_0/2})$ such that for $\forall\xi\in\mathcal{O}\setminus(\mathcal{O}_1\cup\mathcal{O}_2)$,
 \begin{equation}\label{88}||((-(k,\omega)+\Lambda^{(1)}+\mathcal B_1))^{-1}||_{\ell_2^{01}\rightarrow\ell_2^{01}}\leq
 K^{C_0},\end{equation} where $C_0$ is chosen such that $C_0- c_{22}>C_0/2$.
By (\ref{63}) and (\ref{64}), we can write $\mathcal B_1=(\mathcal B_1(i,j):\; |i|,|j|\le K^{c_{22}}),$ where $\mathcal B_1(i,j)$ is the matrix elements of $\mathcal B_1$. Therefore,
 \begin{equation}\label{y88}||((-( k,\omega)+\Lambda^{(1)}+\mathcal B_1))^{-1}||_{h_p^{01}\to h_p^{01}}\leq K^{p\, c_{22}}
 K^{C_0}:=K^{C_1}.\end{equation}

 \noindent Note that $||(\Lambda^{(1)})^{-1}||_{h_p^{01}\to h_p^{01}}\leq K^{\kappa\, c_{22}}$. It follows from \eqref{y88}, \eqref{170520-1} and \eqref{77} that
 \begin{equation}\label{89}||(\Lambda^{(1)})^{-1}F^{(1)}||_{h_{p}^{01}}\leq K^C(||(\Lambda^{(1)})^{-1}R^{(1)}||_{h_{p}^{01}}+||(\Lambda^{(2)})^{-1}R^{(2)}||_{h_{p}^{02}})\leq
 K^C||\Lambda^{-1}\hat{R}(k)||_{h_p},\end{equation}
where $C\gg c_{22}$ and   $\hat{R}(k)=\hat{R}(k,\xi)$ depends $\xi$.
 Moreover, by \eqref{78}, \eqref{201705-y3+1}, \eqref{68} and \eqref{89},
 \begin{eqnarray}\nonumber ||(\Lambda^{(2)})^{-1}F^{(2)}||_{h_{p}^{02}}&=
 &||\big((\Lambda^{(2)})^{-1}\mathcal{B}^{-1}_2\Lambda^{(2)}\big)(\Lambda^{(2)})^{-1}R^{(2)}||_{h_{p}^{02}}\\
 \nonumber&&+||\big((\Lambda^{(2)})^{-1}\mathcal{B}^{-1}_2\Lambda^{(2)}\big)\big((\Lambda^{(2)})^{-1}B^{(21)}\big)
 (\Lambda^{(1)})\big((\Lambda^{(1)})^{-1}F^{(1)}\big)||_{h_{p}^{02}}\\
 \label{90}&\leq&C_1\, K^{c_{21}}||(\Lambda^{(2)})^{-1}R^{(2)}||_{h_p^{2}}+C_2 \, K^{c_{21}}\cdot\epsilon_0\cdot ||(\Lambda^{(1)})^{-1}F^{(1)}||_{h_p^{01}}\\
 \nonumber &\leq& K^C\, ||\Lambda^{-1}\hat{R}(k)||_{h_p}.
 \end{eqnarray}
 By \eqref{89} and \eqref{90},
 \begin{equation}\label{91}  ||\hat{F}(k)||_{h_{q}}= ||\Lambda^{-1}\hat{F}(k)||_{h_{p}}\leq K^C ||\Lambda^{-1}\hat{R}(k)||_{h_p},\quad \xi\in\mathcal{O}_m\setminus(\mathcal{O}_1\cup\mathcal{O}_2),\quad 0<|k|\leq K.\end{equation}
Note $||\widehat{{F}}(k)||_{h_{q}}=||\lfloor \widehat{{F}}(k)\rceil||_{h_{q}}$ and $||\widehat{{R}}(k)||_{q}=||\lfloor \widehat{{R}}(k)\rceil||_{q}.$
By \eqref{91},
\begin{eqnarray*}
||\lfloor F \rceil||^{2}_{q,s',\mathcal{O}_{1}}&=&\sum_{0\neq k\in\mathbb{Z}^{N}}||\widehat{{F}}(k)||^{2}_{q}\cdot e^{2|k|s'}\\
&\leq &K^{2C}\sum_{0\neq k\in\mathbb{Z}^{N}}||\lfloor \widehat{{R}}(k)\rceil||^{2}_{q}\cdot e^{2|k|s}\\
&=&(K^{C}||\lfloor R\rceil||_{q,s,\mathcal{O}})^{2}.
\end{eqnarray*}
It follows that
$$||\lfloor F\rceil||_{q,s',\mathcal{O}_{1}}\leq K^{C}||\lfloor R\rceil||_{q,s,\mathcal{O}}.$$
This proves \eqref{2.9}.

Applying $\partial_{\xi}$ to both sides of \eqref{4.9}, we have
$$\Gamma(({\bf i}\omega\cdot \partial_{x}+\Lambda+B)\partial_{\xi}F)=\Gamma(\partial_{\xi}R)-\Gamma((\partial_{\xi}({\bf i}\omega\cdot \partial_{x}+\Lambda +B))F).$$
Repeating the previous procedure, we can prove \eqref{2.10}.

\section{Solution of linear equation for the second Melnikov conditions}

\begin{lem} \label{lem5.1} Assume $\m O$, $\omega=\omega(\xi)$, $\lambda_j$'s satisfy Assumptions {\bf A, B}, accordingly, by replacing $\m O_0$, $\omega^0$ by $\m O$, $\omega$. Assume both $B$ and $\breve B$ are self-adjoint in the square summable space $\ell_2(\mathbb Z^d)$ and   satisfy   Assumption {\bf E}.
 And assume the operator-value function $\Lambda^{-1}R$ obeys that
\be \Lambda^{-1}\lf R(x,\xi)\rc:\;\mathbb T^N_s\times\m O\to \mathcal{L}(h_{p},h_p)\ee
is analytic in $x\in\mathbb T^N_s$, $C^1$ in $\xi\in\m O$.
  Then there is a subset $\m O_2\subset\m O$ with
\be \text{Meas}\;\m O_2=(\text{Meas}\;\m O)(1-O(K^{-N}))\ee such that
for any $\xi\in\m O_2$, the homological equation
\be \label{II}\Gamma\left( ( -{\bf i}\omega\cdot\p_x\pm(\Lambda+ B))F \right)\pm \Gamma\left( F(\Lambda+ \breve B)\right)=(\Gamma R)(x,\xi)\ee
has unique solution $$F=F(x,\xi)=\widehat{F}(0,\xi)+\sum_{0\neq k\in\mathbb Z^N,|k|\le K}\widehat{F}(k,\xi)\,e^{\mi (k,x)}:=\widehat{F}(0,\xi)+\tilde F(x,\xi)$$ fulfilling \be\label{5.4} \tilde F:\;\mathbb T^N_s\times\m O_2\to \m L( h_p,h_q)\ee and
\be \label{5.8}\sup_{x\in \mathbb T^N_{\tilde s},\xi\in\m O_2}||\lf \tilde F\rc||_{h_p\to h_q}\le \frac{1}{(s-\tilde s)^N}\, K^C\, \sup_{x\in \mathbb T^N_s,\xi\in\m O}||\lf R\rc||_{h_p\to h_q},\; 0<\tilde s< s,
\ee
\be\label{5.10}\sup_{x\in \mathbb T^N_{\tilde s},\xi\in\m O_2}||\lf \p_\xi\tilde F\rc ||_{h_p\to h_q}\le \frac{1}{(s-\tilde s)^N}\, K^C\, \sup_{x\in \mathbb T^N_s,\xi\in\m O}(||\lf R\rc||_{h_p\to h_q}+||\lf \p_\xi R\rc||_{h_p\to h_q}),\; 0<\tilde s< s,\ee
\be\label{2017-7-6-1} \sup_{\xi\in\m O_2}||\lf \widehat{ F}(0,\xi)\rc||_{h_{\tilde p}\to h_{\tilde p}}\le K^C\, \sup_{\xi\in\m O}||\lf \widehat{R}(0,\xi)\rc||_{h_p\to h_q}, \; \tilde p\in\{p,q\},\ee
\be\label{2017-7-6-2} \sup_{\xi\in\m O_2}||\lf \p_\xi\widehat{ F}(0,\xi)\rc||_{h_{\tilde p}\to h_{\tilde p}}\le K^C\, \sup_{\xi\in\m O}\left(||\lf \widehat{R}(0,\xi)\rc||_{h_p\to h_q}+||\lf \p_\xi \widehat{R}(0,\xi)\rc||_{h_p\to h_q}\right), \; \tilde p\in\{p,q\},\ee
where we require that
\be \int_{\mathbb T^N} R(x,\xi)\,d \, x=0,\quad \forall\; \xi\in\mathcal O\ee when the sign $\pm$ in \eqref{II} appears in the form
\[\pm (\Lambda+B)\, F\pm F(\Lambda+\breve B)=\pm \left( (\Lambda+B)\, F- F (\Lambda+\breve B)\right).\]

\end{lem}
\vspace{5pt}
\noindent{\bf Proof:} By passing to Fourier coefficients, we get

\be\label{17-3-12-1} (( k,\omega) \pm(\Lambda+B))\widehat{{F}}(k)\pm \widehat{{F}}(k)(\Lambda+\breve B)=\widehat{{R}}(k),\quad k\in\mathbb Z^N,\;|k|\le K.\ee

\ \

\noindent{\bf Case 1. $k\neq 0$.}

\ \

In order to partition the operator $B$, we need the following lemma.

\begin{lem}\label{lemma3}
There are a large constant $c$
%
and a subset $\mathcal{O}_4\subset\mathcal{O}$ with $\mbox{meas}(\mathcal{O}_4)\leq K^{-N}$ such that
\begin{equation}\label{103*}|( k,\omega)\pm\lambda_j|\geq\frac12 K^{-c},\quad \mbox{for}\;\; k\in\mathbb Z^N,\; 0<|k|\leq K
,\;j\in\mathbb{Z}^d,\;\xi\in\mathcal{O}\setminus\mathcal{O}_4,\end{equation}

\end{lem}
\begin{proof} By Assumption {\bf A}, assume $\omega(\xi)\equiv \xi$ without loss of generality.

\noindent {\bf Part I}:
Let $y=(3d/\kappa)+3>3$. Choose a constant $c>0$ such that \begin{equation}\label{104}c>\frac{c}{y}>100 \, N(1+\kappa+d)\end{equation}
and \begin{equation}\label{105}c(1-\frac{3d}{y\,\kappa})>100\, N(1+\kappa+d).\end{equation}
By \eqref{104}, there is a subset $\mathcal{O}_{41}\subset\mathcal{O}$ with
 \begin{equation}\label{106}\mbox{meas}(\mathcal{O}_{41})\leq K^{-N}\end{equation} such that for any $\xi\in\mathcal{O}\setminus\mathcal{O}_{41}$,
 \begin{equation}\label{107}|( k,\omega)|\geq K^{-\frac{c}{y}}>K^{-c}.\end{equation}
 Note $c_{11}|j|^{-\kappa}\leq\lambda_j\leq c_{12}|j|^{-\kappa}.$ For simplicity, we assume $c_{12}=1$ here. So we have
 \begin{eqnarray}\nonumber|(k,\omega)\pm\lambda_j|&\geq&K^{-\frac{c}{y}}-|\lambda_j|\\
 \label{108}&\geq&K^{-\frac{c}{y}}-|j|^{-\kappa}\\
 \nonumber &\geq&(1/2)K^{-c},\end{eqnarray}
if
\be \label{17-3-16-1}|j|\geq (2K^{c/y})^{3/\kappa}=(2^{\frac{1}{\kappa}}K^{\frac{c}{\kappa y}})^{3}:=K_2.\ee
In addition, by the definition of $K_2$,
\[|\lambda_j|<(\frac12 K^{-c/y})^{3},\quad |j|>K_2.\]

\begin{rem} If without assuming $c_{12}=1$ , one can take
$K_2=(2 c_{12}^{-1}\, K^{c/y})^{3/\kappa}$.

\end{rem}
\noindent {\bf Part II}: $\;\;$ Consider $j$ with $|j|<K_2$.  Let
\begin{equation}\label{109}\mathcal{O}_{42}:=\bigcup_{0<|k|<K\atop{|j|<K_2}}\{\xi\in\mathcal{O}:
|(k,\omega)\pm\lambda_j|<K^{-c}\}.\end{equation}
Then by \eqref{usefortwist} \begin{eqnarray}\nonumber\mbox{meas}\mathcal{O}_{42}&\leq& C\,K^{-c}K^N\cdot(2^{\frac{1}{\kappa}}K^{\frac{c}{\kappa y}})^{3d}\\
\nonumber &=&C\,2^{\frac{3d}{\kappa}}K^{-(1-\frac{3d}{\kappa y})c+N}\\
 \label{110}(\operatorname{by}\, \eqref{105})&\le &  C\, K^{-N}.\end{eqnarray}
By \eqref{106} and \eqref{110}, letting $\mathcal{O}_4=\mathcal{O}_{41}\cup\mathcal{O}_{42}$, then
\begin{equation}\label{111}\mbox{meas}\mathcal{O}_{4}\leq CK^{-N},\end{equation}
and using \eqref{108} and \eqref{109}, one has
\begin{equation}\label{112}|( k,\omega)\pm\lambda_j|\geq \frac12 K^{-c},\quad \mbox{for}\ \xi\in\mathcal{O}\setminus\mathcal{O}_4, \;j\in
\mathbb{Z}^d, \;|k|\leq K.\end{equation}
This completes the proof.$\qed$
\end{proof}

\begin{lem}\label{lemma3.2} Assume that there are real numbers $\mu_j=\mu_j(\xi)$'s with $j\in\mathbb Z^d$ and $|j|\le K_2$ which satisfy
\[|\partial_{\xi}\mu_j(\xi)|\ll 1.\]
Then there is a  subset $\mathcal{O}_4^{\prime}\subset\mathcal{O}$ with $\mbox{meas}(\mathcal{O}_4^{\prime})\leq K^{-N}$ such that
\begin{equation}\label{103}|(k,\omega)\pm\mu_j|\geq K^{-c},\quad \mbox{for}\;\; k\in\mathbb Z^N,\; 0<|k|\leq K
,\;j\in\mathbb{Z}^d,\;|j|\le K_2,\;\xi\in\mathcal{O}\setminus\mathcal{O}^{\prime}_4.\end{equation}

\end{lem}
\begin{proof} The proof is the same as Part II in the proof of Lemma \ref{lemma3}.

\end{proof}

We are now prepared to partition the matrices $\Lambda$ and $B$. To this end, let
\be\label{K_3} K_3=\max\left\{K_2^{100(N+1)\,\kappa\, y},\, 2^{1/\kappa}\, K^{100(1+\frac{p+\kappa}{\kappa\, y})\frac{c}{\kappa}}\right\},\ee where $K_2$ is defined in \eqref{17-3-16-1}.
By the definition of $K_3$, we have
\be\label{importantsplit} |\lambda_j|<\min\{K^{-100c},K^{-10}_2\},\, \;\operatorname{if}\; |j|>K_3. \ee

  %
  Recall $\Lambda=\mbox{diag}(\lambda_j:j\in\mathbb{Z}^d)$.
 Write \[\Lambda=\Lambda^{(1)}\oplus\Lambda^{(2)}=\left(\begin{array}{ll}
\Lambda^{(1)}&0\\
0&\Lambda^{(2)}\end{array}\right)\]
 with $\Lambda^{(1)}=\mbox{diag}(\lambda_j:|j|\leq K_3), \Lambda^{(2)}=\mbox{diag}(\lambda_j:|j|> K_3).$

According to the partition of $\Lambda$, we partition $B$  as follows
$$B=\left(\begin{array}{ll}
B^{(11)}&B^{(12)}\\
B^{(21)}&B^{(22)}\end{array}\right), $$
where $B^{(11)}=(B_{ij}:\; |i|\le K_3, |j|\le K_3)$, $B^{(12)}=(B_{ij}:\; |i|\le K_3, |j|> K_3)$, $B^{(21)}=(B_{ij}:\; |i|>K_3, |j|\le K_3)$ and $B^{(22)}=(B_{ij}:\; |i|>K_3, |j|> K_3)$.
Similarly,
$$\breve B=\left(\begin{array}{ll}
\breve B^{(11)}&\breve B^{(12)}\\
\breve B^{(21)}&\breve B^{(22)}\end{array}\right),\quad \widehat{R}(k)=\left(\begin{array}{ll}
R^{(11)}&R^{(12)}\\
R^{(21)}&R^{(22)}\end{array}\right), \quad \widehat{{F}}(k)=\left(\begin{array}{ll}
F^{(11)}&F^{(12)}\\
F^{(21)}&F^{(22)}\end{array}\right).$$

Again according to the partition of $\Lambda$, split $h_{p}=h_{p}^1\oplus h_{p}^2,$ and $\ell_2(\mathbb Z^d)=\ell_2^1\oplus\ell_2^2$,  corresponding. It follows, from  Assumption {\bf E} and   Lemma \ref{splitlemma} in the Appendices,   that for $b\in \{B,\breve B\}$,
\begin{equation}\label{116}||(\Lambda^{(i)})^{-1}\lf b^{(ij)}\rc||_{h^i_{p}\rightarrow h^j_{p}}\le\epsilon_0,\quad ||\partial_\xi\left((\Lambda^{(i)})^{-1}\lf b^{(ij)}\rc\right)||_{h^i_{p}\rightarrow h^j_{p}}\le\epsilon_0
,\quad i,j\in\{1,2\}. \end{equation}
And we have that for $i,j\in\{1,2\}$,
\begin{equation}\label{117}||(\Lambda^{(i)})^{-1}\lf R^{(ij)}\rc||_{h^i_{p}\rightarrow h^j_{p}}\le ||\Lambda^{-1}\lf R\rc||_{h_{p}\to h_{p}},\;
||(\Lambda^{(i)})^{-1}\, \lf\partial_\xi\,R^{(ij)}\rc||_{h^i_{p}\rightarrow h^j_{p}}\le ||
\Lambda^{-1}\, \lf\partial_\xi\,R\rc||_{h_{p}\to h_{p}}.\end{equation}

%
%
With the above notations, the homological equation \eqref{17-3-12-1} becomes
\begin{eqnarray}\nonumber&&\left(\begin{array}{ll}
(k,\omega)\pm(\Lambda^{(1)}+B^{(11)})&B^{(12)}\\
B^{(21)}&(k,\omega)\pm(\Lambda^{(2)}+B^{(22)})\end{array}\right)\left(\begin{array}{ll}
F^{(11)}&F^{(12)}\\
F^{(21)}&F^{(22)}\end{array}\right)\\
&\pm &\label{120}\left(\begin{array}{ll}
F^{(11)}&F^{(12)}\\
F^{(21)}&F^{(22)}\end{array}\right)\left(\begin{array}{ll}
\Lambda^{(1)}+\breve B^{(11)}&\breve B^{(12)}\\
\breve B^{(21)}&\Lambda^{(2)}+\breve B^{(22)}\end{array}\right)=\left(\begin{array}{ll}
R^{(11)}&R^{(12)}\\
R^{(21)}&R^{(22)}\end{array}\right).\end{eqnarray}
For brevity, let

\br\label{17-3-17-1}
{M}^{(11)}=( k,\omega)\pm(\Lambda^{(1)}+B^{(11)}),\; {M}^{(22)}=( k,\omega)\pm(\Lambda^{(2)}+B^{(22)}),\;
M^{(12)}=\pm B^{(12)},\; M^{(21)}=\pm B^{(21)},
\er
\be\label{17-3-17-2}
{N}^{(11)}=\pm(\Lambda^{(1)}+\breve B^{(11)}),\;{N}^{(22)}=\pm(\Lambda^{(2)}+\breve B^{(22)}),\;
N^{(12)}=\pm \breve B^{(12)},\; N^{(21)}=\pm \breve B^{(21)}.
\ee

Then \eqref{120} becomes
 %
  \begin{eqnarray}
&&{M}^{(11)}F^{(11)}+F^{(11)}{N}^{(11)}+M^{(12)}F^{(21)}+F^{(12)}N^{(21)}=R^{(11)},\label{121}\\
&&{M}^{(22)}F^{(21)}+F^{(21)}{N}^{(11)}+M^{(21)}F^{(11)}+F^{(22)}N^{(21)}=R^{(21)},\label{123}\\
 &&{M}^{(11)}F^{(12)}+F^{(12)}{N}^{(22)}+M^{(12)}F^{(22)}+F^{(11)}N^{(12)}=R^{(12)},\label{122}\\
&&{M}^{(22)}F^{(22)}+F^{(22)}{N}^{(22)}+M^{(21)}F^{(12)}+F^{(21)}N^{(12)}=R^{(22)} \label{124}.\end{eqnarray}
 Denote by $X=G_{22}(Y)$ the unique solution of the operator equation $${M}^{(22)}X+X{N}^{(22)}=Y.$$
 That is to say, we formally denote by $G_{22}$ the Green function:
 \be \label{17-3-30-1}G_{22}(\cdot)=\left({M}^{(22)}(\cdot)+(\cdot){N}^{(22)}\right)^{-1},\ee
 where
 \[\left({M}^{(22)}(\cdot)+(\cdot){N}^{(22)}\right)X={M}^{(22)}(X)+(X){N}^{(22)}.\]
 Then by \eqref{124},
 \begin{equation}\label{125}F^{(22)}=G_{22}(R^{(22)})-G_{22}(M^{(21)}F^{(12)})-G_{22}(F^{(21)}N^{(12)}).\end{equation}
Similarly, let the Green function $G_{21}$ be formally defined by
 \begin{equation}\label{126}G_{21}(\cdot)=\left({M}^{(22)}(\cdot)+(\cdot){N}^{(11)}-
 G_{22}((\cdot)N^{(12)})N^{(21)}\right)^{-1}.\end{equation}
 Then by inserting \eqref{125} into \eqref{123}, \begin{equation}\label{127}F^{(21)}=-G_{21}(M^{(21)}F^{(11)})+G_{21}\left(G_{22}(M^{(21)}F^{(12)})N^{(21)}\right)+\tilde{R}^{(21)},\end{equation}
 where \begin{equation}\label{128}\tilde{R}^{(21)}=G_{21}\left(R^{(21)}-G_{22}(R^{(22)})N^{(21)}\right).\end{equation}
 Inserting \eqref{127} into \eqref{125}, we get
 \begin{eqnarray}\nonumber F^{(22)}&=&G_{22}(R^{(22)})-G_{22}(M^{(21)}F^{(12)})+G_{22}\left(G_{21}(M^{(21)}F^{(11)})N^{(12)}\right)\\
 \label{129}&&-G_{22}\left(G_{21}\left(G_{22}(M^{(21)}F^{(12)})N^{(21)}\right)N^{(12)}\right)
 -G_{22}(\tilde{R}^{(21)}N^{(12)}).\end{eqnarray}
 Inserting \eqref{129} into \eqref{122}, one has
 \begin{eqnarray}\nonumber&&{M}^{(11)}F^{(12)}+F^{(12)}{N}^{(22)}+M^{(12)}G_{22}\left(G_{21}(M^{(21)}F^{(11)})N^{(12)}\right)+F^{(11)}N^{(12)}\\
 \label{130}&&-M^{(12)}\left(G_{22}(G_{21}(G_{22}(M^{(21)}F^{(12)})N^{(21)})N^{(12)})+G_{22}(M^{(21)}F^{(12)})\right)\\
 \nonumber &&=-M^{(12)}\left(G_{22}(R^{(22)})-G_{22}(\tilde{R}^{(21)}N^{(12)})\right)+R^{(12)}.
 \end{eqnarray}
Let the Green function $G_{12}$ be formally defined by
\begin{equation}\label{131}G_{12}(\cdot)=\left({M}^{(11)}(\cdot)+(\cdot){N}^{(22)}-M^{(12)}\left(
G_{22}\left(G_{21}(G_{22}(M^{(21)}(\cdot)
)N^{(21)})N^{(12)}+(M^{(21)}(\cdot))\right)\right)\right)^{-1}.\end{equation}
It follows from \eqref{130} that \begin{equation}\label{132}F^{(12)}=\mathcal{L}(F^{(11)})+\tilde{R}^{(12)},\end{equation}
where
\begin{equation}\label{133}\mathcal{L}(F^{(11)}):=-G_{12}\left(M^{(12)} G_{22}\left(G_{21}
(M^{(21)}F^{(11)})N^{(12)}\right)+F^{(11)}N^{(12)}\right),\end{equation}
\begin{equation}\label{134}\tilde{R}^{(12)}:=-G_{12}\left(M^{(12)}\left(G_{22}(R^{(22)})-G_{22}(\tilde{R}^{(21)}N^{(12)})\right)-R^{(12)}\right).\end{equation}
Inserting \eqref{132} into \eqref{127}, we have
\begin{eqnarray}\label{135}
F^{(21)}&=&-G_{21}(M^{(21)}F^{(11)})+G_{21}\left(G_{22}(M^{(21)}\mathcal{L}(F^{(11)}))N^{(21)}\right)\nonumber\\
&&+G_{21}\left(G_{22}(M^{(21)}\tilde{R}^{(12)})N^{(21)}\right)+\tilde{R}^{(21)}.
\end{eqnarray}
Finally, inserting \eqref{132} and \eqref{135} into \eqref{121}, we have
\begin{equation}\label{136} {M}^{(11)}F^{(11)}+F^{(11)}{N}^{(11)}+\mathcal{L}_1(F^{(11)})=\mathcal{R}^{(11)},\end{equation}
where
\begin{eqnarray}\nonumber \mathcal{L}_1(F^{(11)}):=&&\!\!\!\!M^{(12)}G_{21}\left(-M^{(21)}F^{(11)}+G_{22}(M^{(21)}\mathcal{L}(F^{(11)}))N^{(21)}\right)\\
\label{137}&&\!\!\!\!-G_{12}\left(M^{(12)}\left(G_{22}(G_{21}(M^{(21)}F^{(11)})N^{(12)})+F^{(11)}N^{(12)}\right)\right)N^{(21)}\end{eqnarray}
and \begin{equation}\label{138}\mathcal{R}^{(11)}:=R^{(11)}-M_{12}G_{21}\left(G_{22}(M^{(21)}\tilde{R}^{(12)})N^{(21)}\right)-M^{(12)}\tilde{R}^{(21)}
-\tilde{R}^{(12)}N^{(21)}.\end{equation}

Our strategy is as follows:
 \begin{enumerate}
 \item to find $F^{(11)}$ by solving \eqref{136};\item to find $F^{(12)}$ by solving \eqref{132};\item to find $F^{(21)}$ by solving \eqref{127}; \item to find $F^{(22)}$ by solving \eqref{125}. \end{enumerate}

In order to solve \eqref{136}, we need a more explicit form of $\m L_1(F^{(11)})$. To this end, we introduce the following notations.

Let $\m H_1$ and $\m H_2$ be Hilbert spaces. For each $\phi_1\in\m H_1$, $\phi_2\in\m H_2$, let $\phi_1\otimes\phi_2$ denote the conjugate bilinear form which acts on $\m H_1\times\m H_2$ by
\be (\phi_1\otimes\phi_2)\<\psi_1,\psi_2\>=(\psi_1,\phi_1)(\psi_2,\phi_2),\quad \forall\; \psi_1\in\m H_1,\psi_2\in\m H_2,\ee
where $(\psi_i,\phi_i)$ is the inner product of $\phi_i$ and $\psi_i$ in $\m H_i$ ($i=1,2$). Let $\m E$ be the set of finite linear combinations of such conjugate linear forms. We define an inner product $(\cdot,\cdot)$ on $\m E$ by
\be (\phi\otimes\psi,\eta\otimes\mu)=(\phi,\eta)(\psi,\mu) \ee
and extending it by linearity to $\m E$. By Proposition 1 in pp.49-50 of \cite{RS}, $(\cdot,\cdot)$ is really an inner product. We define $\m H_1\otimes\m H_2$ to be the completion of $\m E$ under the inner product $(\cdot,\cdot)$. Let $X$ and $Y$ be densely defined operators on Hilbert spaces $\m H_1$ and $\m H_2$ respectively. Let $D(X)\in\m H_1$ and $D(Y)\in\m H_2$ be the definition domains of $X$ and $Y$, respectively. Denote by $D(X)\otimes D(Y)$ the set of finite linear combinations of vectors of the form $\phi\otimes\psi$ where $\phi\in D(X)$ and $\psi\in D(Y)$. Then $D(X)\otimes D(Y)$ is dense in $\m H_1\otimes\m H_2$. We define $X\otimes Y$ on $D(X)\otimes D(Y)$ by
\be \label{5.91}(X\otimes Y)(\phi\otimes\psi)=X\phi\otimes Y\psi\ee
and extend it by linearity. Then the operator $X\otimes Y$ is well defined. If $X$ and $Y$ are bounded operators on Hilbert space $\m H_1$ and $\m H_2$, respectively, then
\be\label{x5.52} ||X\otimes Y||=||X||||Y||,\ee where the $||\cdot||$ of $||X\otimes Y||$ is the operator norm from $\m H_1\otimes\m H_2$ to itself and  the $||\cdot||$ of $||X||$ (and $||Y||$) is the operator norm from $\m H_1$ (and $\m H_2$) to itself .
See Chapter VIII.10 in \cite{RS} for the details. By (\ref{5.91}) one see that if write $X=(X_{ij})$ then
\be\label{17-3.45} X\otimes Y=(X_{ij}Y),\ee
where $X_{ij}$'s are the matrix elements of $X$ under the basis of $\m H_1$.
 In partitioned multiplication of matrix, we immediately get
\be \label{x5.54}(X\otimes Y)(U\otimes V)=(X U)\otimes (Y V)\ee
if the operation above is reasonable. Moreover,
\be\label{x5.55} (X\otimes Y)^{-1}=X^{-1}\otimes Y^{-1}\ee
if the inverses exist. In addition, we see from \eqref{17-3.45} that
\be\label{17-3.48} (X\otimes Y)^*=X^*\otimes Y^*,\quad (\cdot)^*=\mbox{adjoint operator of}\;(\cdot). \ee
%
 For a matrix $X=(X_{ij}\in\mathbb C:\; i,j\in\mathbb Z^d)$, write $X=(...,X_1,X_2,...,X_j,...)$ with $X_j$'s are column vectors of $X$. Then formally,
\[\operatorname{Vec}\, X=\begin{pmatrix} \vdots\\ X_1\\ \vdots\\ X_j\\ \vdots\end{pmatrix}.\]
In partitioned multiplication of matrix, then the following holds true formally
\be \label{x5.60}\text{Vec}\; (X U Y)= (Y^T\otimes X)\text{Vec}\;U \ee for the three linear operators $X, Y, U$.

With those notations above, by applying ``$\operatorname{Vec}$'' to  \eqref{136} we get
\be \label{1703-2}\left(1\otimes M^{(11)}+N^{(11)T}\otimes 1+1\otimes\mathcal{L}_1\right)\mbox{Vec}\,F^{(11)}=\mbox{Vec}\,\mathcal {R}^{(11)}.\ee
By \eqref{138},
\be \label{1703-1}\mbox{Vec}\,\mathcal {R}^{(11)}=\mbox{Vec}\,{R}^{(11)}+\mbox{Vec}\,\left( -M^{(12)}G_{21}(G_{22}(M^{(21)}\tilde{R}^{(12)})N^{(21)})-M^{(12)}\tilde{R}^{(21)}
-\tilde{R}^{(12)}N^{(21)}\right). \ee
By applying ``Vec" to both left sides and right sides of \eqref{121}-\eqref{124}, we have formally that
\begin{eqnarray}\nonumber&&\left(\begin{array}{llll}
1\otimes M^{(11)}+N^{(11)T}\otimes 1&1\otimes M^{(12)}& N^{(21)T}\otimes1&0\\ 1\otimes M^{(21)}&1\otimes M^{(22)}+N^{(11)T}\otimes1&0& N^{(21)T}\otimes1  \\
N^{(12)T}\otimes1&0&1\otimes M^{(11)}+N^{(22)T}\otimes 1&1\otimes M^{(12)}\\ 0& N^{(12)T}\otimes 1&1\otimes M^{(21)}& 1\otimes M^{(22)}+N^{(22)T}\otimes 1
\end{array}\right)\cdot\\
&&\left(\begin{array}{l}
\mbox{Vec}\,F^{(11)}\\  \mbox{Vec}\, F^{(21)}\\
\mbox{Vec}\, F^{(12)}\\ \mbox{Vec}\, F^{(22)}\end{array}\right)=\left(\begin{array}{l}
\mbox{Vec}\,R^{(11)}\\  \mbox{Vec}\, R^{(21)}\\
\mbox{Vec}\, R^{(12)}\\ \mbox{Vec}\, R^{(22)}\end{array}\right) .\end{eqnarray}
It follows that
\be \label{1703-3}\left( 1\otimes M^{(11)}+N^{(11)T}\otimes 1+\mathcal A \right) \mbox{Vec}\, F^{(11)}=\mbox{Vec}\, R^{(11)}+\underline{R}, \ee
where
\be \label{17-3-14-1}\mathcal A=-\begin{pmatrix} 1\otimes M^{(12)}& N^{(21)T}\otimes 1& 0\end{pmatrix}\cdot Q^{-1}\cdot\begin{pmatrix} 1\otimes M^{(21)}\\ N^{(12)T}\otimes 1\\0 \end{pmatrix},\ee

\begin{eqnarray}\nonumber&&Q=\left(\begin{array}{lll}
 1\otimes M^{(22)}+N^{(11)T}\otimes1&0& N^{(21)T}\otimes1  \\
0&1\otimes M^{(11)}+N^{(22)T}\otimes 1&1\otimes M^{(12)}\\  N^{(12)T}\otimes 1&1\otimes M^{(21)}& 1\otimes M^{(22)}+N^{(22)T}\otimes 1
\end{array}\right)\end{eqnarray}
and
\be \underline{R}=-\begin{pmatrix} 1\otimes M^{(12)}& N^{(21)T}\otimes 1& 0\end{pmatrix}\cdot Q^{-1}\cdot\begin{pmatrix} \mbox{Vec}\, R^{(21)}\\\mbox{Vec}\, R^{(12)}\\ \mbox{Vec}\, R^{(22)} \end{pmatrix}. \ee
In priori, assume that there exist the inverse of $1\otimes M^{(11)}+N^{(11)T}\otimes 1+1\otimes \m L_1$ and $1\otimes M^{(11)}+N^{(11)T}\otimes 1+\m A$. Regard  $R^{(ij)}$ ($i,j\in\{1,2\})$ as ``variables"  in the left-hand sides of  both (\ref{1703-2}) and (\ref{1703-3}). Then comparing (\ref{1703-2}) and (\ref{1703-3}), we get
\be\label{17-3.65} 1\otimes \mathcal L_1=\mathcal A=\begin{pmatrix} 1\otimes M^{(12)}& N^{(21)T}\otimes 1& 0\end{pmatrix}\cdot Q^{-1}\cdot\begin{pmatrix} 1\otimes M^{(21)}\\ N^{(21)}\otimes 1\\0 \end{pmatrix}.\ee
Using $1\otimes \mathcal L_1=\mathcal A$ and  comparing (\ref{1703-2}) and (\ref{1703-3}) again, we get
\be\label{17-3.66} \underline{R}=\mbox{Vec}\,\left( -M_{12}G_{21}(G_{22}(M^{(21)}\tilde{R}^{(12)})N^{(21)})-M^{(12)}\tilde{R}^{(21)}
-\tilde{R}^{(12)}N^{(21)}\right).  \ee

We are now in position to prove that $\mathcal A$ is a self-adjoint operator (i.e., Hermitian matrix) from $\ell_2^1$ to $\ell_2^1$. Actually, it is obvious. Recall that $B$ is self-adjoint in $\ell_2$. So we have
\[M^{(ij)*}=M^{(ji)},\quad N^{(ij)*}=N^{(ji)},\quad i,j\in\{1,2\},\]
where the operation $*=$ the complex conjugate plus transpose of matrix. And note that $(X\otimes Y)^*=X^*\otimes Y^*$. It follows that $Q^* =Q$. Moreover, $\mathcal A^* =\mathcal A$ if there exists the inverse of $Q$. Now let us prove that there does exist the inverse of $Q$. Let
\be \mathcal H=\begin{pmatrix} h_p^1\otimes h_p^2\\ h_p^2\otimes
h_p^1\\ h_p^2\otimes h_p^2 \end{pmatrix}=(h_p^1\otimes h_p^2)\oplus (h_p^2\otimes
h_p^1)\oplus(h_p^2\otimes h_p^2)  .\ee For $(x,y,z)^T\in\mathcal H$, define
\[||(x,y,z)^T||_{\mathcal H}=\sqrt{||x||^2_{h_p^1\otimes h_p^2}+||y||^2_{h_p^2\otimes h_p^1}+
||z||^2_{h_p^2\otimes h_p^2}}.\] Then $\mathcal H$ is a Hilbert space with an inner product corresponding to $||\cdot||_{\m H}$.

{
By \eqref{116}, \eqref{17-3-17-1}, \eqref{17-3-17-2}   and using  Lemma \ref{splitlemma} in the Appendices,
\be\label{17-3-15-4} ||\lf M^{(ij)}\rc||_{h_p^i\to h_p^j}\le ||\Lambda^{(i)}||_{h_p^i\to h_p^i}||(\Lambda^{(i)})^{-1}\lfloor M^{(ij)\rceil}||_{h_p^i\to h_p^j} \le \e_0,\;i\neq j\in\{1,2\},\ee
 \be \label{170513-01}||\lf\p_\xi M^{(ij)}\rc||_{h_p^i\to h_p^j}\le||\Lambda^{(i)}||_{h_p^i\to h_p^i}
 ||(\Lambda^{(i)})^{-1}\lf \p_\xi\,M^{(ij)}\rc||_{h_p^i\to h_p^j}\le \e_0,\;\;i\neq j\in\{1,2\},\ee
\be\label{17-3-15-5}||\lf N^{(ij)}\rc||_{h_p^i\to h_p^j}\le \e_0,||\lf\p_\xi N^{(ij)}\rc||_{h_p^i\to h_p^j}\le \e_0,\;\;i\neq j\in\{1,2\}, \ee
\be \label{17-3-15-6}||\lf M^{(ii)}-((k, \omega)\pm\Lambda^{(i)})\rc||_{h_p^i\to h_p^i}\le \e_0, ||\lf \p_\xi (M^{(ii)}-((k, \omega)\pm\Lambda^{(i)}))\rc|
|_{h_p^i\to h_p^i}\le \e_0,\;\;i\in\{1,2\},\ee
\be\label{17-3-15-7}||\lf N^{(ii)}\mp\Lambda^{(i)}\rc||_{h_p^i\to h_p^i}\le \e_0,||\lf \p_\xi ( N^{(ii)}\mp\Lambda^{(i)})\rc||_{h_p^i\to h_p^i}\le \e_0,\;\;i\in\{1,2\}. \ee}
By  Assumption {\bf B},
\be\label{17-3-15-8}||\Lambda||_{h_p\to h_p}\le C.\ee
By (\ref{17-3-15-4})-(\ref{17-3-15-8}) and noting that $||X\otimes Y||=||X||\; ||Y||$, we get that each of $Q$ and $\p_\xi\, Q$ is a bounded linear operator from $\mathcal H$ to $\mathcal H$.
\begin{lem}\label{lemma3.3} (i). The operator
\[1\otimes M^{(22)}+N^{(11)T}\otimes 1:h_p^1\otimes h_p^2\to h_p^1\otimes h_p^2\]
has a unique bounded inverse with
\be \label{17-3.73}||\lf (1\otimes M^{(22)}+N^{(11)T}\otimes 1)^{-1} \rc||\le
K_{2}^{p+\frac{d}{2}}\cdot K^{c}\cdot K^{2c/y}:=K_4, \ee
where $||\cdot||$ is the operator norm from $h_p^1\otimes h_p^2$ to $h_p^1\otimes h_p^2$.

\ \

(ii).
The operator
\[1\otimes M^{(11)}+N^{(22)T}\otimes 1:\, h_p^2\otimes h_p^1\to h_p^2\otimes h_p^1\]
has a unique bounded inverse with
\be \label{17-3.73}||\lf (1\otimes M^{(11)}+N^{(22)T}\otimes 1)^{-1} \rc||\le K_4, \ee
where $||\cdot||$ is the operator norm from $h_p^2\otimes h_p^1$ to $ h_p^2\otimes h_p^1$.

\ \

(iii).
The operator
\[1\otimes M^{(22)}+N^{(22)T}\otimes 1:\, h_p^2\otimes h_p^2\to h_p^2\otimes h_p^2\]
has a unique bounded inverse with
\be \label{17-3.73}||\lf (1\otimes M^{(22)}+N^{(22)T}\otimes 1)^{-1}\rc ||\le K_4, \ee
where $||\cdot||$ is the operator norm from $h_p^2\otimes h_p^2$ to $ h_p^2\otimes h_p^2$.

\end{lem}
\begin{proof} We give the proof only for the case (ii). The remaining proofs are similar.

By \eqref{17-3-17-1} and \eqref{17-3-17-2},
\[1\otimes M^{(11)}+N^{(22)T}\otimes 1=1\otimes(( k,\omega)\pm(\Lambda^{(1)}+B^{(11)}))\pm (\Lambda^{(2)}+\breve B^{(22)})\otimes 1. \]
Since $B$ is self-adjoint operator from $\ell_2$ to $\ell_2$ and $\Lambda$ are real diagonal matrix, the matrix $\Lambda^{(1)}+B^{(11)}$ is Hermitian (or self-adjoint in $\ell_2^1$). Make a finer partition of $\Lambda^{(1)}+B^{(11)}$ as follows:
\be\label{bucong-3.74}\Lambda^{(1)}=\begin{pmatrix}\Lambda^{(1)}_1& 0\\ 0& \Lambda^{(1)}_2\end{pmatrix} \ee with
$\Lambda^{(1)}_1=\mbox{diag}\; (\lambda_j:\; |j|\le K_2)$ and $ \Lambda^{(1)}_2=\mbox{diag}\; (\lambda_j:\; K_2<|j|\le K_3)$.
See \eqref{17-3-16-1} for $K_{2}$ and see \eqref{K_3} for $K_3$.  In this principle as above,  make partition
\be\label{bucong-3.75} B^{(11)}=\begin{pmatrix}B^{(11)}_{11}& B^{(11)}_{12}\\ B^{(11)}_{21}& B^{(11)}_{22} \end{pmatrix}\ee and
$$h_p^1=h_p^{11}\oplus h_p^{12},\; \ell_2^1=\ell_2^{11}\oplus \ell_2^{12},\, h_q^1=h_q^{11}\oplus h_q^{12}.$$
With those partitions, one has a formal equality:
\be\label{17-3-29-1}\begin{pmatrix}1 &- B_{12}^{(11)}\m B^{-1}\\ 0&1 \end{pmatrix} (( k,\omega)\pm(\Lambda^{(1)}+B^{(11)}))\begin{pmatrix}1 &0\\ -\m B^{-1}B_{21}^{(11)} &1 \end{pmatrix}=\begin{pmatrix}( k,\omega)+\m C &0\\0& \m B \end{pmatrix},\ee
where \[\m B=(k,\omega)\pm(\Lambda_2^{(1)}+B_{22}^{(11)}),\]
\be\label{170513bu1}\m C=\pm\left(\Lambda^{(1)}_1+B^{(11)}_{11}\right)-B_{12}^{(11)}\left(( k,\omega) \pm(
\Lambda_2^{(1)}+B^{(11)}_{22})\right)^{-1}B_{21}^{(11)}.\ee
By \eqref{116}, 
\[\begin{array}{lll}||\lf B^{(11)}_{22}\rc||_{h_p^{12}\to h_p^{12}}&\le &
||\Lambda_2^{(1)}||_{h_p^{12}\to h_p^{12}}\cdot ||(\Lambda_2^{(1)})^{-1}\lf B^{(11)}_{22}\rc||_{h_p^{12}\to h_p^{12}}
\\ &\le & \sup\{\lambda_j:\; |j|\ge K_2\}||\Lambda^{-1}\lf B\rc||_{h_p\to h_p}\\
&\le &((1/2)\;K^{-c/y})^{3}\e_0. \end{array}\]
Moreover,
\be\label{17-3-3.74}||\lf \Lambda_2^{(1)}+B^{(11)}_{22}\rc||_{h_p^{12}\to h_p^{12}}\le \sup_{|j|>K_2}\{\lambda_j\}+||\lf B^{(11)}_{22}\rc||_{h_p^{12}\to h_p^{12}}\le (1/8)(1+O(\e_0))K^{-3c/y}.\ee
Recall that $|( k,\omega)|\ge |K|^{-c/y}$. Using Neumann series and Lemma \ref{absolute-summation}, we get
\be \label{201707-23} \begin{array}{lll}||\lf \m B^{-1}\rc ||_{h_p^{12}\to h_p^{12}}&=&
||\lf (( k,\omega) \pm(\Lambda_2^{(1)}+B^{(11)}_{22}))^{-1}\rc||_{h_p^{12}\to h_p^{12}}\\&=&||\frac{1}{|(k,\omega)|}\lf \left(1\pm\frac{1}{(k,\omega)}(\Lambda_2^{(1)}+B^{(11)}_{22})
\right)^{-1}\rc||_{h_p^{12}}
\\&=&||\frac{1}{|(k,\omega)|}\lf \sum_{j=0}^\infty\left(\pm\frac{1}{(k,\omega)}(\Lambda_2^{(1)}+B^{(11)}_{22})
\right)^{j}\rc||_{h_p^{12}}\\ &\le & \frac{1}{|(k,\omega)|} \sum_{j=0}^\infty\left(\frac{1}{|(k,\omega)|}(||\lf \Lambda_2^{(1)}+B^{(11)}_{22}\rc ||_{h_p^{12}})
\right)^{j}\\ &\le &  K^{c/y} \sum_{j=0}^{\infty} (\frac{1+O(\e_0)}{2})^j  \\&
\le &3 K^{c/y}.\end{array}\ee
Again by \eqref{116}
$\,$ for $t=0,1$,
\be\label{lijin1}\begin{array}{lll}||\lf \p_\xi^t B^{(11)}_{21}\rc||_{h_p^{11}\to h_p^{12}}&\le & || \Lambda_2^{(1)}||_{h_p^{12}\to h_p^{12}}\cdot ||(\Lambda_2^{(1)})^{-1}\lf \p_\xi^t B^{(11)}_{21}\rc||_{h_p^{11}\to h_p^{12}}\\ &\le & \sup_{ |j|\ge K_2} \{\lambda_j\}||\Lambda^{-1}\,\lf \p_\xi^t\,B\rc||_{h_p\to h_p}\\&\le& (1/8)\e_0 \;K^{-3c/y},\end{array}\ee
and, similarly,
\be\label{lijin2} ||\lf \p_\xi^t B^{(11)}_{12}\rc||_{h_p^{12}\to h_p^{11}}\le  \sup_{j\in\mathbb Z^d}\{\lambda_j\}||\Lambda^{-1}\, \lf \p_\xi^t\,B\rc||_{h_p\to h_p}\le C\,\e_0,\;t=0, 1. \ee
Therefore,
\be \label{17-3-3.76}||\lf \p_\xi^t\,\left(B_{12}^{(11)}\left(( k,\omega)\pm(\Lambda_2^{(1)}+B^{(11)}_{22})\right)^{-1}B_{21}^{(11)}\right) \rc||_{h_p^{11}\to h_p^{11}}\le C\,\e_0^2\, K^{2c/y}\,K^{-3c/y}K
\le \e_0\ll 1. \ee
By applying \eqref{116} and \eqref{17-3-3.76} to \eqref{170513bu1},
\be \label{17-3-3.77} ||\lf \p_\xi \; \mathcal C\rc||_{h_p^{11}\to h_p^{11}}\le C\e_0\ll 1.\ee
Since $B$ is self-adjoint in $\ell_2$, it is easy to verify that $\mathcal C$ is Hermitian.
Thus, by \eqref{17-15-3},
\be \label{17-3-3.77+1} ||\p_\xi \; \mathcal C||_{\ell_2\to \ell_2}\le ||\p_\xi \; \mathcal C||_{h_p^{11}\to h_p^{11}}\le ||\lf \p_\xi \; \mathcal C\rc||_{h_p^{11}\to h_p^{11}}
\le C\e_0\ll 1.\ee
And since $\mathcal C$ is Hermitian, there are a unitary matrix $U=U(\xi)$ and a diagonal matrix $\mathcal M=\mbox{diag}\; (\mu_j(\xi):\;j\in\mathbb Z^d,\; |j|\le K_2)$ such that
\be \label{1704031}\mathcal C=U^* \mathcal M \; U \ee where the star $*$ is the complex conjugate plus transpose.
By variation principle of eigenvalues and \eqref{17-3-3.77},
\[|\p_\xi\mu_j|\le C\e_0\ll 1.\]
By Lemma \ref{lemma3.2}, therefore, there is a subset $\mathcal {O}_4^{\prime}$ with its measure $\le K^{-N}$ such that for $\xi\in\mathcal O\setminus\mathcal {O}_4^{\prime},$
\be \label{17-3-3.79}||\mbox{diag}\;((k,\omega)+\mu_j:\; |j|\le K_2)^{-1}||_{\ell_2^{11}\to\ell_2^{11}}\le K^c.\ee
Note that $||U^*||_{\ell_2^{11}\to\ell_2^{11}}=||U||_{\ell_2^{11}\to\ell_2^{11}}=1$. It follows that
\be \label{17-3-3.79+1}||((k,\omega)+\mathcal C)^{-1}||_{\ell_2^{11}\to\ell_2^{11}}\le K^c.\ee
Let $\m C_{ij}$'s  be the elements of $\m C$. By the definition of $\m C$,
\[|i|\le K_2,\; |j|\le K_2.\]
Thus,
\[||(( k,\omega)+\mathcal C)^{-1}||_{h_p^{11}\to h_p^{11}}\le K_2^p\,K^c.\]
Moreover,   by Lemma \ref{finite-norm},
\be \label{17-3-3.79+1}||\lf(( k,\omega)+\mathcal C)^{-1}\rc ||_{h_p^{11}\to h_p^{11}}\le K_2^p\, K_2^{d/2}\,K^c<K_4.\ee
By applying \eqref{201707-23}, \eqref{lijin1}, \eqref{lijin2} and \eqref{17-3-3.79+1} to \eqref{17-3-29-1}, we have
\[||\lf\left(( k,\omega)\pm\left(\Lambda^{(1)}+B^{(11)}\right)\right)^{-1}\rc||_{h_p^1\to h_p^1}\le K_4.\]
So
\be \label{17-3-81}\begin{array}{lll}||\lf (1\otimes M^{(11)} )^{-1}\rc||_{h_p^2\otimes h_p^1\to h_p^2\otimes h_p^1}&=& ||\lf 1\otimes (M^{(11)})^{-1} \rc||_{h_p^2\otimes h_p^1\to h_p^2\otimes h_p^1}\\ &=&||\lf  (M^{(11)})^{-1} \rc||_{h_p^1\to h_p^1}  \\ &=&||\lf(( k,\omega)\pm(\Lambda^{(1)}+B^{(11)}))^{-1}\rc||_{h_p^1\to h_p^1}\\&\le & K_4.\end{array}\ee
 Noting
$N^{(22)T}=N^{(22)}.$
Again by \eqref{116} and \eqref{K_3},
\be\label{17-3.84} \begin{array}{lll}||\lf N^{(22)T})\otimes 1\rc||_{h_p^2\otimes h_p^1\to h_p^2\otimes h_p^1}&=&||\lf \Lambda^{(2)}+ \breve B^{(22)}\rc||_{h_p^2\to h_p^2}\\ &\le & || \Lambda^{(2)}||_{h_p^2\to h_p^2}\cdot ||(\Lambda^{(2)})^{-1}\, \lf \breve B^{(22)}\rc||_{h_p^2\to h_p^2}\\ &\le & \sup_{ |j|\ge K_3} \{\lambda_j\}||\Lambda^{-1}\,\lf \breve B\rc||_{h_p\to h_p}\\& \le& \e_0\, K_3^{-\kappa}.\end{array} \ee

Finally, the proof is finished by using Neumann series, \eqref{17-3-81} and \eqref{17-3.84}:
\[\begin{array}{lll}||\lf (1\otimes M^{(11)}+N^{(22)T}\otimes 1)^{-1} \rc ||&\le & ||\lf(1\otimes M^{(11)} )^{-1}\rc||||\lf (1+(1\otimes M^{(11)} )^{-1}(N^{(22)T}\otimes 1))^{-1} \rc||\\ &\le & ||\lf(1\otimes M^{(11)} )^{-1}\rc||\sum_{j=0}^\infty ||\lf((1\otimes M^{(11)} )^{-1}(N^{(22)T}\otimes 1))\rc||^j\\
&\le & K_4\sum_{j=0}^\infty (K_4\, K_3^{-\kappa}\,\e_0)^j\\ &\le & 2 K_4,
\end{array}
\] where $||\cdot||=||\cdot||_{h_p^2\otimes h_p^1\to h_p^2\otimes h_p^1}$. This completes the proof of the lemma.
\end{proof}

\ \

By \eqref{K_3} ( the definition of $K_3$) and \eqref{116},
\be\label{17-3.87} \begin{array}{lll}|\lf |N^{(12)T}\otimes 1\rc||_{h_p^1\otimes h_p^2\to h_p^2\otimes h_p^2 }&=&||\lf N^{(12)T}\rc||_{h_p^1\to h_p^2}=||\lf N^{(21)}\rc||_{h_p^1\to h_p^2}\\&\le &||\Lambda^{(2)}||_{h_p^2\to h_p^2}||(\Lambda^{(2)})^{-1}\lf N^{(21)}\rc||_{h_p^1\to h_p^2}\\ &\le& \e_0 K_3^{-\kappa}. \end{array} \ee
Similarly,
\be\label{17-3.88+1} ||\lf 1\otimes M^{(21)}\rc||_{h_p^2\otimes
h_p^1\to h_p^2\otimes
h_p^2}\le ||\lf M^{(21)} \rc||_{h_p^1\to h_p^2}\le\e_0 K_3^{-\kappa}.\ee
By \eqref{116},
\be\label{17-3.88} ||\lf N^{(21)T}\otimes 1\rc||_{h_p^2\otimes h_p^2\to h_p^1\otimes h_p^2 }=||N^{(12)}||_{h_p^2\to h_p^1}\le \e_0 \ee
 and
\be \label{17-3.88+2}  ||\lf 1\otimes M^{(12)}\rc||_{h_p^2\otimes h_p^2\to h_p^2\otimes h_p^1}\le
||\lf M^{(12)} \rc||_{h_p^2\to h_p^1}\le \e_0. \ee
\color{black}
In order to see that $Q$ is invertible,
we can partition $Q=Q_1+Q_2$ with
\begin{eqnarray}\nonumber&&Q_1=\left(\begin{array}{lll}
 1\otimes M^{(22)}+N^{(11)T}\otimes1&0& N^{(21)T}\otimes1  \\
0&1\otimes M^{(11)}+N^{(22)T}\otimes 1&1\otimes M^{(12)}\\  0& 0& 1\otimes M^{(22)}+N^{(22)T}\otimes 1
\end{array}\right)\end{eqnarray}
and
\begin{eqnarray}\nonumber&&Q_2=\left(\begin{array}{lll}
 0&0& 0  \\
0&0&0\\  N^{(12)T}\otimes 1&1\otimes M^{(21)}&0
\end{array}\right).\end{eqnarray}
By  \eqref{17-3.87} and \eqref{17-3.88+1},
\[||\lf Q_2\rc||_{\m H\to \m H}\le \e_0 K_3^{-\kappa}.\]
 And note that $Q_1$ is upper triangle matrix and
 \begin{eqnarray}\nonumber&&\m Q_0 Q_1=\left(\begin{array}{ccc}
 1\otimes M^{(22)}+N^{(11)T}\otimes1&0&0  \\
0&1\otimes M^{(11)}+N^{(22)T}\otimes 1&0\\  0& 0& 1\otimes M^{(22)}+N^{(22)T}\otimes 1
\end{array}\right),\end{eqnarray}
 where
  \begin{eqnarray}\nonumber&&\m Q_0=\left(\begin{array}{ccc}
 1&0&-(1\otimes M^{(22)}+N^{(22)T}\otimes 1)^{-1}(1\otimes N^{(21)T})  \\
0&1&-(1\otimes M^{(22)}+N^{(22)T}\otimes 1)^{-1}(1\otimes M^{(12)})\\  0& 0& 1
\end{array}\right).\end{eqnarray} Also note
 \begin{eqnarray}\nonumber&&\m Q_0^{-1}=\left(\begin{array}{ccc}
 1&0&(1\otimes M^{(22)}+N^{(22)T}\otimes 1)^{-1}(1\otimes N^{(21)T})  \\
0&1&(1\otimes M^{(22)}+N^{(22)T}\otimes 1)^{-1}(1\otimes M^{(12)})\\  0& 0& 1
\end{array}\right).\end{eqnarray}
In view of \eqref{17-3.88}, \eqref{17-3.88+2}   and Lemma \ref{lemma3.3},
\[||\lf Q_1^{-1}\rc||_{\m H\to \m H}\le K_4^2.\]
Note that $K_4^2\ll K_3^{\kappa}$.
In view of \eqref{17-3.87}, \eqref{17-3.88+1} and using Neumann series, it is easy to get
 \be \label{17-3.89} ||\lf Q^{-1}\rc||_{\m H\to \m H}= ||\lf Q_1^{-1}(1+Q_1^{-1}\, Q_2)^{-1}\rc||_{\m H\to \m H}\le C K^{2}_4.\ee
By \eqref{116},
\be\label{17-3.90} ||\lf \partial_\xi \, Q\rc||_{\m H\to \m H}\le C\, K.\ee
In view of \eqref{116}, \eqref{17-3.65}, \eqref{17-3.89} and \eqref{17-3.90},
\br \label{17-3.91}&&||\lf \partial_\xi( 1\otimes \m L_1)\rc||_{h_p^1\otimes h_p^1}=||\lf \partial_\xi\m A\rc||_{h_p^1\otimes h_p^1}\nonumber\\
&
\le& ||\lf \p_\xi\, Q^{-1}\rc||_{\m H\to\m H}\left( || (\Lambda^{2})^{-1}\lf M^{(21)}\rc||_{h_p^1\to h_p^2}+
||  (\Lambda^{2})^{-1} \lf N^{(21)T}\rc||_{h_p^1\to h_p^2}\right)\sup_{|j|>K_{3}}|\lambda_{j}|\\
&\le &
 \varepsilon_{0}\,C K\, K_4^{4}\, K_3^{-\kappa}\ll K^{-c}. \nonumber\er
By \eqref{17-15-3} in Appendices,
\be \label{17-3.91+1}||\partial_\xi( 1\otimes \m L_1)||_{\ell_2^1\to\ell_2^1}=||\partial_\xi\, \m A||_{\ell_2^1\to\ell_2^1}
\le ||\lf\partial_\xi\, \m A\rc||_{h_p^1\to h_p^1}
\le K^{-c}\ll \e_0. \ee

Note
\[1\otimes M^{(11)}+N^{(11)T}\otimes 1+\m A=1\otimes(( k,\omega))\pm 1\otimes B^{(11)}\pm \breve B^{(11)T}\otimes 1+\m A. \]
By \eqref{usefortwist}, \eqref{116} and \eqref{17-3.91+1}, we have
\be \partial_{k(\xi)} (1\otimes M^{(11)}+N^{(11)T}\otimes 1+\m A)\ge c^*-C\e_0>c^*/2>0,\ee
where $k(\xi)$ is the direction derivative such that $(k(\xi)\cdot \partial_\xi)(( k,\xi))=|k|$. (We have used the notation $X\ge Y$ when $X-Y$ is a Hermitian positive definite matrix for Hermitian matrices $X$ and $Y$.  ) In addition, $1\otimes M^{(11)}+N^{(11)T}\otimes 1+\m A$ is obviously Hermitian, since $B$ is self-adjoint in $\ell_2$. Let $\mu=\mu(\xi)$ be any eigenvalue of $1\otimes M^{(11)}+N^{(11)T}\otimes 1+\m A$. Then
\[ \left|\partial_{k(\xi)}\mu(\xi)\right|\ge c^*/2,\]
by the variation of eigenvalues for Hermitian matrix. Moreover, there is a subset $\m O_2\subset\m O$ with Meas $\m O_2\le K^{-N}$ such that for any $\xi\in \m O\setminus \m O_2$,
\be\label{1704041} ||(1\otimes M^{(11)}+N^{(11)T}\otimes 1+\m A)^{-1} ||_{\ell_2^1\otimes\ell_2^1\to \ell_2^1\otimes\ell_2^1}\le K^C,\ee
where $C$ is chosen large enough such that $K^{-C}\cdot K_{3}^{d}<K^{-N}.$
Thus, by \eqref{1703-2},
\[||\operatorname{Vec}\, F^{(11)}||_{\ell_2^1\otimes\ell_2^1}\le  K^C \, ||\m R^{(11)}||_{\ell_2^1\otimes\ell_2^1}.\]

\ \

Now we are in position to estimate $||\m R^{(11)}||_{\ell_2}$.
%
To that end, we introduce some notations. For $\tilde p,\ti q \in\{p,q\}$, let
$h_{\tilde p,\ti q}^{ij}=\m L(h_{\tilde p}^i\to h_{\tilde q}^j)$ be the set of all the bounded linear operators from $h_{\tilde p}^i$ to $h_{\tilde q}^j$, (recalling $h_{\tilde p}=h_{\tilde p}^1\oplus h_{\tilde p}^2$), and for any $X\in h_{\tilde p,\ti q}^{ij},$ define
\[||X||_{h_{\tilde p,\ti q}^{ij}}=||X||_{h_{\tilde p}^i\to h_{\tilde q}^j}.\]
Then $(h_{\tilde p,\ti q}^{ij},||\cdot||_{h_{\tilde p,\ti q}^{ij}})$ is a Banach space for $i,j\in\{1,2\},\; \ti p,\ti q\in\{p,q\}$.

\begin{lem}\label{lemma3.5} (i). Arbitrarily take $\lf Y\rc\in h_{p,q}^{22}$ (it follows
 $(\Lambda^{(2)})^{-1}\lf Y\rc\in h_{p,p}^{22}$ and $\lf Y\rc(\Lambda^{(2)})^{-1}\in h_{q,q}^{22}$).
 Then $\lf G_{22}(Y)\rc\in h_{p,p}^{22}\cap h_{q,q}^{22}$  and
\[||(\Lambda^{(2)})^{-1}\lf G_{22}(Y)\rc||_{h_{p,p}^{22}}\le 2K^{c} ||(\Lambda^{(2)})^{-1}\,\lf Y\rc||_{h_{p,p}^{22}},\; ||\lf G_{22}(Y)\rc(\Lambda^{(2)})^{-1}||_{h_{q,q}^{22}}\le 2\, K^{c} ||\lf Y\rc\,(\Lambda^{(2)})^{-1}||_{h_{q,q}^{22}}. \]

\ \

(ii).
Arbitrarily take $\lf Y\rc\in h_{p,q}^{12}$ ( it follows
 $(\Lambda^{(2)})^{-1}\lf Y\rc\in h_{p,p}^{12}$ and $\lf Y\rc(\Lambda^{(1)})^{-1}\in h_{q,q}^{12}$ ). Then $\lf G_{21}(Y)\rc\in h_{p,p}^{12}\cap h_{q,q}^{12}$ and
\[||(\Lambda^{(2)})^{-1}\lf G_{21}(Y)\rc||_{h_{p,p}^{12}}\le K_5 ||\lf Y\rc||_{h_{p,q}^{12}},\; ||\lf G_{21}(Y)\rc(\Lambda^{(1)})^{-1}||_{h_{q,q}^{12}}\le K_5 ||\lf Y\rc||_{h_{p,q}^{12}},\; K_5:=K^{4cq}. \]
(iii).
Arbitrarily take $\lf Y\rc\in h_{p,q}^{21}$ (it follows
 $(\Lambda^{(1)})^{-1}\lf Y\rc\in h_{p,p}^{21}$ and $\lf Y\rc(\Lambda^{(2)})^{-1}\in h_{q,q}^{21}$). Then $\lf G_{12}(Y)\rc\in h_{p,p}^{21}\cap h_{q,q}^{21}$ and
\[||\lf G_{12}(Y)\rc(\Lambda^{(2)})^{-1}||_{h_{p,p}^{21}}\le K_5\,||\lf Y\rc||_{h_{p,q}^{21}},\; ||(\Lambda^{(1)})^{-1}\lf G_{12}(Y)\rc||_{h_{q,q}^{21}}\le K_5\,||\lf Y\rc||_{h_{p,q}^{21}}. \]
\end{lem}
\begin{proof}

Recall $|\lambda_j|\le K_3^{-1/\kappa}$ for $|j|>K_3$.
By Assumption {\bf E},
 \be\label{17-04-3.95} ||\lf b^{(22)}\rc||_{h_{p,p}^{22}}\le \sup_{|j|>K_3}|\lambda_j|\, ||(\Lambda^{(2)})^{-1}
 \,\lf  b\rc||_{h_{p,p}^{22}}\le \e_0 K_3^{-\kappa}, \;
  b\in\{B,\breve B\},\ee
\[ ||\lf b^{(22)}\rc||_{h_{q,q}^{22}}\le \sup_{|j|>K_3}|\lambda_j|\, ||\lf b\rc\,(\Lambda^{(2)})^{-1}
 ||_{h_{q,q}^{22}}\le \e_0 K_3^{-\kappa},\;
  b\in\{B,\breve B\},\]
\be \label{17-04-3.95+1}||\lf \Lambda^{(2)}+b^{(22)}\rc||_{h_{p,p}^{22}}\le C\,K_3^{-\kappa},\; ||\lf \Lambda^{(2)}+b^{(22)}\rc||_{h_{q,q}^{22}}\le C\, K_3^{-\kappa},\; b\in\{B,\breve B\}. \ee

 {\it Proof of (i).}
Consider the operator equation with the unknown variable $X$:
\be \label{17-4-3.98} M^{(22)}X+X N^{(22)}=Y.\ee
Recall \eqref{17-3-17-1},
\[M^{(22)}= ( k,\omega)\pm(\Lambda^{(2)}+B^{(22)}),\quad N^{(22)}=\pm (\Lambda^{(22)}+\breve B^{(22)}).\]
Let
\[\m N_l=\pm(\Lambda^{(2)}+B^{(22)}),\; \m N_r=\pm (\Lambda^{(22)}+\breve B^{(22)}).\]
Note that the $(k,\omega)$ in the linear operator $M^{(22)}$ is a product by a scalar $(k,\omega)$ and an identity map. So the equation \eqref{17-4-3.98} can be rewritten as
\[\left(\frac{(k,\omega)}{2}+\m N_l\right)\, X+X\, \left(\frac{(k,\omega)}{2}+\m N_r\right)=Y.\]
Without loss of generality, we can assume $(k,\omega)>0$. Thus, by \eqref{107}, we have $(k,\omega)>K^{-c/y}>0$. According to \eqref{17-04-3.95+1},
\[||\m N_l||_{h_{\ti p,\ti p}}^{22}<C K_3^{-\kappa},\; ||\m N_r||_{h_{\ti p,\ti p}}^{22}<C K_3^{-\kappa}, \;\; \ti p\in\{p,q\}.\]
Note $K_3^{-\kappa}\ll K^{-c/y}.$ Thus for any $t>0,$
\[||\lf \exp(-t (\frac{(k,\omega)}{2}+\m N_l))\rc||_{h_{\ti p,\ti p}}\le e^{-\frac{t}{2}(K^{-c/y}-K_3^{-\kappa})}\le e^{- (t/4)\, K^{-c/y}}\] and
\[||\lf \exp(-t (\frac{(k,\omega)}{2}+\m N_r))\rc||_{h_{\ti p,\ti p}}\le e^{-\frac{t}{2}(K^{-c/y}-K_3^{-\kappa})}\le e^{- (t/4)\, K^{-c/y}}.\]
Thus
\[X=\int_0^\infty \exp(-t (\frac{(k,\omega)}{2}+\m N_l)) \, Y\, \exp(-t (\frac{(k,\omega)}{2}+\m N_r)) \, dt\]
is well-defines and solves \eqref{17-4-3.98}.
Moreover,
\[(\Lambda^{(2)})^{-1}X=\int_0^\infty \exp(-t (\frac{(k,\omega)}{2}+(\Lambda^{2)})^{-1}\,\m N_l\,\Lambda^{(2)})) \, ((\Lambda^{(2)})^{-1}\,Y) \, \exp(-t (\frac{(k,\omega)}{2}+\m N_r)) \, dt.\]
Thus,
\[||(\Lambda^{(2)})^{-1}\lf X\rc ||_{h_{ p, p}}\le ||(\Lambda^{2)})^{-1}\lf Y\rc ||_{h_{ p, p}} \int_0^{\infty}\; e^{-\frac{t}{2} K^{-c/y}}\; dt=2K^{c/y}\,||\lf Y\rc ||_{h_{ p,q}}
\leq 2K^{c}\,||\lf Y\rc ||_{h_{ p,q}}.\]
Similarly,
\[||\lf X\rc (\Lambda^{2)})^{-1}||_{h_{ p, p}}\le 2 K^c\, ||\lf Y\rc ||_{h_{ p,q}}.\]
This completes the proof of (i).

\ \

{\it Proof of (ii).} By the definition of $G_{21}$  (See \eqref{126}),  the operator valued equation $G_{21}(Y)=X$ reads
\begin{equation}\label{126+1}
 \pm(\Lambda^{(2)}+B^{(22)})X+X\left( ( k,\omega)\pm (\Lambda^{(11)}+B^{(11)})\right)-G_{22}(X (\pm\breve B^{(12)}))(\pm\breve B^{(21)})=Y,\end{equation} where we used $(k,\omega)\, X=X\, (k,\omega)$.
Set $\m I:=( k,\omega)\pm(\Lambda^{(11)}\pm \breve B^{(11)})$. 
By \eqref{17-3-29-1} and \eqref{1704031}, one has
\be \m I^{-1}=\begin{pmatrix} 1& \breve B_{12}^{(11)}\m B^{-1}\\0&1\end{pmatrix}\begin{pmatrix} U & 0\\  0&1\end{pmatrix}\begin{pmatrix} ((k,\omega)+\m M)^{-1} & 0\\  0&\m B^{-1}\end{pmatrix} \begin{pmatrix} U^{*} & 0\\  0&1\end{pmatrix}\begin{pmatrix} 1 &0\\\m B^{-1} \breve B^{(11)}_{21} &1\end{pmatrix},\ee
where $\m B= ( k,\omega)\pm(\Lambda_2^{(1)}+\breve B^{(11)}_{22})$ and the star $*$ means complex conjugate plus transpose.

In view of the partition \eqref{bucong-3.74} and \eqref{bucong-3.75},   and using  Lemma \ref{splitlemma} in the Appendices,  and  Assumption {\bf E}, we have that
for $\tilde p\in \{p,q\}$ and $b\in\{B,\breve B\}$,
\be ||\lf b^{(11)}_{22}\rc||_{h_{\ti p}^{12}\to  h_{\ti p}^{12} }, \;||\lf b^{(11)}_{12}\rc||_{ h_{\ti p}^{12}\to  h_{\ti p}^{11} },\;
 ||\lf b^{(11)}_{21}\rc||_{h_{\ti p}^{11}\to  h_{\ti p}^{12} }\le ||\lf b^{(11)}\rc||
 _{h_{\ti p}^{11}\to h_{\ti p}^{11}}\le \e_0,\ee
\be\begin{array}{ll} ||\lf \Lambda_2^{(1)}+b^{(11)}_{22}\rc||_{ h_{ p}^{12}\to  h_{ p}^{12} }&\le ||\Lambda_2^{(1)}||_{ h_{ p}^{12}\to  h_{ p}^{12} }+||\Lambda_2^{(1)}||_{ h_{ p}^{12}\to  h_{ p}^{12} }\,||(\Lambda_2^{(1)})^{-1}\lf b_{22}^{(11)}\rc||_{ h_{ p}^{12}\to  h_{ p}^{12} }\\
&\le (1+\e_0)\sup_{|j|\ge K_2}|\lambda_j|\le (1+\e_0)K_2^{-\kappa}\\&\le\frac12(1+\e_0)K^{-c/y}\end{array}\ee
and
\be\label{17-0701-3.128}\begin{array}{ll} ||\lf \Lambda_2^{(1)}+b^{(11)}_{22}\rc||_{ h_{ q}^{12}\to  h_{ q}^{12} }&\le ||\Lambda_2^{(1)}||_{ h_{ q}^{12}\to  h_{ q}^{12} }+||\Lambda_2^{(1)}||_{ h_{ q}^{12}\to  h_{ q}^{12} }||\lf b_{22}^{(11)}\rc \,(\Lambda_2^{(1)})^{-1}||_{ h_{ q}^{12}\to  h_{ q}^{12} }\\
&\le (1+\e_0)\sup_{|j|\ge K_2}|\lambda_j|\\&\le(1+\e_0)(\frac{1}{2} K^{-c/y}).\end{array}\ee
Write
\[\m B= (k,\omega)\left( 1+\frac{1}{\pm(k,\omega)}(\Lambda_2^{(1)}+\breve B_{22}^{(11)})\right).\]

Recall $|(k,\omega)|\ge K^{-c/y}$. By Neumann series and using \eqref{17-0701-3.128},
\be \label{1704032}||\lf \m B^{-1}\rc||_{ h_{\ti p}^{12} \to h_{\ti p}^{12} }\le K^{c/y}\sum_{j=0}^\infty \left(||\frac{1}{
|(k,\omega)|}\lf(\Lambda_2^{(1)}+\breve B^{(11)}_{22})\rc||_{ h_{\ti p}^{12} \to h_{\ti p}^{12} }\right)^j\le 2K^{c/y}.\ee
Since $U$ is unitary, so $||U||_{ \ell_{2}^{12} \to \ell_{0}^{12} }=1$, and furthermore,
\[ ||U||_{ h_{\ti p}^{12} \to h_{\ti p}^{12} }=||\text{diag}\,(|j|^{\ti p}:|j|\le K_2)U \text{diag}\,(|j|^{-\ti p}:|j|\le K_2)||_{ \ell_{2}^{12} \to \ell_{2}^{12} }\le K_2^{\ti p}||U||_{ \ell_{2}^{12} \to \ell_{2}^{12} }\le K^{3c q/y}. \]
By Lemma \ref{finite-norm}, moreover,
\be\label{20170724} ||U||_{ h_{\ti p}^{12} \to h_{\ti p}^{12} }\le K^{3c q/y}K_2^{d/2}. \ee
By \eqref{103} and noting $\m M$ is diagonal,
\be \label{1704035}||\lf (( k,\omega)+\m M)^{-1}\rc||_{\ti h_{\ti p}^2\to \ti h_{\ti p}^2 }=||(( k,\omega)+\m M)^{-1}||_{\ti h_{\ti p}^2\to \ti h_{\ti p}^2 } \le K^c.\ee
Consequently,
\be \label{1704036} ||\lf \m I^{-1}\rc||_{h_{\ti p,\ti p}^{11}}\le C (K^c+K^{cq/y})^3< K_5\ll K_3.\ee
Let $\pm=+$ without loss of generality in \eqref{126+1}.
Construct a Picard sequence as follows:
\be\begin{array}{ll} X_0&=Y\m I^{-1},\\
X_{\nu+1}&=X_0- \left((\Lambda^{(2)}+B^{(22)})X_\nu+(G_{22}(X_\nu \breve B^{(12)}))\breve B^{(21)})\right)\m I^{-1}, \quad\nu=0,1,2,... .\end{array}\ee
It follows
\[\begin{array}{lll}(\Lambda^{(2)})^{-1}X_{\nu+1}&=&(\Lambda^{(2)})^{-1}Y\m I^{-1} \\ & & - \left((1+(\Lambda^{(2)})^{-1}B^{(22)}) \Lambda^{(2)}((\Lambda^{(2)})^{-1}\,X_\nu)\right)\m I^{-1}\\&&+\left(\left((\Lambda^{(2)})^{-1}(G_{22}(X_\nu \breve B^{(12)}))\right)\Lambda^{(2)}((\Lambda^{(2)})^{-1}\,\breve B^{(21)})\right)\m I^{-1}.\end{array}\]
By \eqref{116},
\[|| (\Lambda^{(2)})^{-1}\lf B^{(22)}\rc||_{h_p^2\to h_p^2}\le \e_0, \;\;|| (\Lambda^{(2)})^{-1}\lf\breve B^{(21)}\rc||_{h_p^1\to h_p^2}\le\e_0.\]
By (i) of this Lemma,
\[\left|\left|(\Lambda^{(2)})^{-1}\lf (G_{22}(X_\nu \breve B^{(12)}))\rc\right|\right|_{h_p^2\to h_p^2}\le
 2K^c\left|\left| (\Lambda^{(2)})^{-1}\, \lf X_\nu\breve B^{(12)}\rc \right|\right|_{h_p^2\to h_p^2}
 \le 2K^c
 \varepsilon_{0}\left|\left|(\Lambda^{(2)})^{-1}\, \lf X_\nu\rc \right|\right|_{h_p^1\to h_p^2}.\]
By the definition of $K_{3},$ $||\Lambda^{(2)}||_{h_p^2\to h_p^2}\leq K_{3}^{-\kappa}.$
Consequently
\[\begin{array}{lll} & &||(\Lambda^{(2)})^{-1}\, \lf X_{\nu+1}\rc- (\Lambda^{(2)})^{-1}\, \lf X_{\nu}\rc||_{h_{p,p}^{12}}\\ &\le & \left((1+\e_0)K_5\, K_3^{-\kappa}+2 \right)||(\Lambda^{(2)})^{-1}\, \lf X_{\nu}\rc- (\Lambda^{(2)})^{-1}\, \lf X_{\nu-1}\rc||_{h_{p,p}^{12}}\\ &\le & \e_0\, ||(\Lambda^{(2)})^{-1}\, \lf X_{\nu}\rc- (\Lambda^{(2)})^{-1}\, \lf X_{\nu-1}\rc||_{h_{p,p}^{12}}.
\end{array}\]
It follows that $(\Lambda^{(2)})^{-1}X_{\nu}$ is convergent to
$(\Lambda^{(2)})^{-1}X$ in the norm $||\lfloor \cdot\rceil||_{h^{11}_{p,p}},$
and
\[ ||(\Lambda^{(2)})^{-1}\lf X \rc||_{h_p^1\to h_p^2} \le K_5 ||(\Lambda^{(2)})^{-1}\lf Y \rc||_{h_p^1\to h_p^2}.\]
Similarly,
\[ ||\lf X\rc\,(\Lambda^{(1)})^{-1} ||_{h_q^1\to h_q^2} \le K_5||\lf Y\rc (\Lambda^{(1)})^{-1} ||_{h_q^1\to h_q^2}.\]
We omit the detail. This completes the proof of (ii).

\ \

{\it Proof of (iii).} The proof is similar to that of (ii). We omit it here.
\end{proof}

 Recall that $$||\Lambda^{-1}\lf R\rc||_{h_{p}\to h_p}\le ||\lf R\rc||_{h_p\to h_q},\;||\lf R\rc\Lambda^{-1}||_{h_{q}\to h_q}\le ||R||_{h_p\to h_q},$$
 $$||\Lambda^{-1}\lf B\rc||_{h_{p}\to h_p}\le ||B||_{h_p\to h_q}\le C\,\e_0,\;\;
 ||\lf B\rc\Lambda^{-1}||_{h_{q}\to h_q}\le ||\lf B\rc||_{h_p\to h_q}\le C\,\e_0$$ and
 $$N^{(21)}=\pm B^{(21)},\;||\lf N^{(21)}\rc||_{h_p^1\to h_p^2}\le ||\Lambda^{(2)}||_{h_p^2\to h_p^2}||
 (\Lambda^{(2)})^{-1}\lf N^{(21)}\rc||_{h_p^1\to h_p^2}\le \e_0\, K_3^{-\kappa}.$$
By (i)  of Lemma  \ref{lemma3.5}, one has
\[\begin{array}{lll}||(\Lambda^{(2)})^{-1}\lf G_{22}(R^{(22)}N^{(21)})\rc||_{h_{p}^1\to h_p^2}&\le& K^{2c} ||(\Lambda^{(2)})^{-1}\lf R^{(22)}\rc||_{h_{p}^2\to h_{p}^{2}} ||\lf N^{(21)}\rc||_{h_{p}^1\to h_p^{2}}\\&
\le &\e_0 \, K^{2c}K_3^{-\kappa}\,||\lf R\rc||_{h_p\to h_q}.\end{array}\]
By applying Lemma  \ref{lemma3.5} to \eqref{128}, one has
\begin{eqnarray*}
&&||(\Lambda^{(2)})^{-1}\lf\ti R^{(21)}\rc||_{h_{p}^1\to h_{p}^{2}}=||(\Lambda^{(2)})^{-1}G_{21}(R^{(21)}-G_{22}(R^{(22)})N^{(21)})||_{h_{p}^1\to h_{p}^{2}}\\
&& \leq K_{5}||\lfloor R^{(21)}-G_{22}(R^{(22)}N^{(21)})\rceil||_{h_{p}^1\to h_{q}^{2}}
\leq K_{5}(||\lfloor R^{(21)}\rceil||_{h_{p}^1\to h_{q}^{2}}+||\lfloor G_{22}(R^{(22)}N^{(21)})\rceil||_{h_{p}^1\to h_{q}^{2}})\\
&&=K_{5}||\lfloor R^{(21)}\rceil||_{h_{p}^1\to h_{q}^{2}}+K_{5}||(\Lambda^{(2)})^{-1}\lfloor G_{22}(R^{(22)}N^{(21)})\rceil||_{h_{p}^1\to h_{p}^{2}}\\
&&\leq K_{5}||\lfloor R^{(21)}\rceil||_{h_{p}^1\to h_{q}^{2}}+K_{5}K^{c}||(\Lambda^{(2)})^{-1}\lfloor (R^{(22)}N^{(21)})\rceil||_{h_{p}^1\to h_{p}^{2}}\\
&&\leq K_{5}||\lfloor R^{(21)}\rceil||_{h_{p}^1\to h_{q}^{2}}+K_{5}K^{c}\varepsilon_{0}||\lfloor R^{(22)}\rceil||_{h_{p}^2\to h_{p}^{2}}\\
&&\le2\e_0\,K^{c}\,K_5\, ||\lf R\rc||_{h_p\to h_q}.
\end{eqnarray*}
Moreover, applying Lemma  \ref{lemma3.5} to \eqref{134}, one has
\begin{eqnarray*}
&&||(\Lambda^{(1)})^{-1}\lf\ti R^{(12)}\rc||_{h_{p}^2\to h_{p}^{1}}=||(\Lambda^{(1)})^{-1}\lf G_{12}\left(M^{(12)}\left(G_{22}(R^{(22)})-G_{22}(\tilde{R}^{(21)}N^{(12)})\right)-R^{(12)}\right)\rc||_{h_{p}^2\to h_{p}^{1}}\\
&&\le K_{5}||(\Lambda^{(1)})^{-1}\left(M^{(12)}\left(G_{22}(R^{(22)})-G_{22}(\tilde{R}^{(21)}N^{(12)})\right)-R^{(12)}\right)||_{h_{p}^2\to h_{p}^{1}}\\
&&\le K_{5}||(\Lambda^{(1)})^{-1}\lfloor M^{(12)}\rceil||_{h_{p}^2\to h_{p}^{1}}\cdot ||\Lambda^{(2)}||_{h_{p}^2\to h_{p}^{2}}
\cdot ||(\Lambda^{(2)})^{-1}\lfloor G_{22}(R^{(22)}-(\tilde{R}^{(21)}N^{(12)}))\rceil||_{h_{p}^2\to h_{p}^{2}}\\
&&+K_{5}||(\Lambda^{(1)})^{-1}R^{(12)}||_{h_{p}^2\to h_{p}^{1}}\\
&&\le K_{5}\,\varepsilon_{0}\,K_{3}^{-\kappa}||(\Lambda^{(2)})^{-1}\lfloor R^{(22)}-\tilde{R}^{(21)}N^{(12)}\rceil||_{h_{p}^2\to h_{p}^{2}}
+K_{5}\,||\lfloor R\rceil||_{h_{p}\to h_{q}}\\
&&\le K_{5}\,\varepsilon_{0}\,K_{3}^{-\kappa}(||\lfloor R^{(22)}\rceil||_{h_{p}^2\to h_{q}^{2}}+||(\Lambda^{(2)})^{-1}\lfloor \tilde{R}^{(21)}\rceil||_{h_{p}^1\to h_{p}^{2}}\cdot \varepsilon_{0})+K_{5}\,||\lfloor R\rceil||_{h_{p}\to h_{q}}\\
&&\le 2\, K_5\, ||\lf R\rc||_{h_p\to h_q}.
\end{eqnarray*}
Finally, applying  Lemma  \ref{lemma3.5} to \eqref{138},
one has
\be  ||\lf\m R^{(11)}\rc||_{h_{p,p}^{11}}\le K_5^2\, ||\lf R\rc||_{h_p\to h_q}.\ee
Note that we can write $\m R^{(11)}=(\m R^{(11)}_{ij}:|i|\le K_3,|j|\le K_3)$. So
\be\label{1704042} ||\text{Vec}\,\m R^{(11)}||_{\ell_2^1\otimes \ell_2^1}\le ||\text{Vec}\,\m R^{(11)}||_{h_p^1\otimes h_p^1}\le C\, K_3^d\,K_5^2 \, ||\lf R\rc||_{h_p\to h_q}.\ee
By \eqref{1704041} and \eqref{1704042} to the equation \eqref{1703-2}, one gets
\[||\text{Vec}\; F^{(11)}||_{\ell_2^1\otimes\ell_2^1}\le K^{C}||\lf R\rc||_{h_p\to h_q},\] where
$C$ is a constant large enough such that  $K^C>K_3^d\, K_5^2$.
Noting that we can write $F^{(11)}=(F^{(11)}_{ij}:|i|\le K_3,|j|\le K_3)$, we have
\be \label{1704043} ||\lf F^{(11)}\rc||_{h_p^1\to h_q^1}\le K^{5C}||\lf R\rc||_{h_p\to h_q}. \ee

By \eqref{1704043} and applying Lemma \ref{lemma3.5} to \eqref{127}, \eqref{129} and \eqref{132}, one has
\be\label{1704045}  ||\lf\hat{F}(k)\rc||_{h_p\to h_q}=||\lf \Lambda^{-1}\hat{F}(k)\rc||_{h_p\to h_p}\le  K^{5C}||\lf\hat{R}(k)\rc||_{h_p\to h_q}, \;\forall\; 0<|k|<K.\ee
It follows \eqref{5.8} for $k\neq 0$.
 Applying $\p_\xi$ to both sides of equation \eqref{II} and using \eqref{5.8}, we can prove \eqref{5.10}. Here we omit the detail.

\ \

\noindent{\bf Case 2. $k= 0$.}

\ \

At this case, the sign  $\pm$ take $+$. So we can rewrite \eqref{17-3-12-1}:
\be\label{2017-7-6-3} (\Lambda+B)F+ F(\Lambda+\breve B)=R, \ee
where $F:=\widehat{{F}}(0),R:=\widehat{{R}}(0).$ Consider the equation with the unknown variable $X$:
\be \label{17-4-3.x1}\Lambda X+X \Lambda=Y.\ee

Set
\[g=\left( \Lambda(\cdot)+ (\cdot)\Lambda\right)^{-1}. \]
Then $g(Y)=X$.
Writing \eqref{17-4-3.x1} in its elements of matrix
\be\label{5.58-1x2} X_{ij}=\frac{Y_{ij}}{\lambda_i+\lambda_j},\quad i,j\in\mathbb Z^d. \ee
Moreover,
\be\label{5.58} X_{ij}=\alpha_{ij}\;\tilde Y_{ij}=\beta_{ij}\breve{Y}_{ij},\quad i,j\in\mathbb Z^d, \ee
where \be \alpha_{ij}=\frac{\lambda_i}{\lambda_i+ \lambda_j},\;\beta_{ij}=\frac{\lambda_j}{\lambda_i+ \lambda_j},\quad \tilde Y_{ij}=\lambda_i^{-1}Y_{ij},\; \breve{Y}_{ij}=Y_{ij}\lambda_{j}^{-1}.\ee
Note that $0<\alpha_{ij}<1$ and $0<\beta_{ij}<1$.
Thus, $|X_{ij}|\le |\tilde Y_{ij}|$ and $|X_{ij}|\le |\breve Y_{ij}|$. Moreover,
\be \label{26-1}||\lf g(Y)\rc||_{h_p\to h_p}=||\lf X\rc||_{h_p\to h_p}\le ||\lf \tilde Y\rc||_{h_p\to h_p}= ||\lf Y\rc||_{h_p\to h_q}\ee and
\be \label{26-2}||\lf g(Y)\rc||_{h_q\to h_q}=||\lf X\rc||_{h_q\to h_q}\le ||\lf \breve Y\rc||_{h_q\to h_q}= ||\lf Y\rc||_{h_p\to h_q}.\ee
%
%
For an operator $X$ with $\lfloor X\rceil \in\mathcal{{L}}(h_{p}\to h_{p})$ and $\lfloor X\rceil \in\mathcal{{L}}(h_{q}\to h_{q}),$ we call that $X$ is a $\tau-$ approximate
solution of equation \eqref{2017-7-6-3}, if
$$||\lfloor (\Lambda+B)X+X(\Lambda+\breve{B})-R\rceil||_{h_{p}\to h_{q}}\leq \tau.$$
Let $X_{0}=g(R).$ By \eqref{26-1} and \eqref{26-2}, we obtain
$$||\lfloor X_{0}\rceil||_{h_{p}\to h_{p}}\leq ||\lfloor R\rceil||_{h_{p}\to h_{q}}\triangleq \delta,\;\;
||\lfloor X_{0}\rceil||_{h_{q}\to h_{q}}\leq ||\lfloor R\rceil||_{h_{p}\to h_{q}}= \delta.$$
Then
\begin{eqnarray}\label{8.14-1}
&&||\lfloor(\Lambda+B)X_{0}+X_{0}(\Lambda +\breve{B})-R\rceil||_{h_{p}\to h_{q}}
=||\lfloor(\Lambda X_{0}+X_{0}\Lambda-R)+BX_{0} +X_{0}\breve{B}\rceil||_{h_{p}\to h_{q}}\nonumber\\
&&=||\lfloor BX_{0} +X_{0}\breve{B}\rceil||_{h_{p}\to h_{q}}
\leq ||\lfloor X_{0}\rceil||_{h_{p}\to h_{p}}\cdot||\lfloor B\rceil||_{h_{p}\to h_{q}} +||\lfloor \breve{B}\rceil||_{h_{p}\to h_{q}}\cdot||\lfloor X_{0}\rceil||_{h_{q}\to h_{q}}\nonumber\\
&&\leq 2\varepsilon_{0}\delta.
\end{eqnarray}
That implies $X_{0}$ is a $2\varepsilon_{0}\delta-$ approximate solution of \eqref{2017-7-6-3}. Let $X=X_{0}+Y$ and insert it into \eqref{2017-7-6-3}.
Then
\be \label{8.14-2}
(\Lambda+B)X_{0}+X_{0}(\Lambda+\breve{B})-R+\Lambda Y+Y\Lambda+BY+Y\breve{B}=0.
\ee
Consider
\be \label{8.14-3}
\Lambda Y+Y\Lambda=R-((\Lambda+B)X_{0}+X_{0}(\Lambda+\breve{B})):=R_{1}.
\ee
Then $$Y=g(R_{1}).$$
By \eqref{8.14-1}, we have
\be\label{8.14-4}
||\lfloor Y\rceil||_{h_{p}\to h_{q}}\leq 2\varepsilon_{0}\delta,\;\;||\lfloor Y\rceil||_{h_{q}\to h_{q}}\leq 2\varepsilon_{0}\delta.
\ee
Let $X_{1}=Y$ and $X^{(1)}=X_{0}+X_{1}.$ By \eqref{8.14-2} and \eqref{8.14-4}, we have
\begin{eqnarray*}
&&||\lfloor (\Lambda+B)X^{(1)}+X^{(1)}(\Lambda+\breve{B})-R\rceil||_{h_{p}\to h_{q}}=||\lfloor BX_{1}+X_{1}\breve B\rceil||_{h_{p}\to h_{q}}\\
&&\leq ||\lfloor X_{1}\rceil||_{h_{p}\to h_{p}}||\lfloor B\rceil||_{h_{p}\to h_{q}}+||\lfloor \breve{B}\rceil ||_{h_{p}\to h_{q}}||\lfloor X_{1}\rceil||_{h_{q}\to h_{q}}
\leq (2\varepsilon_{0})^{2}\delta.
\end{eqnarray*}
Thus $X^{(1)}$ is $(2\varepsilon_{0})^{2}\delta-$ approximate solution of \eqref{2017-7-6-3}.

Assume that
$$X^{(m)}=X_{0}+X_{1}+\cdots +X_{m-1}$$
is a $(2\varepsilon_{0})^{m+1}\delta-$ approximate solution of \eqref{2017-7-6-3}. Thus
$$R_{m}:=R-((\Lambda+B)X^{(m)}+X^{(m)}(\Lambda+\breve{B}))$$
satisfies
$$||\lfloor R_{m}\rceil||_{h_{p}\to h_{q}}\leq (2\varepsilon_{0})^{m+1}\delta.$$
Take $X_{m}=g(R_{m}).$ That means  $\Lambda X_{m}+X_{m}\Lambda=R_{m}.$
Then
\be\label{8.14-5}\begin{array}{c}
                   ||\lfloor X_{m}\rceil|| _{h_{p}\to h_{p}}\leq ||\lfloor R_{m}\rceil||_{h_{p}\to h_{q}}\leq (2\varepsilon_{0})^{m+1}\delta,\\
                   ||\lfloor X_{m}\rceil|| _{h_{q}\to h_{q}}\leq ||\lfloor R_{m}\rceil||_{h_{p}\to h_{q}}\leq (2\varepsilon_{0})^{m+1}\delta.
                 \end{array}
\ee
Let $$X^{(m+1)}=X_{0}+X_{1}+\cdots +X_{m-1}+X_{m}=X^{(m)}+X_{m}.$$
Then
It follows \eqref{8.14-5} that
\begin{eqnarray*}
&&||\lfloor (\Lambda +B)X^{(m+1)}-X^{(m+1)}(\Lambda+\breve{B})-R_{m}\rceil||_{h_{p}\to h_{q}}\\
&&=||\lfloor \Lambda X_{m}-X_{m}\Lambda-R_{m}+BX_{m}+X_{m}\breve{B}\rceil||_{h_{p}\to h_{q}}\\
&&=||\lfloor BX_{m}+X_{m}\breve{B}\rceil||_{h_{p}\to h_{q}}\leq
||\lfloor X_{m}\rceil||_{h_{p}\to h_{p}}||\lfloor B\rceil||_{h_{p}\to h_{q}}+||\lfloor \breve{B}\rceil ||_{h_{p}\to h_{q}}||\lfloor X_{m}\rceil||_{h_{q}\to h_{q}}\\
&&\leq (2\varepsilon_{0})^{m+2}\delta.
\end{eqnarray*}
Thus $X^{(m+1)}$ is a $(2\varepsilon_{0})^{m+2}\delta-$ approximate solution.

By induction, we can assume
$$\sum_{m=0}^{\infty}X_{m}=X,\;\;\text{in}\;\;||\cdot||_{h_{p}\to h_{p}},$$
$$\sum_{m=0}^{\infty}X_{m}=\widetilde{X},\;\;\text{in}\;\;||\cdot||_{h_{q}\to h_{q}}.$$
For any $L>0$ and $M=(m_{ij}\in \mathbb{C}: i, j \in \mathbb{Z}^{d}),$ define
$$M^{L}=(m_{ij}\in \mathbb{C}: \;|i|\leq L,\;|j|\leq L).$$
Similarly, we can define $h_{p}^{L}$ and $h_{q}^{L}.$
By Lemma \ref{splitlemma}, we have
$$\sum_{m=0}^{\infty}X_{m}^{L}=X^{L},\;\;\text{in}\;\;||\cdot||_{h_{p}^{L}\to h_{p}^{L}},$$
$$\sum_{m=0}^{\infty}X_{m}^{L}=\widetilde{X}^{L},\;\;\text{in}\;\;||\cdot||_{h_{q}^{L}\to h_{q}^{L}}.$$
By $\dim X^{L}<\infty$ and $\dim \widetilde{X}^{L}<\infty,$ we have $X^{L}=\widetilde{X}^{L}.$
Since $L>0$ is arbitrary, we have $X=\widetilde{X}=\widehat{F}(0, \xi).$
Therefore, $$X=\sum_{m=0}^{\infty} X_{m}$$ solves \eqref{2017-7-6-3} and
\be\label{1704049}||\lfloor X\rceil||_{h_{p}\to h_{p}}\leq \delta+\sum_{m=0}^{\infty}(2\varepsilon_{0})^{m+1}\delta\leq 2\delta,\;\;\;\;||\lfloor X\rceil||_{h_{q}\to h_{q}}\leq \delta+\sum_{m=0}^{\infty}(2\varepsilon_{0})^{m+1}\delta\leq 2\delta.\ee
This completes the proof of \eqref{2017-7-6-1}. Applying $\p_\xi$ to both sides of \eqref{2017-7-6-3},  we can prove \eqref{2017-7-6-2} similarly.

Up to now, the proof of Lemma \ref{lem5.1} is completed.$\qed$

\ \
\begin{rem} Comparing \eqref{1704049}  and \eqref{1704045}, we find that $ \widehat{F}(k)$ with $k\neq 0$ is of regularity of order $\kappa$, while $ \widehat{F}(0)$ has no regularity.
\end{rem}

\section{\label{section4}Iterative Lemma}

Before giving the iterative lemma, we need the following iterative constants and domains:
\begin{itemize}

\item $m$ -  number of the iterative steps;

\item $C,C_1,C_2,...$- positive constants which arrive in estimates. They are independent of $\e_{0}$ and $m$, maybe different in different position of the text;

\item $\e_m=\e_0^{(1+\varrho_0)^m}$, which measures the size of the perturbation in the $m^{\text{th}}$ iteration, $m=1,2,...,$ where $\varrho_0>0$ is an absolute constant;

\item $e_m=\frac{1^{-2}+2^{-2}+\cdots+m^{-2}}{2(1^{-2}+2^{-2}+\cdots)}$ (so $0<e_m<\frac12$ for all $m$);

\item $s_m=s_0(1-e_m)$ (so $s_m>\frac12 s_0$ for all $m$), which measure the width of the angle variable $x$ in  the $m^{\text{th}}$ iteration, here $s_0>0$ is an
absolute constant;

\item  $r_m=r_0(1-e_m)$ (so $r_m>\frac12 r_0$ for all $m$), which measure the radius of the action variable $y$ as well as the normal coordinate $(z,\bar z)$ in  the $m^{\text{th}}$ iteration, $r_0>0$ is an absolute constant;
\item $s_m^j=(1-\frac{j}{6})s_m+\frac{j}{6}s_{m+1}$, \; ($j=0,1,...,6$) which is a bridge between $s_m$ and $s_{m+1}$;

\item  $r_m^j=(1-\frac{j}{6})r_m+\frac{j}{6}r_{m+1}$,\; ($j=0,1,...,6$) which is a bridge between $r_m$ and $r_{m+1}$;

\item $K_m=\frac{2}{s_m^5-s_m^6}|\log\e_m|$ which truncates a periodic function into essential part and unessential one in the term of its Fourier coefficient;

\item  $C(m), C_1(m),...$ functions of $m$ and of the form $C_1 m^{C_2}$ or $C_1 2^{C\,m}$;

\item $D_p(s_m,r_m)=\{(x,y,z,\bar z)\in \m P^{p}:\; |\Im x|< s_m, |y|< r^{2}_m, ||z||_{p}< r_m, ||\bar z||_{p}< r_m\}$, which denotes a complexificated neighborhood of the torus
    \[\m T_0:=\mathbb T^N\times\{0\}\times\{0\}\times\{0\}\subset\m P^{p}. \] Obviously,
    \[D_p(s_0,r_0)\supset D(s_1,r_1)\supset\cdots\supset D_p(s_m,r_m)\supset\cdots\supset D_{p}(\frac{s_0}{2},\frac{r_0}{2});\]

\item $\mathbb T^N_{s_m}=\{x\in \mathbb C^n/(2\pi\mathbb Z)^n:\: |\Im x|\le s_m\}$, which is a complixificated neighborhood of $\mathbb T^N$ with strip width $s_m$. Obviously,
\[\mathbb T^N_{s_0}\supset \mathbb T^N_{s_1}\supset\cdots\supset \mathbb T^N_{s_m}\supset\cdots\supset \mathbb T^N_{s_0/2}.\]

\end{itemize}

The proof of Theorem \ref{theorem2} will be completed by $m\rightarrow+\infty$ in the following lemma:

\begin{lem}\label{lemma1} Let $\omega^0,\lambda_j$'s and $B^0$ obey the assumptions {\bf A}, {\bf B} and {\bf E} in Theorem \ref{theorem2}, respectively. Suppose that we have had $m+1$ Hamiltonian functions
\begin{equation}\label{10} H^{(l)}=H_0^{(l)}+R^{(l)}+P^{(l)},\quad l=0,1,\cdots,m,\end{equation}
where
\begin{equation}\label{11}H_0^{(l)}=(\omega^{(l)}(\xi),y)+
\sum_{j\in\mathbb{Z}^d}\lambda_j(\xi)z_j\bar{z}_j+\langle B^{(l)z\bar{z}}(\xi)z,\bar{z}\rangle.\end{equation}
And suppose that there are $m+1$ closed parameters sets
\begin{equation}\label{12}
\mathbb{R}^{N}\supset \mathcal{O}_0\supset\mathcal{O}_1
\supset\cdots\supset\mathcal{O}_m\end{equation}
and $m+1$ domains
\begin{equation}\label{13}D_p(s_0,r_0)\supset\cdots\supset D_p(s_l,r_l)\supset\cdots D_p(s_m,r_m)\supset D_p(s_0/2,r_0/2)\end{equation}
such that
\\
\noindent $(1)_l$
\begin{equation}\label{14}\omega^{(l)}=\omega^{0}+\sum_{j=1}^{l}\omega_{(j)},\quad l=1,2,\cdots, m,\quad\omega^{(0)}=\omega^{0},\end{equation}
where the function $\omega_{(j)}=\omega_{(j)}(\xi):\mathcal{O}_j\rightarrow\mathbb{R}^{N}$ is smooth, and
\begin{equation}\label{15}\sup_{\xi\in\mathcal{O}_j}|\omega_{(j)}|\leq C(j-1)\,\epsilon_{j-1},\quad
\sup_{\xi\in\mathcal{O}_j}|\partial_{\xi}\omega_{(j)}|\leq C(j-1)\,\epsilon_{j-1},\quad j=1,\cdots,l;\end{equation}
\noindent $(2)_l$
\begin{equation}\label{16}B^{(l)z\bar z}=\sum_{j=0}^{l}B_{(j)}^{z \bar z},\quad j=1,\cdots,l,\, \;\;B_{(0)}^{z\bar z}=B^0,\end{equation}
where the operator-value functions $\lf B^{z\bar{z}}_{(j)}\rc=\lf B^{z\bar{z}}_{(j)}(\xi)\rc:\mathcal{O}_j\rightarrow\mathcal{L}(h_p,h_q)$ is smooth in $\xi\in\mathcal{O}_j$ with
\begin{equation}\label{17}\sup_{\xi\in\mathcal{O}_j}||\lf B_{(j)}^{z\bar z}(\xi)\rc||_{h_p\rightarrow
h_{q}}\leq C(j-1)\, \epsilon_{j-1},\; j=1,...,l,\end{equation}
\begin{equation}\label{18}\sup_{\xi\in\mathcal{O}_j}||\lf \partial_{\xi}(B_{(j)}^{z\bar z}(\xi))\rc||_{{h_p}\rightarrow
h_q}\leq C(j-1)\,\epsilon_{j-1},\; j=1,...,l;\end{equation}
\noindent $(3)_l$ For $ l=0,1,...,m$,
the vector fields $X_{R^{(l)}}:\; D_p(s_l,r_l)\times\mathcal{O}_l\subset\mathcal{P}^{p}\times\mathcal{O}_l\to \mathcal{P}^{q}$ are analytic in $(x,y,z,\bar z)\in D_p(s_l,r_l)$ for fixed $\xi\in \m O_l$ and smooth in $\mathcal{O}_l$ for fixed $(x,y,z,\bar z)\in D_p(s_l,r_l)$, and
 \begin{equation}\label{19} \W \lf X_{R^{(l)}}\rc\W_{q,D_p(s_l,r_l)\times\mathcal{O}_l}\le C(l)\, \e_l,\;
  \W \lf \p_\xi\,X_{R^{(l)}}\rc\W_{q,D_p(s_l,r_l)\times\mathcal{O}_l}\le C(l)\, \e_l; \: \end{equation}
\noindent $(4)_l$ For $ l=0,1,...,m$, the  vector fields $X_{P^{(l)}}:\; D_p(s_l,r_l)\times\m O_l\subset\m P^{p}\times\m O_l\to \m P^{q}$ analytic in $(x,y,z,\bar z)\in D_p(s_l,r_l)$ for fixed $\xi\in \m O_l$ and smooth in $\mathcal{O}_l$ for fixed $(x,y,z,\bar z)\in D_p(s_l,r_l)$, and
     \be\label{22} P^{(l)}=O(|y|^2+|y|||z||_{p}+||z||_{p}^3)\ee and
    \be\label{23}  \W \lf X_{P^{(l)}}\rc\W_{q,D_p(s_l,r_l)\times\mathcal{O}_l}\le C,\;  \W \lf \p_\xi\,X_{P^{(l)}}
    \rc\W_{q,D_p(s_l,r_l)\times\mathcal{O}_l}\le C;\ee
 \noindent $(5)_l$  The Hamiltonian functions $H_0^{(l)}$ and $R^{(l)}$ and $P^{(l)}$ are real when $(x,y)$ are real and $\bar z$ is the complex conjugate of $z$, for $(x,y,z,\bar z;\xi)\in D_p(s_l,r_l)\times \m O_l$.
    %
\noindent

Then there exists a subset $\m O_{m+1}\subset\m O_m$ with
\be\label{25} \text{Meas}\; \m O_{m+1}=(\text{Meas}\;\m O_m)(1-O(K_m^{-C}))\ee
and a symplectic transformation
\be\label{26} \Psi_m:\; D_p(s_{m+1},r_{m+1})\times\m O_{m+1}\to D_p(s_{m},r_{m})\times\m O_{m} \ee such that
\be\label{27} H_{m+1}:=H_m\circ\Psi=N^{(m+1)}+R^{(m+1)}+P^{(m+1)}\ee
satisfies the above conditions $(1)_l$--$(5)_l$ with $l=m+1$.

\end{lem}

\section{Derivation of homological equations}

Step 1: Splitting the perturbation.

In the iterative lemma, the step number $l=0,1,\cdots,m$. Consider the perturbation $R^{(l)}$ with $l=m$. Decompose $R^{(m)}$ into $$R^{(m)}=R^{(2m)}+R^{(3m)}$$
with
%
%
 \begin{eqnarray}\label{}R^{(2m)}&=& R^x(x,\xi)+( R^y(x,\xi),y) \\ && +\langle R^z(x,\xi),z\rangle+\langle R^{\bar z}(x,\xi),\bar z\rangle \\ &&+\langle R^{zz}(x,\xi)z,z\rangle+\langle R^{z\bar z}(x,\xi)z,\bar z\rangle+\langle R^{\bar z\bar z}(x,\xi)\bar z,\bar z\rangle\end{eqnarray}
and $$R^{(3m)}=R^{(m)}-R^{(2m)}=O(|y|^2+|y|||z||_{p}+||z||^3_{p}),$$
where we can assume $\widehat{R^{x}}(0,\xi)$, the $0$-Fourier coefficient of $R^{x}(x,\xi)$, vanishes, since it does not affect the dynamics.

\begin{Lemma}\label{lemma2}Let $u,v\in\{z,\bar z\}$. The terms $R^x(x,\xi)$, $R^y(x,\xi)$, $R^u(x,\xi)$ and $R^{uv}(x,\xi)$ are analytic in $x\in \mathbb T^N_{s_m}$ for fixed $\xi$ and smooth in $\xi\in\m O_m$ for fixed $x$,
%
 and obey the following estimates
\begin{enumerate}
\item[(i)] \be\label{28}|\lfloor R^x\rceil|_{s_{m},\m O_m}\le C(m)\,\e_m,\quad |\lfloor\p_\xi R^x\rceil|_{s_{m},\m O_m}\le C(m)\,\e_m ,\ee
    \be\label{29} |\lfloor R^y\rceil|_{s_{m},\m O_m}\le C(m)\,\e_m,\quad |\lfloor\p_\xi R^y\rceil|_{s_{m},\m O_m}\le C(m)\,\e_m ,\ee

\item[(ii)]
\be \label{30} ||\lfloor R^u\rceil||_{q,s_{m},\m O_m}\le C(m)\,\e_m,\quad  ||\lfloor\p_\xi  R^u\rceil||_{q,s_{m},\m O_m}\le C(m)\,\e_m,\; u\in\{z,\bar z\},\ee

\item[(iii)] for $ u,v\in\{z,\bar z\}$,
    \be\label{31} \sup_{\mathbb T^N_{s_m}\times\m O_m}||\lf R^{uv}(x,\xi)\rc||
    _{h_p\to h_q}\le C(m)\,\e_m,\quad \sup_{\mathbb T^N_{s_m}\times\m O_m}||\lf \p_\xi \,R^{uv}(x,\xi)\rc||_{h_p\to h_q}\le C(m)\,\e_m.\ee

\item[(iv)] The perturbation $R^{(3m)}$ is analytic in $(x,y,z,\bar z)\in D_p(s_m,r_m)$ for fixed $\xi$ and smooth in $\xi\in\m O_m$ for fixed $(x,y,z,\bar z)\in D_p(s_m,r_m)$, and real when $x$ and $y$ are real and $\bar{z}$ is the complex conjugate of $z$, and obeys  the following estimates
\be\label{32} \W \lf X_{R^{(3m)}}\rc \W_{q,D_p(s_m,r_m)\times\m O_m}\le C(m),\quad \W\lf\p_\xi X_{R^{(3m)}}\rc \W_{q,D_p(s_m,r_m)\times\m O_m}\le C(m).\ee
\end{enumerate}
\end{Lemma}
\begin{proof}
 The proofs for \eqref{28}, \eqref{29} and \eqref{30} and \eqref{32} are trivial or simpler than that of \eqref{31}. We give only the proof for the first inequality in (\ref{31}) with $u=z,v=\bar z$. Note that
\be \lf{R^{z\bar z}}(x,\xi)\rc=\left.\left( \p_z\lf\p_{\bar z} R^{(m)}\rc\right)\right|_{z=\bar z=0},\ee
 where $\lf\p_{\bar z} R^{(m)}\rc$ is one entry of the modulus   $\lf X_{R^{(m)}}\rc$.
 By Cauchy's estimate,
\be \sup_{\mathbb T^N_{s_m}\times\m O_m}||\lf R^{z\bar z}(x,\xi)\rc||_{h_p\to h_q}
\le\frac{1}{r_m}\W \lf X_{R^{(m)}}\rc\W_{q,D_p(s_m,r_m)\times\m O_m}\le\frac{1}{r_m}\e_m\le \frac{4}{r_0}\e_m\le C(m)\e_m,\ee
where (\ref{19}) is used in the second inequality.

\end{proof}

Step 2. New form of Hamiltonian $H^{(m)}$. Let \begin{equation}\label{33}\omega^{-}_m:=\widehat{R^y}(0,\xi).\end{equation}
By \eqref{29},
 \begin{equation}\label{34}\sup_{\xi\in\mathcal{O}_m}|\omega^{-}_m(\xi)|\leq C(m)\,\epsilon_m,\quad
 \sup_{\xi\in\mathcal{O}_m}|\partial_{\xi}\omega^{-}_m(\xi)|\leq C(m)\,\epsilon_m.\end{equation}
 Let \begin{equation}\label{35}\tilde{\omega}^{(m+1)}=\omega^{(m)}+\omega^{-}_m,\end{equation}
\begin{equation}\label{36}\tilde{B}^{(m+1)}=B^{(m)}+\widehat{R^{z\bar z}}(0,\xi).\end{equation}
And set
\begin{equation}\label{37}\tilde{H}_0^{(m+1)}=(\tilde{\omega}^{(m+1)},y)
+\langle\Lambda z,\bar z\rangle+
\langle \tilde{B}^{(m+1)}(\xi)z,\bar z\rangle,\end{equation}
\begin{equation}\label{38}R_*^y(x,\xi)=R^y(x,\xi)-\widehat{R^y}(0,\xi),\end{equation}
\begin{equation}\label{39}R_*^{z\bar z}(x,\xi)=R^{z\bar z}(x,\xi)-\widehat{R^{z\bar z}}(0,\xi)\end{equation}
and
\begin{eqnarray}\label{2017-3-8-1x1} R^{(2m)}_* &=& R^x(x,\xi)+( R^y_*(x,\xi),y) \\ \label{2017-3-8-1x2} && +\langle R^z(x,\xi),z\rangle+\langle R^{\bar z}(x,\xi),\bar z\rangle \\ \label{2017-3-8-1x3}&&+\langle R^{zz}(x,\xi)z,z\rangle+\langle R^{z\bar z}_*(x,\xi)z,\bar z\rangle+\langle R^{\bar z\bar z}(x,\xi)\bar z,\bar z\rangle,\end{eqnarray}
where $\widehat{R^x}(0,\xi)=0$, $\widehat{R^y_*}(0,\xi)=0$ and $\widehat{R_{*}^{z\bar z}}(0,\xi)=0$.

Consequently, the Hamiltonian $H^{(m)}$ takes on a new form
\begin{equation}\label{41}H^{(m)}=\tilde{H}^{(m+1)}_0+R_*^{(2m)}+R^{(3m)}+P^{(m)}.\end{equation}
Step 3. Derivation of homological equations. Suppose the to-be-specified Hamiltonian $F$ is of the same form as of $R_*^{(2m)}$:
\begin{eqnarray}\label{2017-3-8-1} F &=& F^x(x,\xi)+( F^y(x,\xi),y) \\ \label{2017-3-8-2}&& +\langle F^z(x,\xi),z\rangle+
\langle F^{\bar z}(x,\xi),\bar z\rangle \\ \label{2017-3-8-3}  && +\langle F^{zz}(x,\xi)z,z\rangle+\langle F^{z\bar z}(x,\xi)z,\bar z\rangle+
\langle F^{\bar z\bar z}(x,\xi)\bar z,\bar z\rangle\end{eqnarray}
with $$\widehat{F^x}(0,\xi)=0,\widehat{F^y}(0,\xi)=0, \widehat{F^{z\bar z}}(0,\xi)=0$$ and it is expected that $$ F=O(\epsilon_m).$$

Let
\begin{equation}\label{43}H^{(m+1)}=H^{(m)}\circ X_F^t|_{t=1},\end{equation}
where $X_F^t$ is the flow of Hamiltonian vector field $X_F$ with Hamiltonian $F$ with the symplectic structure \eqref{2.2}.
By Taylor's formula,
\begin{eqnarray}\label{44}H^{m+1}&=&H^{(m)}+\{H^{(m)},F\}+\int_0^1\int_0^t\{\{H^{(m)},F\},F\}\circ X_F^{\tau}d\tau dt\\
\nonumber&=&\tilde{H}^{(m+1)}_0+R_*^{(2m)}+R^{(3m)}+P^{(m)}+\{\tilde{H}^{(m+1)}_0,F\}+\{R_*^{(2m)},F\}\\
\nonumber &&+\{R^{(3m)}+P^{(m)},F\}+\int_0^1\int_0^t\{\{H^{(m)},F\},F\}\circ X_F^{\tau}d\tau dt,
\end{eqnarray}
where $\{\cdot,\cdot\}$ is the Poisson bracket with the symplectic structure (\ref{2.2}), that is, for any two Hamiltonian functions $f$ and $g$ of $(x,y,z,\bar z)$,
\[ \{f,g\}=\partial_x {f}\cdot\partial_y g-\partial_y{g}\cdot \partial_x f+ {\bf i}\<\p_{z}f,\p_{\bar z} g \>-{\bf i}\<\p_{\bar z} f,\p_{ z}g\>.\]

Set \be \label{170411-5.28-}\tilde{H}_0:=\tilde{H}^{(m+1)}_0,\quad B=\tilde{B}^{(m+1)},\quad \omega=\tilde{\omega}^{(m+1)}.\ee

 By calculation,
\be \label{6.52}\begin{array}{lll} \{\tilde H_0,F\}&=& \p_x\tilde H_0\cdot\p_y F-\p_y\tilde H_0\cdot\p_x F-
{\bf i}\<\p_{\bar z}\tilde H_0,\p_z F \>+{\bf i}\<\p_z \tilde H_0,\p_{\bar z}F\>\\&=&
-\omega\cdot\p_x F^x-(\omega\cdot\p_x F^y,y)-\<\omega\cdot\p_x F^z,z\>-\<\omega\cdot\p_x F^{\bar z},\bar z\>\\&&-\<\omega\cdot\p_x F^{zz}z,z\>-\<\omega\cdot\p_x F^{\bar z\bar z}\bar z,\bar z\>- \<\omega\cdot\p_x F^{ z\bar z} z,\bar z\>\\&&
+{\bf i}\left( \<(\Lambda+B)F^{\bar z},\bar z\>+\<(\Lambda+B)F^{z\bar z}z,\bar z\>+\<(\Lambda+B)F^{\bar z\bar z}\bar z,\bar z\>
+\<F^{\bar z\bar z}(\Lambda+B^T)\bar z  ,\bar z\>\right)\\&&
-{\bf i}\left( \<(\Lambda+B^T)F^{z},z\>+\<F^{z\bar z}(\Lambda+B)z,\bar z\>+\<(\Lambda+B^T)F^{ z z} z, z\>+\<F^{ z z}(\Lambda+B) z, z\>\right),
 \end{array} \ee where $T$ denotes the transpose of matrix.
By Taylor's formula, decompose $\{R^{(3m)}+P^{(m)},F\}$ into the lower order terms $R_+^{(2m)}$  and the higher order terms $R_+^{(3m)}$:
\be \{R^{(3m)}+P^{(m)},F\}=R_+^{(2m)}+R_+^{(3m)}, \ee
where
\begin{eqnarray}\label{6.21}
R_+^{(2m)}&=&R_+^x(x,\xi)+(R_+^y(x,\xi),y)+\\
\label{6.22} &  &\<R_+^z(x,\xi),z\>+\<R_+^{\bar z}(x,\xi),\bar z\>+
\\\label{6.23}
 && \<R_+^{zz}(x,\xi)z,z\>+\<R_+^{\bar z\bar z}(x,\xi)\bar z,\bar z\>+\\ \label{6.24}&&
 \<R_+^{z\bar z}(x,\xi)z,\bar z\>,
\end{eqnarray}
\be R_+^{(3m)}=O(|y|^2+|y|||z||_{p}+||z||_{p}^3).\ee
Noting that \be\label{170408-5.35}\tilde P^{(m)}:=R^{(3m)}+P^{(m)}=O(|y|^2+|y|||z||_{p}+||z||_{p}^3).\ee
By calculation,
we have
\be R^x_+(x,\xi)\equiv 0,\ee
\be \label{7.43}R^y_+(x,\xi)=-\left.\p_y\left(\p_y\tilde P^{(m)}\cdot\p_x F^x-{\bf i}\<\p_{\bar z}\tilde P^{(m)},F^z\>+{\bf i} \<\p_{ z}\tilde P^{(m)},F^{\bar z}\>\right)\right|_{y=0,z=\bar z=0},\ee
\be \label{7.44} R_+^z(x,\xi)=-\left.\p_z\left(\p_y\tilde P^{(m)},\p_x F^x \right)\right|_{y=0,z=\bar z=0} ,\ee
\be\label{7.45} R_+^{\bar z}(x,\xi)=-\left.\p_{\bar z}\left(\p_y\tilde P^{(m)},\p_x F^x \right)\right|_{y=0,z=\bar z=0}, \ee
\be\label{7.46}R_+^{zz}(x,\xi)=\left.\p_z\p_z\left(\p_y\tilde P^{(m)}\cdot\p_x \left(\<F^z,z\>+\<F^{\bar z},\bar z\>\right)+\mi \<\p_z \tilde P^{(m)}, F^{\bar z}\>-\mi \<\p_{\bar z} \tilde P^{(m)}, F^{z}\>\right) \right|_{y=0,z=\bar z=0},\ee
\be \label{7.47}R_+^{\bar z\bar z}(x,\xi)=\left.\p_{\bar z}\p_{\bar z}\left(\p_y\tilde P^{(m)}\cdot\p_x \left(\<F^z,z\>+\<F^{\bar z},\bar z\>\right)+\mi \<\p_z \tilde P^{(m)}, F^{\bar z}\>-\mi \<\p_{\bar z} \tilde P^{(m)}, F^{z}\>\right) \right|_{y=0,z=\bar z=0},\ee
\be\label{7.48} R_+^{ z\bar z}(x,\xi)=\left.\p_{ z}\p_{\bar z}\left(\p_y\tilde P^{(m)}\cdot\p_x \left(\<F^z,z\>+\<F^{\bar z},\bar z\>\right)+\mi  \<\p_z \tilde P^{(m)}, F^{\bar z}\>-\mi \<\p_{\bar z} \tilde P^{(m)}, F^{z}\>\right) \right|_{y=0,z=\bar z=0}.
\ee
Let
\be\label{3-23-pm427} R^{(2m)}_{+,*}:=R_+^{(2m)}-(\widehat{R^y_+}(0,\xi),y)-\<\widehat{R_+^{z\bar z}}(0,\xi)z,\bar z\>.\ee
As a whole, the homological equation obeyed by $F$ reads:
\be\label{6.67}\Gamma\left( \{\tilde H_0,F\}
%
+R_*^{(2m)}+R^{(2m)}_{+,*}\right)=0. \ee
Let
\be\label{yuan7.50} B^{(m+1)}:=\tilde B^{(m+1)}
%
+\widehat{R^{z\bar z}_+}(0,\xi)=B^{(m)}+\widehat{R^{z\bar z}}(0,\xi)
+\widehat{R^{z\bar z}_+}(0,\xi),\ee
\be \label{omegam+1}\omega^{(m+1)}:=\tilde\omega^{(m+1)}+\widehat{R^y_+}(0,\xi)
=\omega^{(m)}+\widehat{R^y}(0,\xi)+ \widehat{R_+^y}(0,\xi).\ee
In view of (\ref{170411-5.28-}), (\ref{6.52}) and (\ref{6.67}), we get
\be H^{(m+1)}=H_0^{(m+1)}+R^{(m+1)}+P^{(m+1)},\ee
where
\be H_0^{(m+1)}
=(\omega^{(m+1)},y)+\sum_{j\in\zd}\lambda_j\,z_j\bar z_j+\<B^{(m+1)}z,\bar z\>,\ee
\begin{eqnarray}\label{3-23-pm41} R^{(m+1)}&=& (1-\Gamma)\left(\{\tilde H_0^{(m+1)},F \} +R_*^{(2m)}+R_{+,*}^{(2m)}\right) \\\label{3-23-pm42} & & +\{R_*^{(2m)},F\}\\\label{170415-1} & &+ \int_0^1\int_0^t\{\{H^{(m)},F\},F\}\circ X^\tau_F\;d\tau dt,\end{eqnarray}
\be\label{3-24-7.56} P^{(m+1)}=R^{(3m)}+P^{(m)}+R_+^{(3m)}=O(|y|^2+|y|||z||_{p}+||z||_{p}^3).\ee
Recall \be \omega=\tilde\omega^{(m+1)}.\ee
Writing (\ref{6.67}) explicitly, we have the following homological equations:
\be\label{7.57} \omega\cdot\p_x F^x=\Gamma R^x,\ee
\be \label{7.58}\omega\cdot\p_x F^y=\Gamma(R^y-\widehat{R^y}(0,\xi)+R_+^y-\widehat{R^y_+}(0,\xi)),\ee
\be\label{7.59} \Gamma(\omega\cdot\p_x F^{z}+{\bf i}((\Lambda+B^T)F^{z}))=-\Gamma (R^{z}+R_+^{z}),\ee
\be\label{7.60} \Gamma(\omega\cdot\p_x F^{\bar z}-{\bf i}((\Lambda+B)F^{\bar z}))=-\Gamma (R^{\bar z}+R_+^{\bar z}),\ee
\be \label{7.61}\Gamma\left(\omega\cdot\p_x F^{zz}+{\bf i}\left((\Lambda+B^T)F^{zz}+F^{zz}(\Lambda+B)\right)\right)=-\Gamma (R^{zz}+R_+^{zz}),\ee
\be\label{7.62} \Gamma\left(\omega\cdot\p_x F^{\bar z\bar z}-{\bf i}\left((\Lambda+B)F^{\bar z\bar z}+F^{\bar z\bar z}(\Lambda+B^T)\right)\right)=\Gamma (R^{\bar z\bar z}+R_+^{\bar z\bar z}),\ee
\be\label{7.62+}  \Gamma\left(\omega\cdot\p_x F^{z\bar z}-{\bf i}\left((\Lambda+B)F^{z\bar z}-F^{ z\bar z}(\Lambda+B)\right)\right)
=\Gamma (R^{ z\bar z}+R_+^{ z\bar z})-\widehat{R^{z\bar z}}(0,\xi)-\widehat{R^{z\bar z}_+}(0,\xi),\ee
where $\Gamma F^x=F^x$, $\Gamma F^y=F^y$, $\Gamma F^z=F^z$, $\Gamma F^{\bar z}=F^{\bar z}$, $\Gamma F^{zz}=F^{zz}$, $\Gamma F^{\bar z\bar z}=F^{\bar z\bar z}$,  $\Gamma F^{zz}=F^{zz}$  and $\widehat{F^{z\bar z}}(0,\xi)=0$.

Finally, we point out that $B$ is self-adjoint in $\ell_2$. By {\bf Assumption D}, $\<B\, z,\bar z\>$ is real when $\bar z$ is regarded as the complex conjugate of $z$. Thus,
\[\<B\, z,\bar z\>=\overline{\<B\, z,\bar z\>}=\<\bar B\bar z,z\>=\<\bar z,\bar B^T\,z\>=\<\bar B^T\, z,\bar z\>,\] where the bar is the complex conjugate.
 It follows that $B$ is self-adjoint in $\ell_2(\mathbb {Z}^d)$, that is,
$B^T=\bar B.$


 \section{Solutions to the homological equations}

We will solve those homological equations \eqref{7.57}-\eqref{7.62+} in the following order:
\[ \eqref{7.57}\Rightarrow \eqref{7.59} \;\text{and}\;\eqref{7.60}\Rightarrow\eqref{7.58}\Rightarrow\eqref{7.61} \;\text{and}\;\eqref{7.62} \;\text{and}\;\eqref{7.62+}.\]
Recall
\be \omega=\tilde \omega^{(m+1)}=\omega^{(m)}+\widehat{R^y}(0,\xi),
\ee
 by \eqref{170411-5.28-} and \eqref{35} and \eqref{33}. In view of \eqref{29} and \eqref{14} and \eqref{15},
 \[\p_\xi \omega=\p_\xi \omega^{0}+O(\e_0).\]
Using Assumption {\bf A}, the map $\omega:\; \m O_m\to \omega(\m O_m)$ is a diffeomorphism between $\m O_m$ and its image $\omega(\m O_m)$ and
\[\left|\text{det}\,\frac{\p \omega}{\p\xi}\right|\ge c_1/2>0. \]
Therefore, we assume $\omega(\xi)\equiv \xi$ without loss of generality. In this section , we always let $K=K_m$.

\begin{lem} \label{lemx6.1}({\bf  Solutions to \eqref{7.57}}) There a subset $\m O_{1m}\subset \m O_m$ with
\[\text{Meas}\; \m O_{1m}=(\text{Meas}\; \m O_m)(1-O(K^{-C})),\]
such that for any $\xi\in \m O_{1m}$, Eq. \eqref{7.57} has a unique solution $F^x(x):\; \mathbb T^N_{s_m^1}\times \m O_{1m}\to \mathbb C^N$ which is analytic in $x\in\mathbb T^N_{s_m^1}$ and smooth  in $\xi\in\m O_{1m}$ and obeys
\be\label{170411-2}|\lfloor F^x\rceil|_{s_{m}^{1},\m O_{1m}}\le C(m) \e_m,\; |\lfloor\p_\xi F^x\rceil|_{s_{m}^{1},\m O_{1m}}\le C(m) \e_m.\ee
\end{lem}
\begin{proof} Note \eqref{28}. The proof is finished by a standard argument in KAM theory. We omit it here.
\end{proof}

\begin{lem} \label{lemx6.2}({\bf  Solutions to \eqref{7.59} and \eqref{7.60}}) Let $u=z$ or $u=\bar z$.
 There a subset $\m O_{2m}\subset \m O_m$ with
\[\text{Meas}\; \m O_{2m}=(\text{Meas}\; \m O_m)(1-O(K^{-C}))\]
such that for any $\xi\in \m O_{2m}$, each equation of  \eqref{7.59} and \eqref{7.60}  has a unique solution $F^u(x)=\m F^u_1(\xi)+\m F^u_2(x,\xi)$, where $\m F^u_1(\xi):\; \m O_{2m}\to h_p$ is smooth in $\xi\in\m O_{2m}$ and obeys
\be\label{170411-w2}\sup_{\xi\in\m O_{2m}}||\lfloor\m F_1^u(\xi)\rceil||_{p}\le K^C \e_m,\; \sup_{\xi\in\m O_{2m}}||\lfloor\p_\xi \m F^u_1\rceil||_{p}\le K^C \e_m\ee
and where $\m F^u_2\; :\; \mathbb T^N_{s_m^2}\times \m O_{2m}\to h_q$ which is analytic in $x\in\mathbb T^N_{s_m^2}$ and smooth in $\xi\in\m O_{2m}$ and obeys
\be\label{170411-z2}||\lfloor\m F^u_2\rceil||_{q,s_m^2,\m O_{2m}}\le K^C \e_m,\;\;\; ||\lfloor\p_\xi \m F^u_2\rceil||_{q,s_m^2,\m O_{2m}}\le K^C \e_m.\ee
\end{lem}
\begin{proof} We point out that for any vector function $f(x,\xi)$, $||f||_{q,s,\m O}=||\lf f\rc||_{q,s,\m O}$.
Recall that $\ti P^{(m)}=R^{(3m)}+P^{(m)}$. By \eqref{32} and \eqref{23} with $l=m$,
\be\label{170411-6.5} \W \lf X_{\ti P^{(m)}}\rc\W_{q,D_p(s_m,r_m)\times\m O_m}\le C(m),\; \W\lf\p_\xi X_{\ti P^{(m)}}\rc\W_{q,D_p(s_m,r_m)\times\m O_m}\le C(m).\ee
Applying \eqref{170411-2} and \eqref{170411-6.5} to \eqref{7.44} and \eqref{7.45}, we get
\be \label{17y1} ||\lfloor R^u_+\rceil||_{q, s^{1}_{m},\m{O}_{1m}}\le C(m) \e_m,\;   ||\lfloor \p_\xi\,R^u_+\rceil||_{q, s^{1}_{m}, \m{O}_{1m}}\le C(m) \e_m.\ee
Set $\m R^u=R^u+R^u_+$. By \eqref{17y1} and \eqref{30},
\be \label{17y2} ||\lfloor\m R^u\rceil||_{q, s^{1}_{m},\m{O}_{1m}}\le C(m) \e_m,\;  \; ||\lfloor\p_\xi\,\m R^u\rceil||_{q, s^{1}_{m},\m{O}_{1m}}\le C(m) \e_m.\ee
%
%
Decompose ${F}^u=\mathcal{F}_1^u(\xi)+\mathcal{F}^u_2(x,\xi)$ with
$$\mathcal{F}_1^u(\xi)= \widehat{F^u}(0,\xi),\quad\mathcal{F}_2^u(x,\xi)=\sum_{0<|k|\leq K}\widehat{F^u}(k)e^{\mi\langle k,x\rangle},$$
and decompose $\m R^u=\m R_1^u(\xi)+\m R_2^u(x,\xi)$ in the same way as done in $F^u$.
 When $k=0$, by \eqref{7.59} and \eqref{7.60},
 \begin{eqnarray*}\widehat{F^u}(0,\xi)&=&\mi \,(\Lambda+B)^{-1}(\widehat{\m R^u}(0,\xi))\\
 &=&{\mi}\,(1+\Lambda^{-1}B)^{-1}(\Lambda^{-1}\widehat{\m R^u}(0,\xi)).\end{eqnarray*}
 Note $||\Lambda^{-1}B||_{h_p\to h_p}=||B||_{h_p\to h_q}\leq\epsilon_0$ and
  \[||\Lambda^{-1}\widehat{\m R^u}(0,\xi)||_{p}\leq  ||\lfloor \m R^u\rceil||_{q, s_{m}^{1}, \m O_{1m}}\le C(m) \e_m,\]
  \[ ||\Lambda^{-1}\p_\xi \widehat{\m R^u}(0,\xi)||_{p}\leq  ||\lfloor \p_\xi \m R^u\rceil||_{q, s_{m}^{1}, \m O_{1m} }\le C(m) \e_m.\]
   So
 \begin{equation}\label{94}||\widehat{F^u}(0,\xi)||_{p}\le ||(1+\Lambda^{-1}\, B)^{-1}||_{h_p\to h_p}||\Lambda^{-1}\widehat{R^u}(0,\xi) ||_{p}\le C(m)\, \e_m,\quad ||\partial_{\xi}\widehat{F^u}(0,\xi)||_{p}\leq C(m)\,\epsilon_m.
\end{equation}
Considering the decompositions  ${F}^u=\mathcal{F}_1^u+\mathcal{F}_2^u$ and $\m R=\m R_1^u+\m R_2^u$, we see that $\m F_2$ obeys
\be \Gamma(\omega\cdot\p_x \m F^u_2\pm{\bf i}((\Lambda+B)\m F^u_2)=\Gamma \m R_2^u,\ee
where
\[\m R_2^u=R^u(x,\xi)+R_+^u(x,\xi)-\widehat{R^u}(0)-\widehat{R^u_+}(0),\; \int_{\mathbb T^N}\m R_2^u(x,\xi)\,d x=0.\]
By Lemma \ref{lem4.1}, one gets that there exists a subset $\m O_{2m}\subset\m {O}_m$ with
\[\operatorname{measure}\, \m O_{2m}=(\text{measure}\;\m O_m)(1-K_m^{-C})\]
 such that
\br\label{1704-6.5}
\begin{array}{ll}
||\lfloor\m F^u_2\rceil||_{q,s_{m}^{2},\m O_{2m}}\le K^C||\lfloor \mathcal{R}_{2}^{u}\rceil||_{q, s_{m},\m O_{1m}}\le K^C \e_m,&\\
 ||\lfloor\p_\xi \m F^u_2\rceil||_{q,s_{m}^{2},\m O_{2m}}\le K^C(||\lfloor \mathcal{R}_{2}^{u}\rceil||_{q, s^{2}_{m},\m O_{1m}}+||\lfloor\p_\xi \m R^u_2\rceil||_{q, s^{2}_{m},\m O_{1m}})\le K^C \e_m .&
\end{array}\er
This completes the proof of this lemma.
%
\end{proof}
\begin{rem} Although $||\widehat{F^u}(0,\xi)||_{p}$ is small, the norm $||\widehat{F^u}(0,\xi)||_{q}$ may be infinite, which is precarious. Fortunately the vector   $\hat{F}^u(0,\xi)$ is independent of variables $(x,y,z,\bar z)$, and $||\widehat{F^u}(k,\xi)||_{q}$
is small for $k\neq 0$. By \eqref{170411-w2} and \eqref{170411-z2}, we have $||\lfloor F^{u}\rceil||_{p, s^{2}_{m},\m{O}_{2m}}\leq K^{c}\varepsilon_{m}$,
$||\lfloor \partial_{\xi}F^{u}\rceil||_{p, s^{2}_{m},\m{O}_{2m}}\leq K^{c}\varepsilon_{m}.$
\end{rem}
\begin{lem} \label{lemx6.3}({\bf  Solution to \eqref{7.58})}
 There a subset $\m O_{3m}\subset \m O_m$ with
\[\text{Meas}\; \m O_{3m}=(\text{Meas}\; \m O_m)(1-O(K^{-C}))\]
such that for any $\xi\in \m O_{3m}$, the equation  \eqref{7.58} has a unique solution $F^y(x,\xi):\; \mathbb T^N_{s_m^3}\times\m O_{3m}\to \mathbb C^N$ is smooth in $\xi\in\m O_{3m}$ and obeys
\be\label{170411-y2}|F^y(x,\xi)|_{s_m^3,\m O_{3m}}\le K^C \e_m,\;
|\p_\xi F^y(x,\xi)|_{s_m^3,\m O_{3m}}\le K^C \e_m.\ee
\end{lem}
\begin{proof} Applying \eqref{170411-2},\eqref{170411-w2},\eqref{170411-z2},\eqref{170411-6.5} to \eqref{7.43}, one gets
\be\label{170411-6.13}\begin{array}{lll}
&&|\lfloor R_+^y\rceil|_{s_m^2,\m O_{2m}}\le C (
|\lfloor \p_x F^x\rceil|_{s_m^2,\m O_{2m}}+||\lfloor F^z\rceil||_{p,s_m^2,\m O_{2m}}+||\lfloor F^{\bar z}\rceil||_{p,s_m^2,\m O_{2m}})
\le K^C \e_m,\\
& &
|\lfloor\p_\xi R_+^y\rceil|_{s_m^2,\m O_{2m}}\le C(\sum_{t=0}^1|\lfloor\p_\xi^t\p_x F^x\rceil|_{s_m^2,\m O_{2m}}
+||\lfloor\p_\xi^t F^z\rceil||_{s_m^2,\m O_{2m}}+||\lfloor \p_\xi^t F^{\bar z}\rceil||_{s_m^2,\m O_{2m}} )
 \le K^C \e_m.\end{array}\ee
By \eqref{170411-6.13} and \eqref{29}, the proof is completed by standard KAM procedure. We omit it.
\end{proof}

\begin{lem} \label{lemx6.4}({\bf  Solutions to \eqref{7.61} \;\text{and}\;\eqref{7.62} \;\text{and}\;\eqref{7.62+}})
Let $u,v\in\{z,\bar z\}$.
 There a subset $\m O_{4m}\subset \m O_m$ with
\[\text{Meas}\; \m O_{4m}=(\text{Meas}\; \m O_m)(1-O(K^{-C}))\]
such that for any $\xi\in \m O_{4m}$, each of equations  \eqref{7.61} \;\text{and}\;\eqref{7.62}
\;\text{and}\;\eqref{7.62+} has a unique solution
$$\lfloor F^{uv}(x,\xi)\rceil:\; \mathbb T^N_{s_m^4}\times\m O_{4m}\to \m L(h_p,h_q)$$ is smooth in $\xi\in\m O_{4m}$
and analytic in $x\in\mathbb T^N_{s_m^4}$
and obeys
\be\label{170412-6.16}\left\{\begin{array}{lll} & &\sup_{\mathbb T^N_{s_m^4}\times\m O_{4m}}||\lf F^{uv}(x,\xi)-\widehat{F^{uv}}(0,\xi)\rc||_{h_p\to h_q}\le K^C \e_m,\\ & &\sup_{\mathbb T^N_{s_m^4}\times\m O_{4m}}
||\lf \p_\xi \left(F^{uv}(x,\xi)-\widehat{F^{uv}}(0,\xi)\right)\rc||_{h_p\to h_q}\le K^C \e_m\end{array}\right.\ee
and
\be\label{170412-6.17}\sup_{\m O_{4m}}||\lf\widehat{F^{uv}}(0,\xi)\rc||_{h_{\ti p}\to h_{\ti p}}\le K^C \e_m,\;\; \sup_{\m O_{4m}}
||\lf\p_\xi \widehat{F^{uv}}(0,\xi)\rc||_{h_{\ti p}\to h_{\ti p}}\le K^C \e_m,\; \ti p\in\{p,q\}.\ee

\end{lem}

\begin{proof}
We firstly give the estimates of $R_+^{uv}$ with $u,v\in\{z,\bar z\}$ which are defined in \eqref{7.46}, \eqref{7.47} and \eqref{7.48}. Without loss of generality, we only give the estimate of $R_+^{zz}$ defined in \eqref{7.46}. According to the decomposition in Lemma \ref{lemx6.2}, write $F^u=\m F_1^u(\xi)+{\m F}_2^u(x,\xi)$ where $u\in\{z,\bar z\}$. A key observation is \[\p_x\, \m F_1^u(\xi)\equiv 0.\] Thus
\be \p_y \ti P^{(m)}\cdot\p_x\left(\p_z(\< F^z,z\>+\<F^{\bar z},\bar z\>)\right)=\p_y \ti P^{(m)}\cdot\p_x \m F_2^z.\ee
By \eqref{170411-z2} and \eqref{170411-6.5},
\[ \p_y \ti P^{(m)}\cdot\p_x \m F_2^z:\; D_p(s_m^2,r_m^2)\times\m O_{2m}\to h_q\]
with
\be \label{1704-6.16}\sup_{D_p(s_m^2,r_m^2)\times\m O_{2m}}|| \lf\p_y \ti P^{(m)}\cdot\p_x \m F_2^z\rc||_{h_q}\le C(m)\,\e_m,\;  \sup_{D_p(s_m^2,r_m^2)\times\m O_{2m}}||\lf\p_\xi( \p_y \ti P^{(m)}\cdot\p_x \m F_2^z)\rc||_{h_q}\le C(m)\,\e_m . \ee
Set\[A_1=\p_z\left( \p_y \ti P^{(m)}\cdot\p_x \m F_2^z\right)\left|_{y=0,z=\bar z=0}.\right.\]
By the Cauchy estimate and using \eqref{1704-6.16},
\be \label{1704-z6.17}\sup_{\mathbb T^N_{s_m^4}\times \m O_{2m}}||\lf A_1\rc||_{h_p\to h_q}\le C(m)\,\e_m,\;
\sup_{\mathbb T^N_{s_m^4}\times \m O_{2m}}||\lf\p_\xi A_1\rc||_{h_p\to h_q}\le C(m)\,\e_m. \ee
By \eqref{170411-6.5} and using the Cauchy estimate, we have that  $\p_z^2\, \ti P^{(m)}:\; h_p\to h_q$ with
\be\label{1704-6.18} \sup_{D_p(s_m,r_m)\times \m O_m}||\lf\p_z^2\, \ti P^{(m)}\rc||_{h_p\to h_q}\le C\frac{1}{r_m-r_m^2}|| \lf X_{\tilde P^{(m)}}\rc||_{h_p\to h_q}\le C(m). \ee
Using \eqref{170411-w2} and \eqref{170411-z2},
\be \label{1704-6.19}
||\lf F^u\rc||_{p,s_{m}^{2},\m O_{2m}}\le K^C\e_m,\;\; ||\lf\p_\xi\,F^u\rc||_{_{p,s_{m}^{2},\m O_{2m}}}\le K^C\e_m ,\; u\in\{z,\bar z\},\ee
where we have used $C(m)K^C\le K^C$ by enlarging the last $C$.
Combining the last two inequalities, we have
\be\label{1704-6.20}
||\lf\left(\p_z^2\ti P^{(m)}\right) F^u\rc||_{q,s_{m}^{2},\m O_{2m}}\le K^C \e_m, \; ||\lf\p_\xi\left(\p_z^2\ti P^{(m)}\right) F^u\rc||_{q,s_{m}^{2},\m O_{2m}}\le K^C \e_m.\ee
Let
\[A_2:=- \p_z\p_z\left(\<\p_z \tilde P^{(m)}, F^{\bar z}\>+\<\p_{\bar z} \tilde P^{(m)}, F^{z}\>\right)\left|_{y=0,z=\bar z=0}.\right.\]
By Cauchy inequality and using \eqref{1704-6.20},
\be \label{1704-6.21}\sup_{\mathbb T^N_{s_m^4}\times \m O_{2m}}||\lf A_2\rc||_{h_p\to h_q}\le C(m) K^C\e_m\le K^C\, \e_m,\; \sup_{\mathbb T^N_{s_m^4}\times \m O_{2m}}||\lf\p_\xi A_2\rc||_{h_p\to h_q}\le\le C(m)K^C\e_m\le K^C \e_m\ee where the last constant $C$ in each inequality is enlarged.
By \eqref{7.46} , we have $R_+^{zz}=A_1+A_2$.
It follows from \eqref{1704-6.21} and \eqref{1704-z6.17} that
\be \label{1704-6.ww21}\sup_{\mathbb T^N_{s_m^4}\times \m O_{2m}}||\lf R^{zz}_+\rc||_{h_p\to h_q}\le K^C\e_m,\; \sup_{\mathbb T^N_{s_m^4}\times \m O_{2m}}||\lf\p_\xi R^{zz}_+\rc||_{h_p\to h_q}\le K^C\e_m,\ee where $C$ is enlarged again.
Similarly, the last inequality holds true for $R^{z\bar z}$ and $R^{\bar z\bar z}$, too.
Noting \eqref{31}, \eqref{1704-6.ww21} and using Lemma \ref{lem5.1}, we finish the proof of this lemma.
\end{proof}

\section{Estimates for new perturbation}

Note $D_p(s_m^i,r_m^i)\subset D_p(s_{m+1},r_{m+1})$ for $i=1,2,3,4$. Let $\m O_{m+1}=\m O\cap\m O_{1m}\cap\m O_{2m}\cap \m O_{3m}\cap \m O_{4m}$. By Lemmas
\ref{lemx6.1}, \ref{lemx6.2}, \ref{lemx6.3} and \ref{lemx6.4},
\[\text{Meas}\; \m O_{m+1}=(\text{Meas}\; \m O_m)(1-O(K_m^{-C})).\]
Recall that $F$ is defined in  \eqref{2017-3-8-1},\eqref{2017-3-8-2} and \eqref{2017-3-8-3}.
The Hamiltonian vector field $X_{F}$ reads
\[\left(
\p_y F(w),- \p_x F(w), {\bf i}\,\p_{\bar z} F(w),-{\bf i}\,\p_{ z}
F(w)\right),\]
where $(w,\xi)=(x,y,z,\bar z;\xi)\in D_{p}(s_{m+1},r_{m+1})\times O_{m+1}$. Note that the tangent space $T_w\, \m P^p=\m P^p$ and $T_w\, \m P^q=\m P^q$.
Denote by $\m V^{\tilde p}$ all functions which map
\[D_{p}(s_{m+1},r_{m+1})\times O_{m+1}\to T_w \m P^{\ti p},\quad \ti p\in\{p,q\}.\]
Recall \eqref{20170729},
\[\W X_F \W_{q,D_p(s,r)\times \m O}=\sqrt{|\p_y\,X_F|^{2}_{ p,s,r,\m O}+
|\p_x\,X_F|^{2}_{p, s,r,\m O}+||\p_{\bar z}\,X_F||^{2}_{ p,  q,s,r,\m O}+||\p_z\,X_F||^{2}_{ p,  q,s,r,\m O}},\]
where $s=s_{m+1},r=r_{m+1}$ and $\m O=\m O_{m+1}$. Since $s_{m+1}, r_{m+1}, \m O_{m+1}$ and the domain $D_p(s_{m+1},r_{m+1})$ are fixed in this section, we write $\W\cdot\W_{q,D_p(s,r)\times \m O}$ as $\W \cdot\W_{q}$. Similarly, we write $\W \cdot\W_{p}:=\W \cdot \W_{p,D_p(s,r)\times \m O}$.
We will denote by $||\cdot||_{\ti p,\ti q}$ the operator norm from $\m V^{\ti p}$ to $\m V^{\ti q}$ where $\ti p,\ti q\in\{p,q\}$.
\begin{lem} \label{out-in-absolute} For $(w,\xi)=(x,y,z,\bar z;\xi)\in D_{p}(s_{m+1},r_{m+1})\times \m O_{m+1}$, we have
\[ ||\lf\m D X_{F(w,\xi)}\rc||_{q,q}\le (N+1) ||\m D\lf X_{F(w(\theta(t)),\xi)}\rc||_{q,q}\] and
  \[ ||\lf\m D X_{F(w,\xi)}\rc||_{p,p}\le (N+1) ||\m D\lf X_{F(w(\theta(t)),\xi)}\rc||_{p,p}.\]
\end{lem}
\begin{proof}
Note
\be \m DX_{F(w,\xi)}=(\p_x\, X_F,\p_y\, X_F,\p_z\, X_F,\p_{\bar z}\, X_F)=\begin{pmatrix} \p_x\p_y  F & \p_y\p_y F&  \p_z\p_y F &-\p_{\bar z}\p_y F\\ -\p_x\p_x F& -\p_y\p_x F& - \p_z\p_x F&\p_{\bar z}\p_x F\\ \mi\,\p_x\p_{\bar z} F & \mi\,\p_y\p_{\bar z} F& \mi\,\p_z\p_{\bar z} F & \mi\,\p_{\bar z}\p_{\bar z} F\\ -\mi\, \p_x\p_{z}F & -\mi\,\p_y \p_{ z}F & -\mi\,\p_z \p_{ z} F & -\mi\,\p_{ z} \p_{\bar z} F
 \end{pmatrix},\ee
 and
 \[\begin{array}{lll} \m D\lf X_{F(w,\xi)}\rc&=&(\p_x\, \lf X_F\rc,\p_y\,\lf X_F\rc,\p_z\,\lf X_F\rc,\p_{\bar z}\, \lf X_F\rc)\\&=&\begin{pmatrix} \p_x\lf \p_y  F\rc & \p_y\lf \p_y F\rc &  \p_z\lf \p_y F\rc &\p_{\bar z}\lf \p_y F\rc

 \\ \p_x\lf \p_x F\rc & \p_y\lf \p_x F\rc& \p_z\lf \p_x F\rc &\p_{\bar z}\lf \p_x F\rc

 \\ \mi\,\p_x\lf \p_{\bar z} F\rc  &  \mi\,\p_y\lf \p_{\bar z} F\rc &\mi\, \p_z\lf \p_{\bar z} F \rc& \mi\,\p_{\bar z} \lf \p_{\bar z} F\rc

 \\ -\mi\, \p_x\lf \p_{ z}F\rc &-\mi\, \p_y \lf \p_{ z}F \rc&-\mi\,\p_z\lf \p_{ z} F \rc&-\mi\, \p_{ z}\lf \p_{\bar z} F\rc
 \end{pmatrix}
 \\&=& \begin{pmatrix} 1 & 0 &  0 &0

 \\ 0&1& 0 &0

 \\ 0 &   &\mi & 0

 \\ 0 &0 &0&-\mi
 \end{pmatrix}

 \begin{pmatrix} \p_x\lf \p_y  F\rc & \p_y\lf \p_y F\rc &  \p_z\lf \p_y F\rc &\p_{\bar z}\lf \p_y F\rc

 \\ \p_x\lf \p_x F\rc & \p_y\lf \p_x F\rc& \p_z\lf \p_x F\rc &\p_{\bar z}\lf \p_x F\rc

 \\ \p_x\lf \p_{\bar z} F\rc  &  \p_y\lf \p_{\bar z} F\rc & \p_z\lf \p_{\bar z} F \rc&\p_{\bar z} \lf \p_{\bar z} F\rc

 \\  \p_x\lf \p_{ z}F\rc &\p_y \lf \p_{ z}F \rc&\p_z\lf \p_{ z} F \rc& \p_{ z}\lf \p_{\bar z} F\rc
 \end{pmatrix}\\&:=&\Xi_0\,\Theta.
  \end{array}\]
Note that $\Xi_0$ is a unitary operator from $\m V^q$ to $\m V^q$. Thus
\[|| \m D\lf X_{F(w,\xi)}\rc||_{q,q}=||\Theta ||_{q,q}.\]
If we write
\[X_F=\sum_{k,\alpha,\beta,\gamma} C_{k\,\alpha\,\beta\,\gamma}\, e^{\mi(k,x)}\, y^\gamma\,z^\alpha\,\bar z^\beta,\] we find that the indices $\alpha,\, \beta,\, \gamma$
of $ z,\, \bar z,\,y$ are non-negative integer vector, except for the index $k$ of $x$. So we have
 \[ \lf \m D X_{F(w,\xi)}\rc=\begin{pmatrix}
 \lf \p_x\lf \p_y  F\rc\rc & \p_y\lf \p_y F\rc &  \p_z\lf \p_y F\rc &\p_{\bar z}\lf \p_y F\rc
 \\ \lf\p_x\lf \p_x F\rc\rc & \p_y\lf \p_x F\rc& \p_z\lf \p_x F\rc &\p_{\bar z}\lf \p_x F\rc
 \\ \lf \p_x\lf \p_{\bar z} F\rc \rc &  \p_y\lf \p_{\bar z} F\rc & \p_z\lf \p_{\bar z} F \rc& \p_{\bar z}\lf \p_{\bar z} F\rc
 \\ \lf \p_x\lf \p_{ z}F\rc\rc & \p_y \lf \p_{ z}F \rc&\p_z\lf \p_{ z} F \rc& \p_{ \bar z}\lf \p_{ z} F\rc
 \end{pmatrix}.\]
 We see that all entries of $\m D\lf X_{F(w,\xi)}\rc$ and $\lf \m D X_{F(w,\xi)}\rc$ are the same except those entries in the first columns of them.
 Thus we partition $\lf \m D X_{F(w,\xi)}\rc:=\digamma_1+\digamma_2,$ where
 \[\digamma_1=\begin{pmatrix}
 \lf \p_x\lf \p_y  F\rc\rc & 0 & 0 &0
 \\ \lf\p_x\lf \p_x F\rc\rc &0& 0 &0
 \\ \lf \p_x\lf \p_{\bar z} F\rc \rc & 0& 0& 0
 \\ \lf \p_x\lf \p_{ z}F\rc\rc &0&0& 0
 \end{pmatrix}=\sum_{j=1}^{N}\oplus \begin{pmatrix}\lf
  \p_{x_j}\lf \p_y  F\rc\rc & 0 & 0 &0
 \\  \lf\p_{x_j}\lf \p_x F\rc\rc&0& 0 &0
 \\  \lf \p_{x_j}\lf \p_{\bar z} F\rc\rc  & 0& 0& 0
 \\ \lf \p_{x_j}\lf \p_{ z}F\rc\rc &0&0& 0
 \end{pmatrix}:=\sum_{j=1}^{N}\oplus  \digamma_{1j},\] and
 \[\digamma_2=\begin{pmatrix}
 0 & \p_y\lf \p_y F\rc &  \p_z\lf \p_y F\rc &\p_{\bar z}\lf \p_y F\rc
 \\ 0& \p_y\lf \p_x F\rc& \p_z\lf \p_x F\rc &\p_{\bar z}\lf \p_x F\rc
 \\ 0 &  \p_y\lf \p_{\bar z} F\rc & \p_z\lf \p_{\bar z} F \rc& \p_{\bar z}\lf \p_{\bar z} F\rc
 \\0 & \p_y \lf \p_{ z}F \rc&\p_z\lf \p_{ z} F \rc& \p_{ \bar z}\lf \p_{ z} F\rc
 \end{pmatrix}.\]
 We see that
 \[\Theta=\tilde \digamma_1+\digamma_2,\]
 where
 \[\tilde \digamma_1=\begin{pmatrix}
  \p_x\lf \p_y  F\rc & 0 & 0 &0
 \\ \p_x\lf \p_x F\rc&0& 0 &0
 \\  \p_x\lf \p_{\bar z} F\rc  & 0& 0& 0
 \\ \p_x\lf \p_{ z}F\rc &0&0& 0
 \end{pmatrix}=\sum_{j=1}^{N}\oplus \begin{pmatrix}
  \p_{x_j}\lf \p_y  F\rc & 0 & 0 &0
 \\  \p_{x_j}\lf \p_x F\rc&0& 0 &0
 \\   \p_{x_j}\lf \p_{\bar z} F\rc  & 0& 0& 0
 \\  \p_{x_j}\lf \p_{ z}F\rc &0&0& 0
 \end{pmatrix}:=\sum_{j=1}^{N}\oplus  \ti\digamma_{1j}.\]

  By Lemma \ref{splitlemma}, we have
 \[||\digamma_2 ||_{q,q}\le ||\Theta ||_{q,q}.\]
 In order to compute $|| \digamma_1||_{q,q}$, we write
 \[\p_u\, F=\sum_k e^{\mi (k,x)}\sum_{\alpha\,\beta,\gamma}\, C_{k,\alpha,\beta,\gamma;u}\, y^\gamma\, z^\alpha\, \bar z^\beta,\quad u\in\{x,y,z,\bar z\}.\]
 Then
 \[\p_{x_j}\,\lf \p_u\, F \rc = \sum_k (\mi\, k_j)e^{\mi (k,x)}\sum_{\alpha\,\beta,\gamma}\,\left| C_{k,\alpha,\beta,\gamma;u}\right|\, y^\gamma\, z^\alpha\, \bar z^\beta:=\sum_k (\mi\, k_j)e^{\mi (k,x)}\sum_{\alpha\,\beta,\gamma}\,g_{ku},\quad u\in\{x,y,z,\bar z\} \]
 and
 \[\lf \p_{x_j}\,\lf \p_u\, F \rc\rc = \sum_k | k_j|\, e^{\mi (k,x)}\sum_{\alpha\,\beta,\gamma}\,
 \left| C_{k,\alpha,\beta,\gamma;u}\right|\, y^\gamma\, z^\alpha\,\bar{z}^{\beta}=\sum_k | k|\, e^{\mi (k,x)}g_{ku}. \]
 It follows that
\[\tilde \digamma_{1j}=\sum_k (\mi\, k_j) \, e^{\mi\, (k,x)}\,\begin{pmatrix}
  g_{ky} & 0 & 0 &0
 \\ g_{kx}&0& 0 &0
 \\ g_{k\bar z}  & 0& 0& 0
 \\ g_{kz} &0&0& 0
 \end{pmatrix},\; \digamma_{1j}=\sum_k |k_j| \, e^{\mi\, (k,x)}\,\begin{pmatrix}
  g_{ky} & 0 & 0 &0
 \\ g_{kx}&0& 0 &0
 \\ g_{k\bar z}  & 0& 0& 0
 \\ g_{kz} &0&0& 0
 \end{pmatrix} .\]
Let $g_{1j}=(g_{1j}^y,\, g_{1j}^x\, g_{1j}^{\bar z},\, g_{1j}^z)^T$  ($\ti g_{1j}$, respectively) be the first column of the operator $\digamma_{1j}$ ( $\ti\digamma_{1j}$ , respectively). Note that all columns of $\digamma_{1j}$ ( $\ti\digamma_{1j}$ , respectively) are zero vectors except  $g_{1j}$  ($\ti g_{1j}$, respectively). Thus
\[|| \ti \digamma_{1j}||_{q,q}=\W \ti  g_{1j}\W_q,\;\; ||  \digamma_{1j}||_{q,q}=\W g_{1j}\W_q.\]
By the definition of $\W\cdot\W_{q}$,
\[\begin{array}{lll}\W g_{1g}\W_q^2&=&|g_{1j}^y|_{p,s,r,\m O}^2+|g_{1j}^x|_{p,s,r,\m O}^2+||g_{1j}^{\bar z}||_{p,q,s,r,\m O}^2+|| g_{1j}^z||_{p,q,s,r,\m O}^2\\&=&\sup_{\xi\in\m O,|y|\le r^2,||z||_{p}\le r,||\bar z||_{p}\le r} \sum_k |k_j|^2\left( |g_{ky}|^2+|g_{kx}|^2+||g_{k\bar z}||_q^2+||g_{kz}||_q^2 \right)\\&
=& \W \ti g_{1g}\W_q^2.\end{array}\]
It follows
\[||\digamma_{1j}||_{q,q}=||\tilde \digamma_{1j}||_{q,q}.\]
 Note that $\ti \digamma_{1j}$ is a column of $\Theta$. By  Lemma \ref{splitlemma}, we have
\[||\digamma_{1j}||_{q,q}=||\tilde \digamma_{1j}||_{q,q}\le ||\Theta||_{q,q}.\]
 Consequently,
 \[ ||\lf\m D X_{F(w,\xi)}\rc||_{q,q}\le (N+1) ||\m D\lf X_{F(w,\xi)}\rc||_{q,q}.\]
In the proof above, by change $q$ by $p$
 we have
  \[ ||\lf\m D X_{F(w,\xi)}\rc||_{p,p}\le (N+1) ||\m D\lf X_{F(w,\xi)}\rc||_{p,p}.\]
This completes the proof.
\end{proof}

%

\begin{lem} \label{lemx7.1}
For $(w,\xi)=(x,y,z,\bar z;\xi)\in D_{p}(s_{m+1},r_{m+1})\times \m O_{m+1}$,
\be\label{1704-xx7.5} ||\lf\m D X_{F(w,\xi)}\rc||_{p,p}\le K^C_m\e_m,\; ||\lf\m D X_{F(w,\xi)}\rc||_{q,q}\le K^C_m\e_m\ee and
\be\label{1704-y7.6} ||\lf\p_\xi \m D X_{F(w,\xi)}\rc||_{p,p}\le K^C_m\e_m,\; ||\lf\p_\xi \m D X_{F(w,\xi)}\rc||_{q,q}\le K^C_m\e_m.\ee
\end{lem}
\begin{proof}
 Let $$f^u=\<F^z(x,\xi),z\>+\<F^{\bar z}(x,\xi),\bar z\>,$$
\[f_0^u:=\< \widehat{F^z}(0),z\>+\< \widehat{F^{\bar z}}(0),\bar z\>\]
and
\[\ti f^u:=\<F^z(x,\xi),z\>+\<F^{\bar z}(x,\xi),\bar z\>-f_0^u=f^u-f_0^u.\]
By Lemma \ref{lemx6.2} (i.e., \eqref{170411-z2}),
\be ||\lf X_{\ti f^u}\rc||_{p,q, s_{m}^{2},r_{m}^{2},\m O_{2m}}\le K^C_m\;\e_m.  \ee
Note $s_m^4-s_{m+1}>\frac{1}{C m^2},\,r_m^4-r_{m+1}>\frac{1}{C m^2}$.
By the Cauchy's  estimate,
\be \sup_{D_p(s_{m+1},r_{m+1})\times \m O_{m+1}}|| \m D \lf X_{\ti f^u}\rc||_{p,q}\le K^C_m\;\e_m .\ee
%
Another important observation is   $$\m D X_{f_0^u}\equiv 0.$$  So
\be\label{1704-7.10} \sup_{D_p(s_{m+1},r_{m+1})\times \m O_{m+1}}||\m D \lf X_{ f^u}\rc||_{p,q}\le K^C_m\e_m.\ee
Note $h_q\subset h_p$. By \eqref{1704-7.10},
we have
\be\sup_{D_p(s_{m+1},r_{m+1})\times \m O_{m+1}} ||  \m D \lf X_{f^u}\rc||_{p,p}\le K_m^C \e_m,\quad \sup_{D_p(s_{m+1},r_{m+1})\times \m O_{m+1}}||\m D \lf X_{f^u}\rc||_{q,q}\le K_m^C \e_m.\ee
Let
\[f^{uu}:=\langle (F^{zz}(x,\xi)-\widehat{F^{zz}}(0,\xi))z,z\rangle+\langle F^{z\bar z}(x,\xi)z,\bar z\rangle+
\langle (F^{\bar z\bar z}(x,\xi)-\widehat{F^{\bar z\bar z}}(0,\xi))\bar z,\bar z\rangle.\] Note $\widehat{F^{z\bar z}}(0,\xi)=0$.
By \eqref{170412-6.16} in Lemma \ref{lemx6.4} and Cauchy's inequality,
\[\sup_{D_p(s_{m+1},r_{m+1})\times \m O_{m+1}}|| \m D \lf X_{f^{uu}}\rc||_{p,q}\le K_m^C \e_m.\]
Thus
\be \sup_{D_p(s_{m+1},r_{m+1})\times \m O_{m+1}}|| \m D \lf X_{f^{uu}}\rc||_{p,p}\le K_m^C \e_m,
\quad \sup_{D_p(s_{m+1},r_{m+1})\times \m O_{m+1}}||\m D \lf X_{f^{uu}}\rc||_{q,q}\le K_m^C \e_m.\ee
Let
\[f^{uu}(0):=\langle \widehat{F^{zz}}(0,\xi) z,z\rangle+
\langle \widehat{F^{\bar z\bar z}}(0,\xi)\bar z,\bar z\rangle.\]
Then
\[\m D X_{f^{uu}(0)}=\begin{pmatrix}0 & 0& 0 &0\\0 & 0& 0 &0\\0& 0& 0 & \mi\, \widehat{F^{\bar z\bar z}}(0,\xi)\\  0&0& -\mi\, \widehat{F^{ z z}}(0,\xi) & 0
 \end{pmatrix},\]
which is independent of $(x,y,z,\bar z)$. By  \eqref{170412-6.17} in Lemma \ref{lemx6.4}, we have
\be \sup_{D_p(s_{m+1},r_{m+1})\times \m O_{m+1}}||  \m D \lf X_{f^{uu}(0)}\rc||_{p,p}\le K_m^C \e_m,\quad \sup_{D_p(s_{m+1},r_{m+1})\times \m O_{m+1}}||\m D \lf X_{f^{uu}(0)}\rc||_{q,q}\le K_m^C \e_m.\ee

Let
\[f^{x,y}:=F^x(x,\xi)+( F^y(x,\xi),y).\]
By Lemmas \ref{lemx6.1}, \ref{lemx6.3}  and Cauchy's inequality,
\be \sup_{D_p(s_{m+1},r_{m+1})\times \m O_{m+1}}|| \m D \lf X_{f^{x,y}}\rc||_{p,p}\le K_m^C \e_m,\quad
\sup_{D_p(s_{m+1},r_{m+1})\times \m O_{m+1}}||\m D \lf  X_{f^{x,y}}\rc||_{q,q}\le K_m^C \e_m.\ee
Consequently,
\be \sup_{D_p(s_{m+1},r_{m+1})\times \m O_{m+1}}|| \m D \lf X_{F}\rc||_{p,p}\le K_m^C \e_m,\quad
\sup_{D_p(s_{m+1},r_{m+1})\times \m O_{m+1}}||\m D \lf  X_{F}\rc||_{q,q}\le K_m^C \e_m.\ee
 We proves \eqref{1704-xx7.5} by using Lemma \ref{out-in-absolute}. By applying the method as above to $\p_\xi X_F$, we can finish the proof of \eqref{1704-y7.6}. We omit the detail. This completes the proof of this lemma.
\end{proof}

\ \

Let $X^t_{F(w(t))}=w(t)$ be the flow of the Hamiltonian vector field $X_F$. Then
\[w(t)-w(0)=\int_0^t\, X^s_{F(w(s),\xi)}\, d\, s.\]
Thus, for $t\in[0,1]$,
\be\begin{array}{ll}\W w(t)-w(0)\W_p&\le \int_0^t\W X_{F(w(s),\xi)}-X_{F(w(0),\xi)}\W_p+\int_0^t\W X_{F(w(0),\xi)}\W_p\,d\, s
\\
&\le K_m\e_m+\int^t_0||\m D X_{F(w(\theta(s)),\xi)}||_{p,p}\W w(s)-w(0)\W_p\, d \, s,
\end{array}\ee
where $\theta(s)\in [0,1]$ is a function of $s$, and $\m D$ is the tangent map of $X_F$.

 By Lemma \ref{lemx7.1} and Gronwall's inequality,
\be\label{190414-7.5} \W w(t)-w(0)\W_p\le K_m^C\,\e_m\exp\left(\int_0^1||\m D X_{F(w(\theta(s)),\xi)}||_{p,p}\, d\, s\right)\le K_m^C\,\e_m,\ee
where $(w(0),\xi)\in D_p(s_{m+1},r_{m+1})\times \m O_{m+1}$ and we have used that $||\m D X_{F(w(\theta(s)),\xi)}||_{p,p}\le ||\m \lf \m D X_{F(w(\theta(s)),\xi)}\rc||_{p,p}$.

Note that $r_{m+1}+K_m^C \e_m<r_m,\; s_{m+1}+K_m^C\e_m<s_m$.
The inequality \eqref{190414-7.5} implies that there does exist the solution $w(t)$ for $t\in[0,1]$ and that $\Psi_m:=w(1)$ obeys
\[\W\Psi_m(w,\xi)-w\W_p\le K_m^C\e_m,\quad (w,\xi)\in D_p(s_{m+1},r_{m+1})\times \m O_{m+1}\]
and
\[\Psi_m:\;\; D_p(s_{m+1},r_{m+1})\times \m O_{m+1}\to D_p(s_{m},r_{m})\times \m O_{m}.\]
This actually proves \eqref{26}.

\ \
We are now in position to estimate \eqref{3-23-pm42}. Recall that $R_*^{(2m)}$ is defined in \eqref{2017-3-8-1x1}, \eqref{2017-3-8-1x2} and \eqref{2017-3-8-1x3}.
By Lemma \ref{lemma2},
\be\label{1704-yuan2} \W \lf X_{R^{(2m)}_*}\rc\W_{q,D_p(s_m,r_m)}\le C(m) \e_m,\;\; \W \lf \p_\xi\,X_{R^{(2m)}_*}\W_{q,D_p(s_m,r_m)}\le C(m) \e_m.\ee
Combining Lemma \ref{lemx7.1} and \eqref{1704-yuan2}, we have
\be\label{1704-7.16} \begin{array}{ll}
\W \lf X_{\{R_*^{(2m)},F\}}\rc\W_{q,D_p(s_{m+1},r_{m+1})}&\le ||\lf\m D\,
 X_{R_*^{(2m)}}\rc||_{p,q}\W X_F\W_p+||\lf\m D X_F \rc||_{q,q}\W\lf X_{R_*^{(2m)}}\rc\W_q \\ &\le 2\e_m\, K_m^C\e_m\le \e_{m+1}\end{array} \ee
and
\br\label{1704-7.17}
&&\W \lf\p_\xi\, X_{\{R_*^{(2m)},F\}}\W_{q,D_p(s_{m+1},r_{m+1})}\nonumber\\
&\le &
\sup(||\lf\p_\xi\,\m D\, X_{R_*^{(2m)}}\rc||_{p,q}) \W\lf X_F\rc\W_{p}+\sup(||\lf\m D\, X_{R^{(2m)}_*}\rc||_{p,q})\W\lf X_F\rc\W_{p}\\
&&+
\sup(||\lf\p_\xi\,\m D\, X_{F}\rc||_{q,q}) \W \lf X_{R_*^{(2m)}}\rc\W_{q}+\sup(||\lf\m D\, X_F\rc||_{q,q})\W \lf \p_\xi\, X_{R_*^{(2m)}}\rc\W_{q }\nonumber\\
&\le & 4\e_m\, K_m^C\e_m\le \e_{m+1},\nonumber
  \er
where the ``$\sup$" runs over $D_p(s_{m+1},r_{m+1})\times \m O_{m+1}$.

Recalling $\lambda_j\approx |j|^{-\kappa}$ and Assumption {\bf E}  and using \eqref{16}-\eqref{23}, we have
\be\label{1704-yuan3} \W \lf X_{H^{(m)}}\rc\W_{q,D_p(s_m,r_m)}\le C(m),\;\; \W \lf\p_\xi\, X_{H^{(m)}}\rc\W_{q,D_p(s_m,r_m)}\le C(m).\ee
Combining Lemma \ref{lemx7.1} and \eqref{1704-yuan3}, we have
\be \begin{array}{ll}\W \lf X_{\{H^{(m)},F\}}\rc\W_{q,D_p(s_{m}^5,r_{m}^5)}&\le \sup||\lf\m D\, X_{H^{(m)}}\rc||_{p,q}\W X_F\W_p+||\lf\m D X_F\rc||_{q,q}\W\lf X_{H^{(m)}}\rc\W_q \\ &\le C(m)\, K_m^C\,\e_m\end{array} \ee
and
\br &&\W \lf \p_\xi\,X_{\{H^{(m)},F\}}\rc\W_{q,D_p(s_{m}^5,r_{m}^5)}\nonumber\\
&\le &\sup(||\lf\p_\xi\,\m D\, X_{H^{(m)}}\rc||_{p,q}) \W \lf X_F\rc\W_{p}+\sup(||\lf\m D\, X_{H^{(m)}}\rc||_{p,q})\W \p_\xi\, X_F\W_{p}\nonumber\\
&&+
\sup(||\lf\p_\xi\,\m D\, X_{F}\rc||_{q,q}) \W \lf X_{H^{(m)}}\rc\W_{q}+\sup(||\lf \m D\, X_F\rc||_{q,q})\W \lf\p_\xi\,X_{H^{(m)}}\rc\W_{q} \\
&\le &C(m)\, K_m^C\,\e_m,\nonumber \er
 where the ``$\sup$" runs over $D_p(s_m^5,r_m^5)\times \m O_{m+1}$.
Repeating the last procedure, we have
\be\label{1704-7.21} \begin{array}{ll}\W\lf X_{\{\{H^{(m)},F\},F\}}\rc\W_{q,D_p(s_{m+1},r_{m+1})}&\le ||\lf\m D\, X_{\{H^{(m)},F\}}\rc||_{p,q}\W\lf X_F\rc\W_p+||\lf\m D X_F\rc||_{q,q}\W\lf X_{\{H^{(m)},F\}}\rc\W_q \\ &\le C(m)\, K_m^C\,\e_m^2<\e_{m+1}\end{array} \ee
and
\be\label{1704-7.22} \begin{array}{ll}&\W \lf\p_\xi\,X_{\{\{H^{(m)},F\},F\}}\rc\W_{q,D_p(s_{m+1},r_{m+1})}\\
&\le\sup(||\lf\p_\xi\,\m D\, X_{\{H^{(m)},F\}}\rc||_{p,q}) \W \lf X_F\rc\W_{p}+\sup(||\lf\m D\, X_{\{H^{(m)},F\}}\rc||_{p,q})\W \lf\p_\xi\,X_F\rc\W_{p}\\
&\quad+\sup(||\lf\p_\xi\,\m D\, X_{F}\rc||_{q,q}) \W\lf X_{\{H^{(m)},F\}}\rc\W_{q}+\sup(||\lf\m D\, X_F\rc||_{q,q})\W \lf \p_\xi\,X_{\{H^{(m)},F\}}\rc\W_{q}
 \\ &\le C(m)\, K_m^C\,\e_m^2\le \e_{m+1},\end{array} \ee
 where the ``$\sup$" runs over $D_p(s_{m+1},r_{m+1})\times \m O_{m+1}$.
By applying \eqref{1704-7.21}, \eqref{1704-7.22} and  \eqref{190414-7.5} to \eqref{170415-1}, we have
\be \label{yyy1}\W\lf X_{\text{\eqref{170415-1}}}\rc\W_{q,D_p(s_{m+1},r_{m+1})\times\m O_{m+1}}\le C(m+1)\, \e_{m+1}, \; \W \lf \p_\xi\,X_{\text{\eqref{170415-1}}}\rc
\W_{q,D_p(s_{m+1},r_{m+1})\times\m O_{m+1}}\le C(m+1)\, \e_{m+1}. \ee
By the definition of $\Gamma_{K_m}$, we have immediately
\be \label{yyy3}\W \lf X_{\text{\eqref{3-23-pm41}}}\rc\W_{q,D_p(s_{m+1},r_{m+1})\times\m O_{m+1}}\le C(m+1)\,\e_{m+1}, \; \W \lf\p_\xi\,X_{\text{\eqref{3-23-pm41}}}\rc
\W_{q,D_p(s_{m+1},r_{m+1})\times\m O_{m+1}}\le C(m+1)\,\e_{m+1}. \ee
By \eqref{1704-7.16}, \eqref{1704-7.17}, \eqref{yyy1} and \eqref{yyy3}, we have
\be \label{yyy4}\W \lf X_{R^{(m+1)}}\rc\W_{q,D_p(s_{m+1},r_{m+1})\times\m O_{m+1}}\le C(m+1)\,\e_{m+1}, \; \W \lf\p_\xi\,X_{R^{(m+1)}}\rc
\W_{q,D_p(s_{m+1},r_{m+1})\times\m O_{m+1}}\le C(m+1)\,\e_{m+1}. \ee
This proves $(3)_l$ with $l=m+1$ in Lemma \ref{lemma1}.

By a similar way,
\be \label{yyy4}\W \lf X_{P^{(m+1)}}\rc\W_{q,D_p(s_{m+1},r_{m+1})\times\m O_{m+1}}\le C(m+1), \; \W\lf\p_\xi\, X_{P^{(m+1)}}\rc
\W_{q,D_p(s_{m+1},r_{m+1})\times\m O_{m+1}}\le C(m+1). \ee
This proves $(4)_l$ with $l=m+1$ in Lemma \ref{lemma1}.

Finally, let us verify that $H_0^{(m+1)},\, R^{(m+1)}$ and $P^{(m+1)}$ are real when $(x,y)$ are real and $\bar z$ is the complex conjugate of $z$ for $(x,y,z,\bar z;\xi)\in D_p(s_{m+1},r_{m+1})\times \m O_{m+1}$. Note that $R^{2m}$ is real when $(x,y)$ are real and $\bar z$ is the complex conjugate of $z$. We assume $x$ is real until the end of this section. It follows that $R^{z\bar z}(x,\xi)$ is real symmetric operator and $R^x(x,\xi)$, $R^y(x,\xi)$ are real vectors, and
\[\overline{R^{z}(x,\xi)}=R^{\bar z}(x,\xi),\; \overline{R^{zz}(x,\xi)}=R^{\bar z\bar z}(x,\xi).\]
By \eqref{7.57}, $F^x(x,\xi)$ is real. By \eqref{7.44} and \eqref{7.45},  $\overline{R^z_+}=R^{\bar z}_+$. Note $B^T=\bar B$. Perform the complex conjugate ``bar" in both sides of \eqref{7.60} and \eqref{7.61}, we get
\be\label{real-1} \overline{F^{z}(x,\xi)}=F^{\bar z}(x,\xi).\ee
Furthermore, by \eqref{7.43} and \eqref{real-1}, we have that $R_+^y(x,\xi)$ is  real. It follows from \eqref{7.58} that $F^y(x,\xi)$ is real. By \eqref{real-1} and \eqref{7.46}, \eqref{7.47} and \eqref{7.48}, we get
\be \label{real-2} \overline{R_+^{zz}(x,\xi)}=R_+^{\bar z\bar z}(x,\xi), \;  \overline{R_+^{z\bar z}(x,\xi)}=R_+^{z\bar z}(x,\xi). \ee
Noting \eqref{real-2} and performing the complex conjugate ``bar" in both sides of \eqref{7.61}, \eqref{7.62} and \eqref{7.62+}, we get
\be \label{real-3} \overline{F^{zz}(x,\xi)}=F^{\bar z\bar z}(x,\xi), \;  \overline{F^{z\bar z}(x,\xi)}=F^{z\bar z}(x,\xi). \ee
Consequently, $F$ defined by \eqref{2017-3-8-1}, \eqref{2017-3-8-2} and \eqref{2017-3-8-3} is real when $(x,y)$ are real and $\bar z$ is the complex conjugate of $z$. Arbitrarily take two Hamiltonian functions $F_1$ and $F_2$ defined by \eqref{2017-3-8-1}, \eqref{2017-3-8-2} and \eqref{2017-3-8-3}. And assume $F_1$ and $F_2$ are real when $(x,y)$ are real and $\bar z$ is the complex conjugate of $z$. Then it is easy to prove that the Poisson bracket $\{F_1,F_2\}$ is also real.
It follows furthermore that $H_0^{(m+1)},\, R^{(m+1)}$ and $P^{(m+1)}$ are real when $(x,y)$ are real and $\bar z$ is the complex conjugate of $z$ for $(x,y,z,\bar z;\xi)\in D_p(s_{m+1},r_{m+1})\times \m O_{m+1}$. This proves $(5)_l$ with $l=m+1$ in Lemma \ref{lemma1}.

By \eqref{30}, \eqref{omegam+1} and \eqref{170411-6.13}, we have
$$\sup_{\xi\in \m O_{m+1}}|\partial_{\xi}^{t}(\widehat{R^{y}}(0, \xi)+\widehat{R_{+}^{y}}(0,\xi))|\leq K^C \varepsilon_{m}\le C(m)\e_m,\;\;t=0, 1.$$
This proves $(1)_l$ with $l=m+1$ in Lemma \ref{lemma1}.

By \eqref{31}, \eqref{yuan7.50} and \eqref{1704-6.ww21}, we obtain
$$||\partial_{\xi}^{t}(\widehat{R^{z\bar{z}}}(0,\xi)+\widehat{R_{+}^{z\bar{z}}}(0,\xi))||_{h_{p}\rightarrow h_{q}}\leq K^{C}\varepsilon_{m}\le C(m)\e_m.$$
This proves $(2)_l$ with $l=m+1$ in Lemma \ref{lemma1}.

\ \

Up to now, we have verified all the assumptions are fulfilled for $l=m+1$. Thus the proof of the iterative lemma is complete. $\qed$

\section{Proof of the main Theorems}

{\bf Proof of Theorem \ref{theorem2}. }

Let \[\m O=\cap_{j=0}^\infty\m O_j,\]
\[\omega(\xi)=\omega^0+\sum_{m=1}^\infty\omega_m,\]
\[B^{z\bar z}=B_{0}^{z\bar z}+\sum_{m=1}^\infty B_{m}^{z\bar z}.\]
And note
\[D_p(s_0,r_0)\supset D_p(s_0,r_0)\supset\cdots D_p(s_m,r_m)\supset\cdots \supset D_p(s_0/2,r_0/2). \]
Using the iterative lemma and let $m\to \infty$, the proof of the main theorem  is finished in the standard procedure in KAM theory. Here we omit the detail.

\vskip8pt

\noindent{\it Proof of Corollary \ref{Cor1}.} Since $\mathcal T_0$ is invariant for the flow of $X_{H^\infty}$, the pull-back tori $\Phi(\mathcal T_0)$ is invariant for the flow of the original Hamiltonian vector field $\Phi^*X_{H^\infty}=X_H$, too.

\vskip8pt

\noindent{\it Proof of Corollary \ref{Cor2}.} Consider a linear Hamiltonian system with Hamiltonian
\[H=\sum_{j\in\mathbb Z^d} \lambda_j \, z_j\, \bar z_j+\l B^0(\xi)\, z,\bar z\r+\l R^{zz}(x,\xi)z,z\r+\l R^{z\bar z}(x,\xi)z,\bar z\r+\l R^{\bar z\bar z}(x,\xi)\bar z,\bar z\r\]
where $(x,y,z,\bar z)\in D_p(s_0/2,r_0/2)$ and $\xi\in \m O$.
Using Theorem \ref{theorem2}, there is a symplectic transformation such that $H$ is reduced to
\[H^\infty=\l(\Lambda+ B^\infty(\xi))\, z,\bar z\r\]
by digging a subset of $\m O$ with small Lebesgue  measure.
Multiplying $\bar z$ in both sides of the equation
\[\dot z={\bf i}\,(\Lambda+B^\infty(\xi))z,\quad z\in h_p,\] one gets
\[\sum_{j\in\zd}\dot z_j\bar z_j={\bf i} \;\<(\Lambda +B^\infty(\xi))z,\bar z\>.\]
Recall that $\<(\Lambda +B^\infty(\xi))z,\bar z\>$ is real for any $z\in h_p$, when $\bar z$ is the complex conjugate of $z$. Write $z=(|j|^p w_j:\; j\in\mathbb Z^d)$ where $w=(w_j:\; j\in\mathbb Z^d)\in \ell_2$. So one gets
\[\frac12\frac{d}{dt}\sum_{j\in\mathbb Z^d}|j|^{2p}|w_j|^2 =\frac12\frac{d}{dt}\sum_{j\in\mathbb Z^d}|z_j(t)|^2=\sum_{j\in\zd}\Re \dot z_j\bar z_j=0,\; \text{while}\; z(t)\in h_p.\]
It follows that $||z(t)||_{p}\equiv$ constant. This completes the proof.$\qed$

\noindent{\bf Proof of Theorem \ref{theorem2+2}. }  In order to distinguish the normal frequencies $\lambda_j$'s in Theorem \ref{theorem2} and the $\lambda_j$'s in Theorem \ref{theorem2+2}, we denote by $\lambda_j^\prime$ the normal frequencies $\lambda_j$'s in Theorem \ref{theorem2+2}. By Assumption {\bf B$^\star$}, we can write
\[\lambda_j^\prime=\varpi+\lambda_j,\quad j\in\mathbb Z^d\]
where $\lambda_j$'s satisfy Assumption {\bf B}. Thus the homological equation \eqref{61} should be replaced by
 \begin{equation}\label{61++} (\varpi-( k,\omega)+\Lambda+{B})\widehat{{F}}(k)=\widehat{{R}}(k),\quad
 \forall \;k\in\mathbb{Z}^N,\;0<|k|\leq K,\end{equation}
where $\Lambda=\text{diag}\; (\lambda_j:\;j\in\mathbb Z^d)$ and $\varpi-( k,\omega)=(\varpi-( k,\omega))E$ with $E$ being an identity from $h_p$ to $h_p$. (at this time, $p=q$.) Moreover, \eqref{62} should be replaced by
 \begin{equation}\label{62++}|\varpi-(k,\omega(\xi))|\geq K^{-c_{21}},\quad 0<|k|\leq K,
   \xi\in\mathcal{O}\setminus\mathcal{O}_1.\end{equation}
And \eqref{17-3-12-1} should be replaced by
\be\label{17-3-12-1++} (( k,\omega) \pm(\varpi+\Lambda+B))\widehat{{F}}(k)\pm \widehat{{F}}(k)(\varpi+\Lambda+\breve B)=\widehat{{R}}(k),\quad k\in\mathbb Z^N,\;|k|\le K.\ee
Moreover,  \eqref{103*}
 should be replaced by
\begin{equation}\label{103*}|(( k,\omega)\pm\varpi)\pm\lambda_j|\geq\frac12 K^{-c},\quad \mbox{for}\;\; k\in\mathbb Z^N,\; 0<|k|\leq K
,\;j\in\mathbb{Z}^d,\;\xi\in\mathcal{O}\setminus\mathcal{O}_4.\end{equation}
After finishing the replacements as above, we can construct Lemmas \ref{lem4.1}  and \ref{lem5.1} for theorem \ref{theorem2+2}. The remaining proof for theorem \ref{theorem2+2} is similar to that of Theorem \ref{theorem2}. We omit it here.



\section{ Application to Benjamin-Bona-Mahony (BBM) equation}

Benjamin-Bona-Mahony (BBM) equation:
  was studied in 1972 by Benjamin, Bona, and Mahony\cite{Benjamin-Bona-Mahony} as an improvement of the Korteweg-de Vries equation (KdV equation) for modeling long surface gravity waves of small amplitude - propagating uni-directionally in $1+1$ dimensions. Also see  \cite{Avrin-Goldstein}, \cite{Goldstein-Wichnoski}
and \cite{Biler-Piotr} for the related topics.
Consider BBM equation subject to periodic boundary condition
\begin{equation}\label{w10}
u_t-u_{xxt}+u_x+uu_x=0,\quad u(t,0)=u(t,T).
\end{equation}
This equation can be written as a Hamiltonian system
\be\label{w1}
u_t=-(1- \p_{xx})^{-1}\p_x \nabla_u H(u)
\ee
with Hamiltonian function
\be\label{w11}
H(u)=\frac{1}{2}\int_0^{T}u^2 dx+\frac{1}{6}\int_0^{T}u^3 dx,
\ee
and the symplectic structure $-(1- \p_{xx})^{-1}\p_x$ and the working space
$$ u\in\mathcal{H}_0^{p_0}=\{u\in\mathcal{H}^{p_0}(\mathbb{T}:\,\mathbb R): \int_0^{T}u dx=0\},$$ where $\mathcal{H}^{p_0}(\mathbb{T})$ is the usual Sobolev space with some $p_0>0$.

Let $\bar{\mathbb{Z}}=\mathbb{Z}\setminus\{0\}$ and set $\tau=\frac{2\pi}{T}.$
Denote by $h_{p_0}$ the discretization of $\m H_0^{p_0}$, i.e.
$$h_{p_0}:=\{z=(z_j\in\mathbb C):\;  \|z\|_{p}^2=\sum_{j\in\bar{\mathbb{Z}}}|z_j|^2 j^{2p_0}<\infty\}.$$
Make Fourier transform $\m F:\; u\mapsto z=(z_j\in\mathbb C:\; j\in\bar{\mathbb Z})$ by
\be\label{170510-y4}
u=\sum_{j\in\bar{\mathbb{Z}}}\delta_j\,z_j\phi_j,\quad ~~\phi_j=\frac{1}{\sqrt{T}}e^{\mi \tau j\cdot x},\quad \delta_j=\sqrt{\frac{\tau\,|j|}{1+\tau^2\,j^2}}.
\ee Note that $\bar z_j=z_{-j}$ if and only if $u\in\mathbb R$.
Then $\m F$ is isometry from $\m H^{p_0}$ to $h_{p_0+\frac12}=h_p$ (with $p_0+1/2=p$) and \eqref{w1} is changed into a Hamiltonian system  with its symplectic structure $-\mi\, \sum_{j\ge 1}dz_j\wedge\, d z_{-j} $:
\be\label{9.5}
 \mi \, \dot{z}_j=\frac{\p\, H}{\p \bar z_j},\quad -\mi \, \dot{\bar z}_j=\frac{\p\, H}{\p  z_j},\quad \bar z_j=z_{-j},
\ee
where
\be\label{w6}
H(z,\bar z)=\sum_{j\ge 1}\lambda_j  z_j\,  z_{-j}+\frac{1}{6\sqrt{T}}\sum_{j+k+l=0,\, j, k,l\in\bar{\mathbb Z}}\delta_j\delta_k\delta_l z_j\,z_k\, z_l:=H_0+R
\ee
and
\be
\lambda_j=\frac{\tau j}{1+\tau^2j^2},\quad |\lambda_j|\approx |j|^{-\kappa},\quad \kappa=1.
\ee
Let \be {\Lambda}=\text{diag}\;(\lambda_j:\,j\in\bar{\mathbb Z})\ee
and
\be \label{9.9}R=\sum_{j+k+l=0, j, k,l\in\bar{\mathbb Z}} R_{jkl} z_j\,z_k\, z_l ,\quad G_{jkl}=\frac{\delta_j\delta_k\delta_l }{6\sqrt T}. \ee
By $X_R$ denote the Hamiltonian vector filed of $R$ with symplectic structure $-\mi\, \sum_{j\ge 1}dz_j\wedge\, d z_{-j} $:
\[X_R=(\mi \sigma_j \frac{\p R}{\p z_{-j}}:\, j\in\bar{\mathbb Z)},\; \sigma_j=\mbox{sign}\; j.\]

\begin{lem} \label{ylem9.1} The function $H(z,\bar z)$ is real when $\bar z_j$ is the complex conjugate of $z_j$ for each $j\in\bar{\mathbb Z}$. The  Hamiltonian vector field $X_R$ of the perturbation $R$ is analytic from $h_p$ to $h_q$ with $q=p+\kappa$ and $\kappa=1$. Moreover,
\be ||\lf X_R\rc||_{q}\le C\, ||z||_p^2.\ee
\end{lem}
\begin{proof} Observe that $\delta_j\approx|j|^{-1/2}$ and note that
\[\frac{\p R}{\p z_{-j}}=3\delta_j\sum_{k+l=j}\delta_k\delta_l\,z_k\,z_l.\]
Hence,
$$\left|\frac{\partial R}{\partial z_{-j}}\right|\leq 3\delta_{j}\sum_{k+l=j}\delta_k\delta_l\,|z_k|\,|z_l|=3\delta_j \cdot(w*w)_j,$$
where $w=(w_j:j\in\bar{\mathbb Z})$ with $w_j=|\delta_j \, z_j|$ and $w*w$ is the convolution of $w$ and $w$.
It follows that $$||\lf X_R\rc||_{q}\leq||w*w||_{q-\frac12}\le C||w||_{q-\frac12}^2=C ||w||^2_{p+\frac12}=C||z||_{p}^2.$$ The remaining statements are obvious.
\end{proof}

Fix two integers $N$ and $\ti N$ with $0<N\le \ti N$. Let
 \[J=\{1\leq j_1<j_t<\cdots<j_N\le \ti N:
 j_t\in\mathbb{N} ~\text{for}~1\leq t\leq N\}.\] Split $z=(z_j)_{j\in\bar{\mathbb Z}}=(\tilde{z}, \hat{z})$ with $\tilde{z}=(z_{j_1},\cdots,z_{j_N},\, z_{-j_1},\cdots,z_{-j_N})$ and  $\hat{z}=z \ominus \tilde z.$
\begin{lem}\label{normal-form} Assume   the number $\tau=2\pi/T$ is transcendental.
 There exists a real analytic symplectic coordinate transformation $\Phi$ which maps the neighborhood of the origin of $h_p$ to $h_p$ such that   the Hamiltonian $H$ defined by \eqref{w6} is changed into a partial Birkhoff normal form up to  order four. More precisely,
 \be
 H\circ \Phi=H_0+\bar{G}+\hat{G}+\tilde{R},
 \ee
 where
 \be
 \bar{G}=\sum_{k,l\geq 1,\, \{k,l\}\cap J\neq \emptyset}\bar{G}_{kl}|z_k|^2|z_l|^2
 \ee
with
 \be\label{y9.12}
 \bar{G}_{kl}=\left\{
 \begin{array}{ll}
 -\frac{1}{T}\frac{\tau^2kl}{[\tau^2(k^2+kl+l^2)+3][\tau^2(k^2-kl+l^2)+3]},& k\neq l,\\
 \frac{1}{12T}\frac{1}{\tau^2k^2+1},& \text{otherwise},
 \end{array}\right.
 \ee
 and \be ||\lfloor X_{\hat{G}}\rceil||_q=O(\|\hat{z}\|_{p}^3),\quad \|\lfloor X_{\tilde{R}}\rceil\|_{q}=O(\|z\|_{p}^4).\ee
\end{lem}

In order to prove the last lemma, we need the following lemmas.

\begin{lem}\label{ylem9.3} Assume   the number $\tau=2\pi/T$ is transcendental.
(1) For any $j,k,l\in\bar{\mathbb Z}$ with $j+k+l=0$, one has
\[\lambda_j+\lambda_k+\lambda_l\neq 0;\]
 (2) For any $j,k,l,m\in\bar{\mathbb Z}$ with $j+k+l+m=0$ and $(j+k)(j+l)(j+m)\neq 0$, one has
\[\lambda_j+\lambda_k+\lambda_l+\lambda_m\neq 0.\]
\end{lem}
\begin{proof} Recall $\lambda_j=\frac{\tau j}{1+\tau^2\,j^2}$ for all $j\in\bar{\mathbb Z}$. In view of $j+k+l=0$, by calculation,
\[\lambda_j+\lambda_k+\lambda_l=-\left( \frac{j\,k\,l\,
      {\tau }^3\,
      \left( 3 +
        \left( k^2 + k\,l +
           l^2 \right) \,
         {\tau }^2 \right) }
      {\left( 1 +
        j^2\,{\tau }^2 \right)
        \,\left( 1 +
        k^2\,{\tau }^2 \right)
        \,\left( 1 +
        l^2\,{\tau }^2 \right)
        } \right).\]
 That is,
 \be \label{9.15}\lambda_j+\lambda_k+\lambda_l= \pm \delta_j^2\,\delta_k^2\,\delta_l^2(3+\tau^2(k^2+kl+l^2)).\ee
Noting $k^2+kl+l^2>0$. It follows  $\lambda_j+\lambda_k+\lambda_l>0$.

     \ \

     When $j+k+l+m=0$, by calculation one has that

     \[\begin{array}{ll}&\tau^{-3}(\left( 1 + j^2\,{\tau }^2 \right) \,
  \left( 1 + k^2\,{\tau }^2 \right) \,
  \left( 1 + l^2\,{\tau }^2 \right) \,
  \left( 1 + m^2\,{\tau }^2 \right))(\lambda_j+\lambda_k+\lambda_l+\lambda_m)\\&=
  3\,\left( k + l \right) \,\left( k + m \right) \,
   \left( l + m \right)  +
  \left( k + l \right) \,\left( k + m \right) \,
   \left( l + m \right) \,
   \left( k^2 + l^2 + l\,m + m^2 +
     k\,\left( l + m \right)  \right) \,{\tau }^2\\&\;\; +
   k\,l\,\left( k + l \right) \,m\,
   \left( k + m \right) \,\left( l + m \right) \,
   \left( k + l + m \right) \,{\tau }^4.
  \end{array}\]
  Since $kl\left( k + l \right) \,\left( k + m \right) \,
   \left( l + m \right)\neq 0$ and that $\tau$ is transcendental, $\lambda_j+\lambda_k+\lambda_l+\lambda_m\neq 0$.

\end{proof}

\begin{lem} Assume   the number $\tau=2\pi/T$ is transcendental.  Set
 \[\Delta_3=\{(j,k,l)\in\bar{\mathbb Z}^3:\;j+k+l= 0,\;\min\{|j|,|k|,|l|\}\le \ti N\},\]  \[\Delta_4=\{(j,k,l,m)\in\bar{\mathbb Z}^4:\;j+k+l+m=0,\; (j+k)(j+l)(j+m)\neq 0,\;\min\{|j|,|k|,|l|,|m|\}\le \ti N\}.\]
 Then there exists a constant $C=C(\ti N)>0$ depending on only $\ti N$ such that


\be \label{y9.15}\inf_{(j,k,l)\in \Delta_3}|\lambda_j+\lambda_k+\lambda_l|\ge C(\ti N);\ee
and
\be\label{y9.16}\inf_{(j,k,l,m)\in \Delta_4}|\lambda_j+\lambda_k+\lambda_l+\lambda_m|\ge C(\ti N).\ee
\end{lem}
\begin{proof}
We give the proof only for \eqref{y9.15} with  $|j|\le \ti N$ and $|k|\le \ti N$ and $l\in\bar{\mathbb Z}$.  The others are similar.
Clearly, there is a constant $C_2=C_2(\ti N)>0$ such that
\[\inf_{\Delta_2}|\lambda_j+\lambda_k|\ge C_2(\ti N),\;\; \text{here}\; \Delta_2=\{(j,k)\in\bar{\mathbb Z}^2:\;j+k\neq 0,\;\max\{|j|,|k|\}\le \ti N\}.\]
 Note $\lim_{|l|\to \infty}\lambda_l=0$. So there is a constant $M=M(\ti N)>0$ depending on $\ti N$ such that $|\lambda_l|\le C_2(\ti N)/4$ when $|l|\ge M(\ti N)$. Thus, for $|l|\ge M(\ti N)$,
 \[\inf_{(j,k,l)\in \Delta_3}|\lambda_j+\lambda_k+\lambda_l|>\inf_{(j,k)\in \Delta_2}|\lambda_j+\lambda_k|-\frac{C_2(\ti N)}{4}\ge \frac{C_2(\ti N)}{2}.\]
 When $|j|\le \ti N$, $|k|\le \ti N$ and $|l|\le M(\ti N)$, using Lemma \ref{ylem9.3}, there is a constant $\tilde C_3(\ti N)=\tilde C_3(\ti N,M(\ti N))>0$ such that
\[\inf_{(j,k,l)\in \Delta_3}|\lambda_j+\lambda_k+\lambda_l|\ge\inf_{j+k+l=0,|j|+|k|+|l|\le 2\ti N+M(\ti N)}|\lambda_j+\lambda_k+\lambda_l|\ge \tilde C_3(\ti N).\]
 The inequality \eqref{y9.15} with  $|j|\le \ti  N$, $|k|\le \ti N$ and $l\in\bar{\mathbb Z}$ is proved by letting $$C(\ti N)=\min\{C_2(\ti N)/2,\tilde C_3(\ti N)\}.$$

\end{proof}

{\bf Proof of Lemma \ref{ylem9.1}.} Let $\Psi^{(3)}=X_{F^{(3)}}^1$ be the time-1-map of the flow of the
hamiltonian vector field $X_{F^{(3)}}$ given by the hamiltonian
\[F^{(3)}=\sum_{j,k,l\in\bar{\mathbb Z},j+k+l=0}F^{(3)}_{jkl}z_j\,z_k\,z_l\]
with coefficients
\[\mi F^{(3)}_{jkl}=\left\{\begin{array}{ll}  \frac{R_{jkl}}{\lambda_j+\lambda_k+\lambda_l},& j+k+l=0, \\ 0, & \text{otherwise.}\; \end{array} \right.\]
Recall that $\lambda_j+\lambda_k+\lambda_l\neq 0$, when $j+k+l=0$. So $F^{(3)}_{jkl}$ is well defined in the last formula.

By \eqref{9.15} and \eqref{9.9}  we have that the $j$-th entry of the vector field $X_{F^{(3)}}$ is
\[\pm\,\mi\,\frac{\p F^{(3)}}{\partial z_{-j}}=\frac{1}{2\sqrt{T}\delta_j}\sum_{k+l=j}\frac{z_k\,z_l}{\delta_k\,\delta_l\left(\tau^2(k^2+l^2+k\,l)+3\right)}.\]
Define a vector field $\lf X_{F^{(3)}}\rc$ with its $j$-th entry being
\[\lf  X_{F^{(3)}}\rc_j=\frac{1}{2\sqrt{T}\delta_j}\sum_{k+l=j}\frac{|z_k|\,|z_l|}{\delta_k\,\delta_l\left(\tau^2(k^2+l^2+k\,l)+3\right)}.\]
Observe that
\[k^2+k\,l+l^2=(k+l)^2-k\,l\ge \frac12 (k+l)^2=\frac12\, j^2. \]
So
\[\lf X_{F^{(3)}}\rc_j\le C |j|^{-3/2}\sum_{k+l=j}\sqrt{|k|\,|l|}|z_k|\,|z_l|.\]
Let $w=(w_k:\; k\in\mathbb Z^d)$ with $w_k=\sqrt{|k|}|z_k|$.
Then, noting $q=p+1$,
\[||\lf X_{F^{(3)}}\rc||_q\le C ||w * w||_{p-\frac12}\le C ||w||_{p-\frac12}^2=C||z||_p^2. \]
Therefore $X_{F^{(3)}}$  is a real analytic vector field which maps a small neighborhood of the origin in $h_p$ to $h_q$.  And  hence $\Psi^{(3)}$ is a real analytic, symplectic change of coordinates
defined at least in a neighborhood of the origin in $h_p$.

Expanding at $t = 0$ and using Taylor's formula we have
\be\label{9.18}\begin{array}{ll} H\circ \Psi^{(3)}&=H\circ X_{F^{(3)}}^t\left.\right|_{t=1}\\
&=H+\{H,F^{(3)}\}+\int_0^1(1-t)\{\{H,F^{(3)}\},F^{(3)}\}\circ  X_{F^{(3)}}^t\, dt\\
&=H_0+R+\{H_0,F^{(3)}\}\\& \quad +\{R,F^{(3)}\}+\int_0^1(1-t)\{\{H,F^{(3)}\},F^{(3)}\}\circ  X_{F^{(3)}}^t\, dt.
   \end{array}\ee
By calculation, we obtain
\be\label{9.19} R+\{H_0,F^{(3)}\}=0.\ee
Again using Taylor's formula in the last term in \eqref{9.18} and noting \eqref{9.19}, we have
\be\label{9.20}  H\circ \Psi^{(3)}=H_0+\frac12\{R,F^{(3)}\}+\frac12\int_0^1(1-t^2)\{\{R,F^{(3)}\},F^{(3)}\}\circ  X_{F^{(3)}}^t\, dt.\ee
By direct calculation, we have
\[\begin{array}{ll}R^4:&=\frac12\{R,F^{(3)}\} \\ & =\frac12\sum_{j\in\bar{\mathbb Z}}\frac{\p\, R}{\p\, z_j}(-\mi\sigma_j)\frac{\p \, F^{(3)}}{\p\,z_{-j}}\\ &=\frac{1}{8T}\sum_{k+l+m+n=0}^\prime \frac{\delta_k\,\delta_l}{\sigma_m\,\sigma_n\delta_m\,\delta_n(\tau^2(m^2+mn+n^2)+3)}
z_k\, z_l\,z_m\,z_n
\\& := \sum_{k+l+m+n=0}^\prime R_{klmn}z_k\, z_l\, z_m\, z_n,\end{array}\]
where $\sum_{k,l,m,n}^\prime $ runs over the set $\{(k,l,m,n)\in \, \bar{\mathbb Z}^4\}$ and $\sigma_k$ is the sign of $k$, i.e. $\sigma_k=1$ if $k>0$ and $\sigma_k=-1$ if $k<0$.
So the $j$-th entry of the vector field $X_{R^4}$ is
\[\begin{array}{ll}\pm \mi\, \frac{\partial R^4}{\partial_{z_{-j}}}=&\delta_j
\sum_{l+m+n=j}^\prime \frac{\,\delta_l}{4T\,\sigma_m\,\sigma_n\delta_m\,\delta_n(\tau^2(m^2+mn+n^2)+3)}
 z_l\,z_m\,z_n\\
 &+
\sum_{k+l+m=j}^\prime \frac{\delta_k\,\delta_l}{4T\,\sigma_m\,\sigma_j\delta_m\,\delta_j^2
(\tau^2(m^2-mj+j^2)+3)}
 z_k\,z_l\,z_m. \end{array}\]
Note that
\[m^2+mn+n^2+3\ge |mn|, \;\;m,n\in\bar{\mathbb Z}.\]
Then
\[|4T\,\sigma_m\,\sigma_n\delta_m\,\delta_n(\tau^2(m^2+mn+n^2)+3)|\ge \frac1 C \sqrt{|m||n|}\] and
\[|4T\,\sigma_m\,\sigma_j\delta_m\,\delta_j^2
(\tau^2(m^2-mj+j^2)+3)|\ge \frac{1}{C}\sqrt{|m|}.\]
Thus,
\[\left| \pm \mi\, \frac{\partial R^4}{\partial_{-z_j}}\right|\le C\frac{1}{\sqrt{|j|}} \sum^{\prime}_{l+m+n=-j}\frac{|z_l|}{\sqrt{|l|}}\,\frac{|z_m|}{\sqrt{|m|}}\,
\frac{|z_n|}{\sqrt{|n|}}.\]
Let $w=(\frac{|z_l|}{\sqrt{|l|}}:\, l\in\bar{\mathbb Z})$. Then
\be \label{9.21}|| \lf X_{R^4}\rc||_q=||\lf X_{R^4}\rc||_{p+1}\le ||w* w* w||_{p+\frac12}\le ||w||^3_{p+\frac12}=||z||_p^3.\ee
Let
\[\Xi_{\le}=\{(k,l,m,n)\in \, \bar{\mathbb Z}^4:\; k+l+m+n=0,\; (k+l)(k+m)(l+m)= 0,\;J\cap\{|k|,|l|,|m|,|n|\}\neq\emptyset\},\]
\[\Xi_{>}=\{\{(k,l,m,n)\in \, \bar{\mathbb Z}^4:\; k+l+m+n=0,\;J\cap\{|k|,|l|,|m|,|n|\}=\emptyset\}.\]
Then we can write
\begin{eqnarray}R^4&=&\sum_{k,l,m,n}^\prime R_{klmn}z_k\,z_l\,z_m\,z_n\\&=&\sum_{\Xi_{\le}}R_{klmn}z_k\,z_l\,z_m\,z_n+
\sum_{\Delta_4}R_{klmn}z_k\,z_l\,z_m\,z_n+\sum_{\Xi_{>}}R_{klmn}z_k\,z_l\,z_m\,z_n\\ &:= &
R^4_{(1)}+R^4_{(2)}+R^4_{(3)}.\end{eqnarray}
By direct calculation,
\[R^4_{(1)}=\sum_{\Xi_{\le}}R_{klmn}=\sum_{k\ge 1,\, l\ge 1}\bar G_{kl}|z_k|^2|z_l|^2,\]
where $\bar G_{kl}$ is defined by \eqref{y9.12}. Let $X^t_{F^{(4)}}$ be the flow of the Hamiltonian vector field with Hamiltonian $F^{(4)}:$
\[F^{(4)}=\sum_{k,l,m,n\in\bar{\mathbb Z}}F_{klmn}^4z_k\,z_l\,z_m\,z_n,\]
where
\[\mi \,F_{klmn}^4=\left\{\begin{array}{ll} \frac{1}{\lambda_k+\lambda_l+\lambda_m+\lambda_n} R_{klmn},&(k,l,m,n)\in\Delta_4, \\ 0, & \text{otherwise}.\end{array}  \right.\]
By \eqref{y9.16}, we have
\be\label{17-5-9.1}|\mi \,F_{klmn}^4|\le C |R_{klmn}|.\ee
By \eqref{17-5-9.1} and \eqref{9.21},
\be ||\lf X_{F^{(4)}}\rc||_q\le C||\lf X_{R^4}\rc||_q\le C||z||_p^3. \ee
Let $\Psi^{(4)}=X_{F^{(4)}}^t\left.\right|_{t=1}$. Thus,
\begin{eqnarray}\label{17-5-9y.1} H\circ\Psi^{(3)}\circ\Psi^{(4)}&=&H_0+\{H_0,F^{(4)}\}+R^4\\ \label{17-5-9y.2}
&&+\int_0^1(1-t)\,\{\{H_0,F^{(4)}\},F^{(4)}\}\circ\, X^t_{F^{(4)}}\, dt\\\label{17-5-9y.3} & &+\int_0^1\{R^4,F^{(4)}\}\circ\, X^t_{F^{(4)}}\, dt\\\label{17-5-9y.4}  &&+ \left(\int_0^1(1-t^2)\{\{R,F^{(3)}\},F^{(3)}\}\circ  X_{F^{(3)}}^t\, dt\right)\circ\Psi^{(4)}.\end{eqnarray}
By direct calculation,
\[\{H_0,F^{(4)}\}+R^4=R^4_{(1)}+R_{(3)}^4=\bar G+O(||\hat z||_p^4).\]
Note that all of $\lfloor X_R\rceil, \; \lfloor X_{R^3}\rceil,\; \lfloor X_{R^4}\rceil,\; \lfloor X_{F^{(3)}}\rceil,\; \lfloor X_{F^{(4)}}\rceil$ are in  $h_q$.
It is easy to verify that
\eqref{17-5-9y.2}$=O(||z||_p^6)$ and \eqref{17-5-9y.3}$=O(||z||_p^6)$ and \eqref{17-5-9y.4}$=O(||z||_p^5)$.
The proof of Lemma \ref{normal-form} is finished.$\qed$

\ \

Restrict $|\tilde z|\le \e_0^{1/4}$ and $||\hat z||_p\le  \e_0^{1/3}$.
Then
\be \label{170510-y3}||\lf X_{\hat G+\tilde R}\rc||_q\le C(|\tilde z|^4+||\hat z||_p^3)\le C\, \e_0.\ee

Introduce action-angle variables $(y,x)$ by
\be\left\{\begin{array}{ll} z_{j_k}=\sqrt{\zeta_k+y_k}\, e^{-\mi x_k}, &
 z_{-j_k}=\sqrt{\zeta_k+y_k}\, e^{\mi x_k},\; k=1,...,n,\; j_k\in J,\\ & \\ z_j=z_j,& z_{-j}=z_{-j},\; j\in\bar{\mathbb Z}\setminus J,  \end{array}\right. \ee
where $\zeta=(\zeta_1,...,\zeta_n)\in\mathbb R_+^n$ and $\e_0^{1/2}<|\zeta|\le 2\e_0^{1/2}$. Then
\[H_0=\sum_{1\le k\le N}\frac{\tau j_k}{1+\tau^2\,j_k^2}(\zeta_k+y_k)+\sum_{\bar{\mathbb Z}\setminus J}\frac{\tau^2 j}{1+\tau^2\, j^2}z_j\,z_{-j}\] and
\begin{eqnarray}\bar G &=&\sum_{1\le k,l\le N}\bar{G}_{j_kl_k}(\zeta_k+y_k)(\zeta_l+y_l)+
\sum_{1\le k\le N,\, l\in\bar{\mathbb Z}\setminus J}\bar G_{j_k l}(\zeta_k+y_k)z_l\,z_{-l}\\ & &+
\sum_{k\in\bar{\mathbb Z}\setminus J,\,1\le l\le N}\bar G_{kj_l}(\zeta_l+y_l)z_k z_{-k}.
\end{eqnarray}
Thus, up to a constant depending on $\zeta$,
\[H_0+\bar G=\sum_{1\le j\le N}\omega_j^0(\zeta)\, y_j+\sum_{j\in\bar{\mathbb Z}\setminus J}\Omega_j^0(\zeta)
z_j\,z_{-j},\]
where
\be\omega^0(\zeta)=\lambda^{(N)}+\m B\, \zeta,\quad \Omega^0(\zeta)=\lambda^\infty+S\, \zeta \ee
with
\be\label{1707-1.9.37}\lambda^{(N)}=\left(\frac{\tau \, j_1}{1+\tau^2\, j_1^2},...,\frac{\tau \, j_N}{1+\tau^2\, j_N^2} \right),\quad \lambda^\infty=(\frac{\tau \, j}{1+\tau^2\, j^2}:\, j\in\bar{\mathbb Z}\setminus J) \ee
and $\m B=(\m B_{kl}:\; k,l\in J)$ with matrix elements
\[
 \m B_{kl}=
 \left\{\begin{array}{ll} -\frac{2}{T}\frac{\tau^2\, j_k\, j_l}{\left(\tau^2(j_k^2+j_k\, j_l+j_l^2)+3\right)\left(\tau^2(j_k^2-j_k\, j_l+j_l^2)+3\right)},& k\neq l,  \\  \frac{1}{6T}\frac{1}{\tau^2\, j_k^2+1}, & k=l, \end{array} \right.\]
and $S=(S_{kl}:\, k\in \bar{\mathbb Z}\setminus J,\,l\in J)$ with matrix elements
\be\label{Skl}
S_{kl}=-\frac2T\frac{\tau^2k\, l}{\left(\tau^2(k^2+kl+l^2)+3\right)\left(\tau^2(k^2-kl+l^2)+3\right)}.
\ee

Let $M =\text{diag}\, (\frac{1}{6T}\frac{1}{\tau^2\, j_k^2+1}:\; k=1,...,N,\;j_k\in J)$. Clearly, there does exist $M^{-1}$ and the matrix elements of the matrix $M^{-1}\,\m B$ are
\[
 (M^{-1}\,\m B)_{kl}=
 \left\{\begin{array}{ll} -3\frac{(\tau^2\, j_k^2+1)(\tau^2\, j_k\, j_l)}{\left(\tau^2(j_k^2+j_k\, j_l+j_l^2)+3\right)\left(\tau^2(j_k^2-j_k\, j_l+j_l^2)+3\right)},& k\neq l,
  \\  1 , & k=l. \end{array} \right.\]
Thus we see that
\[\text{det}\, (M^{-1}\, \m B)=\frac{F(\tau)}{G(\tau)},\] where $F(\tau)$ and $G(\tau)$ are  polynomials of integral coefficients in $\tau$ and they have no common factor, i.e. $(F,G)=1$. Observe that $\text{det}\, (M^{-1}\,\m  B)=1$ when $\tau=0$. In view of the assumption that $\tau$ is transcendental, we have that $ \text{det}\, (M^{-1}\,\m  B)\neq 0$. So $ \text{det}\, ( \m B)\neq 0$. Take $\bar\Pi=[0,\,\e_0^{1/2}]^N$. Then $\text{Measure}\, \bar\Pi=\e_0^{N/2}$ and $\text{Diameter}\,\bar\Pi=\e_0^{1/2}$. Define $\Phi:\zeta\mapsto \xi$ by
\[\lambda^{(N)}+\m B\, \zeta=\xi.\]
And let $\Pi=\Phi(\bar\Pi)$. Thus,
\[\frac{1}{C}\e_0^{N/4}\le\,\text{Measure}\,\Pi\le C\e_0^{N/4},\,\quad \frac{1}{C}\e_0 \le \text{Diameter}\,\Pi\le C\e_0,\]
and
\be\label{17510-1}\omega^0(\zeta)=\xi,\quad \Omega^0(\xi)=\lambda^\infty-S\, \m B^{-1}\,\lambda^{(N)}+S\, \m B^{-1}\,\xi.\ee
Observe that $|S_{kl}|\le C/|k|^2$ for $k\in\bar{\mathbb Z}\setminus J,\; l\in J$.
It follows that the Assumptions {\bf A} and \eqref{3-08-morning} and \eqref{170512-1} of Assumption {\bf B} are fulfilled with taking $\kappa=1$.
Arbitrarily take an infinite dimensional integer vector $l=(l_{j}\in \mathbb{Z}: j\in \overline{\mathbb{Z}}\backslash J)$ with $1\leq |l|\leq 2$
(here $|l|=\sum_{j\in \overline{\mathbb{Z}}\backslash J}|l_{j}|$). Observe that where $0\neq k\in \mathbb{Z}^{N},$
the function
$  (k,\xi)+\<l,\Omega^0(\xi)\>$
is a affine function of $\xi$. Thus, letting $\xi_0$ is the direction such that $\frac{d}{d\xi_0}\,(k,\xi)=|k|$ along the direction, then
\[\frac{d}{d\,\xi_0}\left((k,\xi)+\<l,\Omega^0(\xi)\>\right)=(k,\xi_0)+\<l,\Omega^0(\xi_0)\>.\]
And note that the affine function can be written as $F(\tau)/G(\tau)$ where $F,\, G$ are polynomials of integral coefficient in $\tau$.
And note that when $\tau=0$ \[(k,\xi_0)+\<l,\Omega^0(\xi_0)\>=(k,\xi_0)=|k|\neq 0.\]
Thus
\[(*):=\frac{d}{d\,\xi_0}\left((k,\xi)+\<l,\Omega^0(\xi)\>\right)=(k,\xi_0)
+\<l,\Omega^0(\xi_0)\>\neq 0.\]
From \eqref{Skl} and the fact $|S_{kl}|\le C/|k|^2$ for any $k\in \bar{\mathbb Z}\setminus J,\, l\in J$, it follows that there are a constant $C_0>0$ and an integer $j_0\in \bar{\mathbb Z}\setminus J$ such that
\[|\l l,\Omega^0(\xi_0)\r|\le C_0/|j_0|.\]
We can choose $J$ such that $|j_0|$ large enough with $ C_0/|j_0|<1$. Thus there is $c_0>$ such that $(*)>c_0>0$.
 This verifies \eqref{170510-2}.
Let
\[D_p:=D_p(\e_0)=\{(x,y,\hat z,\bar {\hat{z}})\in \, \mathbb C^N/(2\pi\mathbb Z)^N\times \mathbb Z^N\times h_p\times h_p:\; |\Im\, x|\le s_0,\, |y|\le \e_0,\, ||\hat z||_p\le \e_0^{1/3},\,||\overline{\hat z}||_p\le \e_0^{1/3}\}.\]
By \eqref{170510-y3},  we verify Assumption  {\bf C} and \be \W \lf X_{\hat G+\tilde R}\rc\W_{q,D_p\times\Pi} \le C\,\epsilon_0,\; \W \lf \p_\xi\,X_{\hat G+\tilde R}\rc\W_{q,D_p\times\Pi} \le C\, \sqrt{\epsilon_0}.\ee
 In \eqref{170510-y4}, $u$ is real if $z_{-j}$ is the complex conjugate of $z_j$. It follows that Assumption {\bf  D} holds true. Finally,  Assumption {\bf E} holds true clearly, since $B=0$.

\ \

By Theorem \ref{theorem2} we have the following theorem.

\begin{thm} \label{on-BBM}Assume $\frac{2\pi}{T}$ is transcendental. Around the neighborhood of  $u=0$,  BBM equation \eqref{w10} has many (the initial value set of $N$-dimensional positive Lebesgue measure) smooth solutions  which are  quasi-periodic in time, linear stable and of zero Lyapunov exponent. More exactly, there exists $\e_0^*=\e_0^*(N,\tau, J)>0$ depending on $N,\tau, J$ such that for any $0<\e_0<\e_0^*$ there is a subset $\breve{\Pi}$ of the initial value set $\Pi_0:=[\sqrt{\e_0},2\sqrt{\e_0}]^N$ with
\[\text{Leb}\;\breve{\Pi}=(\text{Leb}\, \Pi_0)\; \left(1-C \frac{1}{|\log\,\e_0|}\right) \] and for any $\xi=(\xi_l:\,l=1,...,N)\in\breve{\Pi}$, BBM equation has a quasi-periodic solution $u(t,x)$ of frequency $\omega\in\mathbb R^N$ in time $t$
\[u(t,x)=\sum_{k\in\mathbb Z^N,j\in\mathbb Z\setminus\{0\}}\;\hat{u}(k,j)\; e^{{\bf i}(k,\omega)}\, e^{{\bf i}j\,\tau\, x} \] satisfying
\[|\omega-\omega_0|\le C\sqrt{\e_0},\; \omega\in\mathbb R^N, \; \omega_0=(\frac{\tau\, j_l}{1+\tau^2\, j_l^2}:\; j_t\in J)\in\mathbb R^N,
\]
\[ \left| \hat{u}(e_l,j_l)-\xi_l\right|<C\, \e_0^{1/3},\; e_l-l^{\text{th}}\, \text{unit vector of}\; \mathbb Z^N,\;
j_l\in J,\,l=1,...,N, \] and
\[\sum_{(k,j)\notin \mathcal{S}}\left|\hat{u}(k,j \right|^2e^{|k|\, s_0+2a\,|j|}|j|^{2p}<C\,\e_0^{1/3},\quad \mathcal{S}={(e_l,j_l):\; l=1,...,N},\]
where some constants $s_0>0,a>0$ and $p>1/2$.
\end{thm}
\begin{rem} Theorem  \ref{theorem2} applies to more general BBM equation:
\begin{equation}\label{w10}
u_t-u_{xxt}+u_x+F(u)\,u_x=0,\quad u(t,0)=u(t,T),
\end{equation}
and Hirota-Satsuma equation
\be u_t-u_{xxt}+u_x-F(u)\, u_t-u_x\p_x^{-1}\, u_t=0,\quad u(t,0)=u(t,T),\ee
where $F(u)=u+\sum_{j\ge 2} c_j\, u^j$ is an analytic function of $u$ with $c_j\in\mathbb R$.
\end{rem}

\begin{rem} We are glad to mention a recent paper \cite{Zengchongchun}, where the linear stability of traveling wave solution of BBM is studied among the other things.
\end{rem}

\section{Application to $d$-dimensional generalized Pochhammer-Chree equation}

The  Pochhammer-Chree (gPC) equation represents a nonlinear model of longitudinal wave
propagation of elastic rods\cite{Chree,Pochhammer}. See also \cite{Bogolu,Clarkson,Park,Parand}, for example.
Consider a $d$-dimensional generalized Pochhammer-Chree (gPC) equation
\begin{equation}\label{L1}
\left\{
  \begin{array}{ll}
    u_{tt}-\Delta u-\Delta u_{tt} +\Delta (u^{3})=0,\;\;x\in \Omega\subset\mathbb{R}^{d},\\
    u|_{\partial \Omega}=0,
  \end{array}
\right.
\end{equation}
where $\Omega=\Pi_{j=1}^{d}[0, T_{j}]$ with $T_{j}>0.$

Let
\begin{equation}\label{L1+}
\phi_{k}(x)=\sin k_{1}\tau_{1}x_{1}\cdots \sin k_{d}\tau_{d}x_{d},\;\;k\in \mathbb{Z}_{+}^{d},\;\;\tau_{j}=\frac{2\pi}{ T_{j}}
\end{equation}
with $\mathbb{Z}_{+}^{d}=\{k=(k_{1}, \cdots, k_{d})\mid k_{j}\in \mathbb{Z}_{+},\;\;j=1, \cdots, d\},$
where $\mathbb{Z}_{+}=\{1, 2, 3, \cdots\}$ is the set of all positive integers.

Let
\begin{equation}\label{L2}
u(t, x)=\sum_{k\in \mathbb{Z}_{+}^{d}}u_{k}(t)\phi_{k}(x)
\end{equation}
and put \eqref{L2} into \eqref{L1}. Then we have
\begin{equation}\label{L3}
\ddot{u}_{k}+\parallel k \parallel^{2}u_{k}+\parallel k \parallel^{2}\ddot{u}_{k}+\parallel k \parallel^{2} \widetilde{G}_{k}(u)=0,\;\;k\in \mathbb{Z}_{+}^{d},
\end{equation}
where
\begin{equation}\label{L4}
\parallel k \parallel^{2}=\tau_{1}^{2}k_{1}^{2}+\cdots +\tau_{d}^{2}k_{d}^{2},\;\;k=(k_{1},\cdots, k_{d}),
\end{equation}
\begin{equation}\label{L5}
\widetilde{G}_{k}(u)=\sum_{m,n,l\in \mathbb{Z}_{+}^{d}}\widetilde{C}_{mnlk}u_{m}u_{n}u_{l},
\end{equation}
and
\begin{equation}\label{L6}
\widetilde{C}_{mnlk}=\int_{\Omega}\phi_{m}\phi_{n}\phi_{l}\phi_{k}dx,\quad m,n,l,k\in\mathbb Z^d_+.
\end{equation}
Rewrite \eqref{L3} as
\begin{equation}\label{L7}
\ddot{u}_{k}+\lambda_{k}u_{k}+\lambda_{k}\widetilde{G}_{k}(u)=0,\;\;k\in \mathbb{Z}_{+}^{d},\;\;\lambda_{k}=\frac{\parallel k\parallel^{2}}{1+\parallel k\parallel^{2}}.
\end{equation}
Let
$$u_{k}=\frac{1}{\sqrt[4]{\lambda_{k}}}w_{k},\;\;\dot{u}_{k}=\sqrt[4]{\lambda_{k}}v_{k},\;\;k\in \mathbb{Z}_{+}^{d}.$$
Then \eqref{L7} reads
\begin{equation}\label{L8}
\left\{
  \begin{array}{ll}
    \dot{w}_{k}=\sqrt{\lambda_{k}}v_{k}, \\
    \dot{v}_{_{k}}=-\sqrt{\lambda_{k}}w_{k}-G_{k}(w),
  \end{array}
\right.
\end{equation}
where
\begin{eqnarray}\label{L9}
G_{k}(w)&=&\frac{\lambda_{k}}{\sqrt[4]{\lambda_{k}}}\widetilde{G}_{k}(u)=\frac{\lambda_{k}}{\sqrt[4]{\lambda_{k}}}\,
\sum_{m, n, l\in\mathbb Z_+^d}\widetilde{C}_{mnlk}\frac{w_{m}w_{n}w_{l}}{\sqrt[4]{\lambda_{m}}\sqrt[4]{\lambda_{n}}\sqrt[4]{\lambda_{l}}}\\
&=& \lambda_{k}\sum_{m, n, l\in\mathbb Z_+^d}C^*_{mnlk}w_{m}w_{n}w_{l},\nonumber
\end{eqnarray}
\begin{eqnarray}\label{L10}
C^*_{mnlk}=\frac{1}{\sqrt[4]{\lambda_{m}}\sqrt[4]{\lambda_{n}}\sqrt[4]{\lambda_{l}}\sqrt[4]{\lambda_{k}}}\int_{\Omega}\phi_{m}\phi_{n}\phi_{l}\phi_{k}dx.
\end{eqnarray}
We write $G_{k}(w)$ in the form of gradient:
\begin{eqnarray}\label{L11}
G_{k}(w)&=&\lambda_{k}\sum_{m,n,l}C^*_{mnlk}w_{m}w_{n}w_{l}\\
&=& \frac{1}{4}(\lambda_{m}\sum_{n,l,k}C^*_{mnlk}w_{n}w_{l}w_{k}+\lambda_{n}\sum_{m,l,k}C^*_{mnlk}w_{m}w_{l}w_{k}\nonumber\\
&&+\lambda_{l}\sum_{m,n,k}C^*_{mnlk}w_{m}w_{n}w_{k}+\lambda_{k}\sum_{m,n,l}C^*_{mnlk}w_{m}w_{n}w_{l})\nonumber\\
&=&\partial_{w_{k}}G(w),
\end{eqnarray}
where
\begin{eqnarray}\label{L12}
G(w)=\sum_{m, n ,l,k}C_{mnlk}w_{m}w_{n}w_{l}w_{k},
\end{eqnarray}
\begin{eqnarray}\label{L13}
C_{mnlk}=\frac{\lambda_{m}+\lambda_{n}+\lambda_{l}+\lambda_{k}}{4\sqrt[4]{\lambda_{m}\lambda_{n}\lambda_{l}\lambda_{k}}}
\int_{\Omega}\phi_{m}\phi_{n}\phi_{l}\phi_{k}dx.
\end{eqnarray}
By \eqref{L1+}, we have
\begin{eqnarray}\label{L14}
C_{mnlk}=0,\quad  \mbox{unless there is a combination of $+$ and $-$ such that} \;\; m\pm n\pm l\pm k=0.
\end{eqnarray}
Now \eqref{L8} can be written as a Hamiltonian system
\begin{eqnarray}\label{L15}
\left\{
  \begin{array}{ll}
     \dot{w}_{k}=\frac{\partial H}{\partial v_{k}},\\                                                             \\
     \dot{v}_{k}=-\frac{\partial H}{\partial w_{k}},\;\;k\in \mathbb{Z}_{+}^{d},
  \end{array}
\right.
\end{eqnarray}
where
\begin{eqnarray}\label{L16}
H=\sum_{k\in \mathbb{Z}_{+}^{d}}\frac{1}{2}\sqrt{\lambda_{k}}(w_{k}^{2}+v_{k}^{2})+G(w),
\end{eqnarray}
\begin{eqnarray}\label{L17}
G(w)=\sum_{m\pm n\pm l\pm k=0}C_{mnlk}w_{m}w_{n}w_{l}w_{k}.
\end{eqnarray}
\begin{lem}\label{Llem1}
Let $\partial_{w}G=(\partial_{w_{k}}G: k\in \mathbb{Z}_{+}^{d}).$ Then
$$\parallel \lfloor \partial_{w} G\rceil\parallel_{p}\leq c \parallel w\parallel_{p}^{3}.$$
\end{lem}
\begin{proof}
Recall \eqref{L12} and \eqref{L14}. Then
$$\partial_{w_{k}}G=G_{k}(w)=\lambda_{k}\sum_{\pm m\pm n\pm l=k}C_{mnlk}w_{m}w_{n}w_{l}.$$
Thus
$$\lfloor \partial_{w_{k}}G\rceil\leq C\sum_{\pm m\pm n\pm l=k}|w_{m}||w_{n}||w_{l}|
=C{(w\ast w\ast w)(k)}$$ where $\ast$ is the convolution in $\ell_2(\mathbb Z_+^d)$.
So $$\parallel \lfloor \partial_{w}G\rceil\parallel_{p}\leq C \parallel w\parallel_{p}^{3}.$$

\end{proof}
Let $$z_{k}=\frac{1}{\sqrt{2}}(w_{k}+\sqrt{-1}v_{k}),\;\;\overline{z}_{k}=\frac{1}{\sqrt{2}}(w_{k}-\sqrt{-1}v_{k}),\;\;k\in \mathbb{Z}_{+}^{d}.$$
This is a symplectic transformation, which changes \eqref{L16} into
\begin{eqnarray}\label{L18}
H=\sum_{k}\sqrt{\lambda_{k}}z_{k}\overline{z}_{k}+G(z, \overline{z}),
\end{eqnarray}
where the symplectic structure is $\sqrt{-1}\, d\,\bar{z}\wedge d\, z,$
and
\begin{eqnarray}\label{L19}
G(z, \overline{z})=G(w)=\sum_{m\pm n\pm l\pm k=0}\frac{1}{4}C_{mnlk}(z_{m}
+\overline{z}_{m})(z_{n}+\overline{z}_{n})(z_{l}+\overline{z}_{l})(z_{k}+\overline{z}_{k}).
\end{eqnarray}

A  polynomial of degree $n$ of $d$-dimension variable $x=(x_1,...,x_d)$ can be written as
\[P(x)=\sum_{\alpha\in\mathbb N^d, |\alpha|\le n} C_{\alpha} \, x^\alpha,\quad x^{\alpha}=x_1^{\alpha_1}\cdots x_d^{\alpha_d},\; \alpha=(\alpha_1,...,\alpha_d)\]
where $C_{\alpha}$'s are coefficients and there is $\alpha$ with $|\alpha|= n$ such that $C_{\alpha}\neq 0$. We call that $P(x)$ is a polynomial  with coefficients in the field of
the rational numbers, if all coefficients $C_{\alpha}$'s are in $\mathbb Q$. By $\mathbb Q_n[x]$ denote the set of all  polynomials of degree $n$ with coefficients in the field of
the rational numbers. Let $\mathbb Q[x]=\bigcup_{n\in\mathbb N}\,  \mathbb Q_n[x]$ where $\mathbb N=\{1,2,...\}$. Clearly, the set $\mathbb Q[x]$ is countable. For any $P\in \mathbb Q[x]$, let $S_P$ be the set of all solutions to the polynomial equation
$P(x)=0$. Let $\Theta=[1,2]^d$. We claim that the Lebesgue measure of $S_P\bigcap \Theta$ is zero. In fact, the result is clear when the dimension $d=1$. The proof for $d\ge 2$ can be finished by Fubini Theorem and mathematical induction.
Let
\[S=\bigcup_{P\in\mathbb Q[x]}\, (S_P\bigcap \Theta).\]
Considering that  $\mathbb Q[x]$ is countable, we have that
$\text{Leb}\; S=0$.
 Define
$\tilde\Theta=\Theta \setminus  S$.
 When $\tau\in\tilde \Theta$, we call $\tau$ is typical. At this time, $\text{Leb}\, \tilde \Theta=1$. Therefore, for any $P\in\mathbb Q[x]$ and any $\tau\in\tilde\Theta$, we have $P(\tau)\neq 0$.

Fix arbitrarily integer $N$ which denotes the number of the incited oscillators. Let
$$J=\{j_{1}, \cdots, j_{N}\mid L< |j_{1}|<|j_{2}|<\cdots <|j_{N}|,\;\;j_{t}\in \mathbb{Z}_{+}^{d},\;\;t=1, 2, \cdots, N\},$$
where $L$ is supposed to be large $L\gg 1$ and $L\gg N$. The large $L\gg 1$ means that the incited oscillators are of high frequency. Here
it should be pointed out that the assumption $L\gg 1$ is just for simplifying the following computation.
Split $z=(z_{j}: j\in \mathbb{Z}_{+}^{d})=(\tilde{z}, \hat{z})$ with
$$\tilde{z}=(z_{j_{1}}, \cdots, z_{j_{N}}),\;\;\hat{z}=z\ominus \tilde{z}.$$
We will eliminate those terms of lower frequencies in $G(z,\bar z)$ , which involve $\tilde z$, as many as possible by Birkhoff normal form.
To this end we need the following lemma.
\begin{lem}\label{Llem2-}
Assume that $\tau=(\tau_{1}, \cdots, \tau_{d})$ is typical, i.e.,$\tau\in\tilde \Theta$.
\begin{enumerate}
  \item [(1)] If $m\pm n\pm k\pm l=0$ and $\{m, l\}\neq \{n, k\},$ then$$\sqrt{\lambda_{m}}-\sqrt{\lambda_{n}}+\sqrt{\lambda_{l}}-\sqrt{\lambda_{k}}\neq 0;$$
  \item [(2)] If $m\pm n\pm k\pm l=0,$ then $$\pm (\sqrt{\lambda_{m}}+\sqrt{\lambda_{n}}+\sqrt{\lambda_{l}}+\sqrt{\lambda_{k}})\neq 0,\; \;\pm (\sqrt{\lambda_{m}}+\sqrt{\lambda_{n}}+\sqrt{\lambda_{l}}-\sqrt{\lambda_{k}})\neq  0;$$
  \item [(3)] $$\sqrt{\lambda_{m}}\pm \sqrt{\lambda_{n}}\pm \sqrt{\lambda_{l}}+p\neq 0,\;\;p=0, \pm 1,\cdots ;$$
  \item [(4)] $$\sqrt{\lambda_{m}}\pm \sqrt{\lambda_{n}}+p\neq 0,\;\;p=0, \pm 1, \pm 2,$$ \textcolor[rgb]{0.00,0.00,0.00}{where $m\neq n$ for  $\sqrt{\lambda_{m}}- \sqrt{\lambda_{n}}+p$.}
\end{enumerate}
\end{lem}
\begin{proof}
We give the proof for only case (1). The remaining is similar. Recall that $\lambda_m$ is a function of $\tau$ and
$$\lambda_{m}=\lambda_m(\tau)=\frac{\parallel m\parallel^{2}}{1+\parallel m\parallel^{2}}=\frac{\tau_{1}^{2}m_{1}^{2}
+\cdots+\tau_{d}^{2}m_{d}^{2}}{1+\tau_{1}^{2}m_{1}^{2}+\cdots +\tau_{d}^{2}m_{d}^{2}},\;\;\forall \; m=(m_1,...,m_d)\in \mathbb{Z}_{+}^{d}.$$ Write $m=(m_1,...,m_d)\in\mathbb Z_+^d$, etc. Since  $\{m, l\}\neq \{n, k\}$, we can assume that $m_1\neq n_1$ or $l_1\neq k_1$ without loss of generality. For $\tau=(\tau_1,\tau_2,...,\tau_d)\in \mathbb R^d$, let $\tau^{(1)}=(\tau_1,0,...,0)$. Then
\[\lambda_m(\tau^{(1)})=\frac{\tau_1^2 m_1^{2}}{1+ \tau_1^2 m_1^{2}}=\frac{ m_1^{2}}{\tau_1^{-2}+ m_1^{2}}=\frac{ m_1^{2}}{s+ m_1^{2}}:=\lambda_m(s),\quad \text{here}\;\; s=\tau_1^{-2}.\]

Set
\[\Gamma(s)=\sqrt{\lambda_{m}(s)}-\sqrt{\lambda_{n}(s)}+\sqrt{\lambda_{l}(s)}-\sqrt{\lambda_{k}(s)}\]
Then by a simple computation we have
\[-2\frac{d\,\Gamma}{d\,s}\left.\right|_{s=0}=\frac{1}{m_1^2}-\frac{1}{n_1^2}+\frac{1}{l_1^2}-\frac{1}{k_1^2}:=(*)\]
and
\[\frac43\frac{d^2\,\Gamma}{d\,s^2}\left.\right|_{s=0}=\frac{1}{m_1^4}-\frac{1}{n_1^4}+\frac{1}{l_1^4}-
\frac{1}{k_1^4}:=(**).\]
By $m_1\pm n_1\pm k_1\pm l_1=0$ and $m_1\neq n_1$ or $l_1\neq k_1$, we get that either $(*)\neq 0$ or $(**)\neq 0$. It follows that there exists a $s_0\in\mathbb R$ such that $\Gamma(s_0)\neq 0$. Moreover,   there exists a $\tau_0\in\mathbb R^d$ such that
\[\gamma(\tau_0):=\sqrt{\lambda_{m}(\tau_0)}-\sqrt{\lambda_{n}(\tau_0)}+
\sqrt{\lambda_{l}(\tau_0)}-\sqrt{\lambda_{k}(\tau_0)}\neq 0 \]
If $\sqrt{\lambda_{m}}-\sqrt{\lambda_{n}}+\sqrt{\lambda_{l}}-\sqrt{\lambda_{k}}=0$ where $\lambda_m=\lambda_m(\tau)$, etc., then $\sqrt{\lambda_{m}}+\sqrt{\lambda_{l}}=\sqrt{\lambda_{n}}+\sqrt{\lambda_{k}}.$
It follows
\begin{equation}\label{L51}
[(\lambda_{m}+\lambda_{l}-\lambda_{n}-\lambda_{k})^{2}-4(\lambda_{n}\lambda_{k}+\lambda_{m}\lambda_{l})]^{2}-64\lambda_{m}\lambda_{n}\lambda_{k}\lambda_{l}=0.
\end{equation}

Multiplying \eqref{L51} by $(1+\parallel m\parallel^{2})^{4}(1+\parallel n\parallel^{2})^{4}(1+\parallel k\parallel^{2})^{4}(1+\parallel l\parallel^{2})^{4},$ and noting
$m\pm n\pm k\pm l=0,$
we get
\begin{eqnarray}\label{L52}
P(\tau):=\sum_{\begin{array}{c}
        10\le |\alpha|\le 26 \\
        \alpha\in\mathbb Z_+^d
       \end{array}}
C_{\alpha}\,\tau^{\alpha}=0
\nonumber
\end{eqnarray}
where $C_{\alpha}\in \mathbb{Z}$.  By $\gamma(\tau_0)\neq 0$, we get that there exists a coefficient $C_\alpha\neq 0$. Thus, $P(x)\in\mathbb Q[x]$. Thus $P(\tau)\neq 0$ when $\tau\in\tilde\Theta$,
which is contradictory to \eqref{L51}. This completes the proof.
\end{proof}
\begin{lem}\label{Llem2--} Assume $\tau\in\tilde\Theta$.
If $m\pm n\pm k\pm l=0,$ and $\{m, n, k, l\}\cap J\neq \emptyset.$ Let $\triangle =\sqrt{\lambda_{m}}\pm \sqrt{ \lambda_{n}}\pm \sqrt{\lambda_{k}}\pm \sqrt{\lambda_{l}}$
(excluding $\Delta=\sqrt{\lambda_{m}}- \sqrt{ \lambda_{n}}+ \sqrt{\lambda_{l}}- \sqrt{\lambda_{k}}$ with $\{m, l\}=\{n, k\}$). Then there exists a  constant $C>0$
such that $|\triangle|>C>0.$
\end{lem}
\begin{proof}
We give proof for only the most difficult case
$$\triangle=\sqrt{\lambda_{m}}- \sqrt{ \lambda_{n}}+ \sqrt{\lambda_{l}}- \sqrt{\lambda_{k}},$$
where there is a combination of $+$ and $-$ such that $m\pm n\pm l\pm k=0,$ and $\{m, l\}\neq \{n, k\},$ and assuming
$$m, n\in J.$$
Let $\widetilde{N}>0$ be a large number which is to be specified later.
If $|l|\leq \widetilde{N}$ (or $|k|\leq \widetilde{N}$), by $m\pm n \pm k\pm l=0$ and $m,n\in J$,  we have that there exists $C_{1}=C_{1}(\widetilde{N})>0$ such that
$|k|\leq C_{1}(\widetilde{N}).$ Recall that $m, n\in J$ implies $|m|+|n|\leq 2(L+j_N).$ So by taking $\tilde N>L+j_N$  and using Lemma \ref{Llem2-} there is a constant $C_{2}=C_{2}(\widetilde{N})>0$ such that
$$|\Delta|\geq C_{2}(\widetilde{N})>0.$$
Now we assume $|l|>\widetilde{N}$ and $|k|> \widetilde{N}.$ Then
\begin{eqnarray}\label{L*}
|\Delta|&=&|\sqrt{\lambda_{m}}- \sqrt{ \lambda_{n}}+ \sqrt{\lambda_{l}}- \sqrt{\lambda_{k}}|\nonumber\\
        &=& \big|\sqrt{\lambda_{m}}- \sqrt{ \lambda_{n}}+ \sqrt{1-\frac{1}{1+\parallel l\parallel^{2}}}- \sqrt{1-\frac{1}{1+\parallel k\parallel^{2}}}\big|\nonumber\\
&\geq & |\sqrt{\lambda_{m}}- \sqrt{ \lambda_{n}}|-\frac{C_{0}}{\widetilde{N}^{2}},
\end{eqnarray}
where $C_{0}=C_{0}(\Theta)>0$ is a constant. Since $m\neq n,$ and $m, n\in J$, we have that there exists a constant $C_{3}=C_{3}(L+j_N)$ such that
$$|\sqrt{\lambda_{m}}- \sqrt{ \lambda_{n}}|\geq C_{3}(L+j_N).$$
Choose $\widetilde{N}$ large enough such that
$$C_{3}-\frac{C_{0}}{\widetilde{N}^{2}}\geq \frac{C_{3}}{2}.$$
Then by \eqref{L*},
$$|\Delta|\geq \frac{C_{3}}{2}.$$
Consequently, let
$C_{4}=C_{4}(N, \widetilde{N})=\min\{C_{2}(\widetilde{N}), \frac{1}{2}C_{3}(N)\}.$
Then $|\Delta|\geq C_{4}(L,N, \widetilde{N})>0.$
\end{proof}
\begin{lem}\label{Llem2} Assume $\tau\in\tilde\Theta$.
There exists a symplectic $\Phi=Id.+O(||z||_p)$ such that
\be\label{HH} H\circ \Phi=\sum_{k\in \mathbb{Z}_{+}^{d}}\sqrt{\lambda_{k}}z_{k}\overline{z}_{k}+\overline{G}+\widehat{G}+\breve{G},\ee
where
\begin{eqnarray*}
&&\overline{G}=\sum_{\begin{array}{c}
    k,l\in\mathbb{Z}_{+}^{d} \\
     \{k, l\} \cap J\neq \emptyset
    \end{array}}\overline{G}_{kl}|z_{k}|^{2}|z_{l}|^{2},\\
&&\overline{G}_{kl}=\left\{
                      \begin{array}{ll}
                        (\frac{3}{8})^{d}T_{1}\cdots T_{d}\sqrt{\lambda_{k}} ,\;\;k=l,\\
                                                                                      \\
                       \frac{1}{2}\sum_{1\leq p\leq d}\left(
                                                        \begin{array}{c}
                                                          d \\
                                                          p \\
                                                        \end{array}
                                                      \right)(\frac{1}{4})^{p}(\frac{3}{8})^{q}T_{1}\cdots T_{d}(\sqrt{\lambda_{k}}+\sqrt{\lambda_{l}}),\;\;k\neq l,
                      \end{array}
                    \right.\\
&& \parallel X_{\widehat{G}}\parallel_{p}\leq C \parallel \hat{z}\parallel_{p}^{{3}},\;\;
\parallel X_{\breve{G}}\parallel_{p}\leq C \parallel {z}\parallel_{p}^{{5}}.
\end{eqnarray*}
\end{lem}
\begin{proof}
Decompose $G(z,\overline{z})$ in \eqref{L19} as follows
\begin{eqnarray*}
 G(z, \overline{z})&=&G^{(1)}(z, \overline{z})+G^{(2)}(z,\overline{ z})+G^{(3)}(z,\overline{ z}),\\
 G^{(1)}(z, \overline{z})&=&\sum_{\{m,k\}\bigcap J\neq \emptyset}C_{mmkk}z_{m}\overline{z}_{m}z_{k}\overline{z}_{k},\\
 G^{(2)}(z, \overline{z})&=&\sum_{\begin{array}{c}
                                    m\pm n\pm l\pm k=0 \\
                                     \{m, n, l, k\}\cap J\neq \emptyset
                                  \end{array}}\frac{1}{4}C_{mnlk}(z_{m}+\overline{z}_{m})(z_{n}+\overline{z}_{n})
(z_{l}+\overline{z}_{l})(z_{k}+\overline{z}_{k})-G^{(1)}(z,\overline{ z}),\\
 G^{(3)}(z, \overline{z})&=&\sum_{\begin{array}{c}
                                    m\pm n\pm l\pm k=0 \\
                                     \{m, n, l, k\}\cap J= \emptyset
                                  \end{array}}\frac{1}{4}C_{mnlk}(z_{m}+\overline{z}_{m})(z_{n}+\overline{z}_{n})
(z_{l}+\overline{z}_{l})(z_{k}+\overline{z}_{k}).
\end{eqnarray*}
By \eqref{L13}, then
$$\overline{G}_{kl}:=C_{kkll}=\frac{1}{2}\frac{\lambda_{k}+\lambda_{l}}{\sqrt{\lambda_{k}\lambda_{l}}}\int_{\Omega}\phi_{k}^{2}\phi_{l}^{2}dx.$$
By \eqref{L1+}, we have
\begin{eqnarray*}
\overline{G}_{kl}&=&\frac{1}{2}(\sqrt{\lambda_{k}/\lambda_l}+\sqrt{\lambda_{l}/\lambda_k})\prod_{j=1}^{d}\int_{0}^{T_{j}}\sin^{2}k_{j}\tau_{j}x_{j}\sin^{2}l_{j}\tau_{j}x_{j}dx_{j}\\
&=&\left\{
     \begin{array}{ll}
       (\frac{3}{8})^{d}\, T_{1}T_{2}\cdots T_{d},\;\;k=l,\\
\\
       \frac{1}{2}\sum_{\begin{array}{c}
                          p+q=d \\
                          p\geq 1
                        \end{array}}\left(
                  \begin{array}{c}
                    d \\
                    p \\
                  \end{array}
                \right)(\frac{1}{4})^{p}(\frac{3}{8})^{q}T_{1}\cdots T_{d}(\sqrt{\lambda_k/\lambda_l}+\sqrt{\lambda_l/\lambda_k}),\;\;k\neq l.
     \end{array}
   \right.
\end{eqnarray*}
Rewrite
\begin{eqnarray*}
G^{(2)}(z, \overline{z})&=&\sum \frac{1}{4}C_{mnlk}z_{m}z_{n}z_{l}z_{k}+\sum \frac{1}{4}C_{mnlk}z_{m}z_{n}z_{l}\overline{z}_{k}
+\sum_{\{m, n\}\neq \{k, l\}} \frac{1}{4}C_{mnlk}z_{m}z_{n}\overline{z}_{l}\overline{z}_{k}\\
&&+\sum \frac{1}{4}C_{mnlk}z_{m}\overline{z}_{n}\overline{z}_{l}\overline{z}_{k}+\sum \frac{1}{4}C_{mnlk}\overline{z}_{m}\overline{z}_{n}\overline{z}_{l}\overline{z}_{k}
\end{eqnarray*} where the sum runs over $m,n,k,l\in\mathbb Z_+^d$ with some $m\pm n\pm k\pm l=0$.
Let $F=F(z,\bar z)$ be of the same form as $G^{(2)}(z, \overline{z})$:
\begin{eqnarray}
F=F(z, \overline{z})&=&\sum \frac{C_{mnlk}\,z_m\,z_n\, z_k\,\, z_l}{4\mathbf{i}(\sqrt{\lambda_{m}}+\sqrt{\lambda_{n}}+\sqrt{\lambda_{l}}+\sqrt{\lambda_{k}})}
+\sum \frac{C_{mnlk}\,z_m\,z_n\, z_k\,\,\bar z_l}{4\mathbf{i}(\sqrt{\lambda_{m}}+\sqrt{\lambda_{n}}+\sqrt{\lambda_{l}}-\sqrt{\lambda_{k}})}\nonumber\\ \label{Real}
&&+\sum \frac{C_{mnlk}\,z_m\,z_n\,\bar z_k\,\,\bar z_l}{4\mathbf{i}(\sqrt{\lambda_{m}}+\sqrt{\lambda_{n}}-\sqrt{\lambda_{l}}-\sqrt{\lambda_{k}})}+\sum \frac{C_{mnlk}\,z_m\,\bar z_n\,\bar z_k\,\,\bar z_l}{4\mathbf{i}(\sqrt{\lambda_{m}}-\sqrt{\lambda_{n}}-\sqrt{\lambda_{l}}-\sqrt{\lambda_{k}})}\\
&&+\sum \frac{C_{mnlk}\,\bar z_m\,\bar z_n\,\bar z_k\,\,\bar z_l}{-4\mathbf{i}(\sqrt{\lambda_{m}}+\sqrt{\lambda_{n}}+\sqrt{\lambda_{l}}+\sqrt{\lambda_{k}})}\nonumber
\end{eqnarray} where the sum runs over the same set as $G^{(2)}(z, \overline{z})$.
By Lemma \ref{Llem2--}, it is easy to get
$$\parallel \lfloor X_{F}\rceil\parallel_{p}\leq C\parallel z\parallel_{p}^{3}.$$
Moreover, $\Psi:=X_{F}^{t}\mid _{t=1}=Id.+O(||z||^3_p)$ maps a neighborhood of $z=0$ in $l_{p}$ into another neighborhood of $z=0$ in $l_{p}$,
and
$$H\circ \Psi=\sum_{k\in \mathbb{Z}_{+}^{d}}\sqrt{\lambda_{k}}z_{k}\overline{z}_{k}+\overline{G}+\widehat{G}+\breve{G},$$
where $\overline{G}=G^{(1)},\;\;\widehat{G}=G^{(3)},$
$$\breve{R}=\int_{0}^{1}\int_{0}^{t}\{\{\sum_{k\in \mathbb{Z}_{+}^{d}}\sqrt{\lambda_{k}}z_{k}\overline{z}_{k}, F\}, F\}
\circ X_{F}^{\tau}d \tau dt+\int_{0}^{1}\{G, F\}\circ X_{F}^{t} dt.$$
It is easy to check
$$\parallel\lfloor X_{\widehat{G}}\rceil\parallel_{p}\leq C\parallel \hat z\parallel_{p}^3,\;\;\;\;
\parallel \lfloor X_{\breve{G}}\rceil\parallel_{p}\leq C\parallel z\parallel_{p}^{5}.$$
\end{proof}
As in Section 9, restrict $\mid\tilde{z}\mid<C\varepsilon_{0}^{1/4}$ and $\parallel \hat{z}\parallel_{p}\leq C\varepsilon_{0}^{1/3}. $
Then
\be\label{2018-3-18-1}\parallel\lfloor X_{\widehat{G}+\breve{G}}\rceil\parallel_{p}\leq C(\parallel \hat{z}\parallel_{p}^{3}+\parallel z\parallel_{p}^{^{5}})\leq C\varepsilon_{0}.\ee
Introduce action-angle variables $(y, x)$ by
\be \label{sym2}z_{j_{k}}=\sqrt{\zeta_{k}+y_{k}}e^{-\mathbf{i}x_{k}},\;\;\bar{z}_{j_{k}}=
\sqrt{\zeta_{k}+y_{k}}e^{\mathbf{i}x_{k}},\;\;k=1, 2, \cdots, N,\;\;j_{k}\in J\ee
 where $\zeta=(\zeta_1,...,\zeta_n)\in\mathbb R_+^n$ and
 \be\label{sym3}\e_0^{1/2}<|\zeta|\le 2\e_0^{1/2}.\ee

Then
\be \label{2018-3-2} \sum_{j\in\mathbb{Z}_{+}^{d}}\sqrt{\lambda_{j}}z_{j}\overline{z}_{j}=\sum_{k=1}^{N}\sqrt{\lambda_{j_{k}}}\zeta_{k}+\sum_{k=1}^{N}\sqrt{\lambda_{j_{k}}}y_{k}
+\sum_{j\in \mathbb{Z}_{+}^{d}\setminus J}\lambda_{j}z_{j}\overline{z}_{j}\ee and
\begin{eqnarray}\label{2018-3-3}\bar G &=&\sum_{1\le k,l\le N}\bar{G}_{j_kl_k}(\zeta_k+y_k)(\zeta_l+y_l)+
\sum_{1\le k\le N,\, l\in{\mathbb Z_+^d}\setminus J}\bar G_{j_k l}(\zeta_k+y_k)z_l\,\bar z_{l}\\\label{2018-3-4} & &+
\sum_{k\in{\mathbb Z_+^d}\setminus J,\,1\le l\le N}\bar G_{kj_l}(\zeta_l+y_l)z_k \bar z_{k}.
\end{eqnarray}
Let
\be \label{2018bu-1}\breve{G}(x,y,\hat z,\overline{\hat z};\zeta)= \breve{G}(\tilde z,\overline{\tilde z},
\hat z,\overline{\hat z})\ee where $(\tilde z,\overline{\tilde z})$ are defined by \eqref{sym2}.
Let
\begin{eqnarray*}
\omega^{0}(\zeta)=\lambda^{(N)}+TB\zeta,\;\;\Omega^{0}(\zeta)=\lambda^{\infty}+TS\zeta,\quad T=T_{1}\cdots T_{d},
\end{eqnarray*}
where
\begin{eqnarray*}
&&\lambda^{(N)}=\left(\sqrt{\frac{\parallel j_{1}\parallel^{2}}{1+\parallel j_{1}\parallel^{2}}}, \cdots,
\sqrt{\frac{\parallel j_{N}\parallel^{2}}{1+\parallel j_{N}\parallel^{2}}}\right),\;\;j_{1}, \cdots,j_{N}\in J,\\
&& \lambda^{\infty}=\left\{\sqrt{\lambda_j}=\sqrt{\frac{\parallel j\parallel^{2}}{1+\parallel j\parallel^{2}}}:\;\;j\in \mathbb{Z}_{+}^{d}\setminus J\right\},\;\;
\\
&&B=(B_{kl}:\;k, l\in J),\\
&&B_{kl}=\left\{
           \begin{array}{ll}
             a,\;\;k=l,\;k\in J, \\
\\
            \frac{1}{2}b(\sqrt{\lambda_{k}/\lambda_l}+\sqrt{\lambda_{l}/\lambda_k}),\;\;k\neq l,\;k, l\in J,
           \end{array}
         \right.\\
&& a=(\frac{3}{8})^{d},\;\;b=\sum_{1\leq p\leq d}\left(
                                                   \begin{array}{c}
                                                     d \\
                                                     p \\
                                                   \end{array}
                                                 \right)(\frac{1}{4})^{p}(\frac{3}{8})^{d-p}=(\frac58)^d-a,\\
&&S=(S_{kl}:\;\;k\in \mathbb{Z}_{+}^{d}\setminus J,\;\;l\in J),\\
&&S_{{kl}}=\overline{G}_{kl}=\frac{1}{2}b(\sqrt{\lambda_{k}}+\sqrt{\lambda_{l}}),\;\;k\in \mathbb{Z}_{+}^{d}\setminus J,\;\;l\in J.
\end{eqnarray*}
Then, up to a constant depending on $\zeta,$ the Hamiltonian $H\circ \Phi$ in \eqref{HH} can be written as $$H=H_0+R^0$$ with
\be \label{2018-3-18-5} H_{0}=\sum_{j=1}^{N}\omega_{j}^{0}(\zeta)y_{j}+\sum_{j\in \mathbb{Z}_{+}^{d}\setminus J}\Omega_{j}^{0}(\zeta)z_{j}\overline{z}_{j},\ee
\be \label{2018-3-18-6} R^0=\sum_{1\le k,l\le N}\bar{G}_{j_kl_k}\, y_k\,y_l+
\sum_{k\in{\mathbb Z_+^d}\setminus J,\,1\le l\le N}\bar G_{kj_l}\,y_l\,z_k \bar z_{k}+\hat G(\hat z,\overline{\hat z})+\breve{G}(\tilde z,\overline{\tilde z},
 \hat z,\overline{\hat z};\zeta).\ee
Let us write the matrix $B$ explicitly:
$$B=\left(
      \begin{array}{cccc}
        a & \frac{1}{2}b(\sqrt{\lambda_{j_{1}}/\lambda_{j_{2}}}+\sqrt{\lambda_{j_{2}}/\lambda_{j_{1}}}) & \cdots & \frac{1}{2}b(\sqrt{\lambda_{j_{1}}/\lambda_{j_{N}}}+\sqrt{\lambda_{j_{N}}/\lambda_{j_{1}}}) \\
        \frac{1}{2}b(\sqrt{\lambda_{j_{2}}/\lambda_{j_1}}+\sqrt{\lambda_{j_{1}}/\lambda_{j_2}}) & a & \cdots & \frac{1}{2}b(\sqrt{\lambda_{j_{2}}/\lambda_{j_{N}}}+\sqrt{\lambda_{j_{N}}/\lambda_{j_{2}}})\\
        \vdots & \vdots & \ddots & \vdots \\
        \frac{1}{2}b(\sqrt{\lambda_{j_{N}}/\lambda_{j_1}}+\sqrt{\lambda_{j_{1}}/\lambda_{j_N}}) & \frac{1}{2}b(\sqrt{\lambda_{j_{N}}/\lambda_{j_2}}+\sqrt{\lambda_{j_{2}}/\lambda_{j_N}}) & \cdots & a\\
      \end{array}
    \right)
$$
where $j_{p}\in J,\;\;p=1, 2, \cdots, N.$
Recall $\sqrt{\lambda_{j}}=\sqrt{\frac{\parallel j\parallel^{2}}{1+\parallel j\parallel^{2}}}$.
(Note $\parallel j \parallel=\sqrt{\tau_{1}^{2}j_{1}^{2}+\cdots+\tau_{d}^{2}j_{d}^{2}}\sim |j|.$) So
\be\label{yuanbu2}\sqrt{\lambda_{j}}=\sqrt{1-\frac{1}{1+\parallel j\parallel^{2}}}=1+O(\frac{1}{L^{2}}),\;\;j\in J.\ee
It follows that
$$B=\left(
      \begin{array}{cccc}
        a & b & \cdots & b \\
        b & a & \cdots & b \\
        \vdots & \vdots & \ddots & \vdots \\
        b & b & \cdots & a \\
      \end{array}
    \right)+O(\frac{1}{L^{2}}).
$$
Then
\begin{eqnarray}
\text{det} B&=&(a-b)^{N-1}(a+4b)+O(\frac{N}{L^{2}})\neq 0,\quad L\gg N,\nonumber\end{eqnarray} and
\begin{eqnarray}\label{L70}B^{-1}&=&\frac{1}{(a-b)(a+(N-1)b)}\left(
                                  \begin{array}{cccc}
                                    a+(N-2)b & -b & \cdots & -b \\
                                    -b & a+(N-2)b & \cdots &  -b\\
                                    \vdots&  \vdots&  \ddots& \vdots\\
                                    -b& -b& \cdots & a+(N-2)b\\
                                  \end{array}
                                \right)+O(\frac{N}{L^{2}})\nonumber\\
&:=&B_{0}^{-1}+O(\frac{N}{L^{2}}).
\end{eqnarray}
That is, $\parallel B^{-1}-B_{0}^{-1}\parallel\leq \frac{C}{L},$ where $C=C(N)$
depends on $N$.
Take $\bar\Pi=[\sqrt{\e_0},\,2\sqrt{\e_0}]^N$. Then $\text{Measure}\, \bar\Pi=\e_0^{N/2}$ and $\text{Diameter}\,\bar\Pi=\e_0^{1/2}$. Define $\Phi:\zeta\mapsto \xi$ by
\[T B\, \zeta=\xi.\]
And let $\Pi=\Phi(\bar\Pi)$. Thus,
\[\frac{1}{C}\e_0^{N/4}\le\,\text{Measure}\,\Pi\le C\e_0^{N/4},\,\quad \frac{1}{C}\e_0 \le \text{Diameter}\,\Pi\le C\e_0,\] and
\be\label{2018-3-19-1}\omega^{0}(\xi)=\lambda^{(N)}+\xi,\ee
\be\label{2018-3-19-2}\Omega^{0}(\xi)=\lambda^{\infty}+SB^{-1}\xi:=(\Omega^0_j(\xi):j\in\mathbb Z_+^d\setminus\, J),\ee
\be\label{2018-3-19-3}H_0=(\omega^0(\xi),y)+\sum_{j\in\mathbb Z_+^d\setminus\, J} \Omega^0_j(\xi)\,z_z\, \bar z_j. \ee
By \eqref{2018-3-19-1}, Assumption {\bf A} is obviously fulfilled.
Recall
\begin{equation}\label{L71}
S=\left(s_{il}=\frac{b}{2}(\sqrt{\lambda_{i}/\lambda_l}+\sqrt{\lambda_{l}/\lambda_i}):\;\;i\in \mathbb{Z}_{+}^{d}\setminus J,\;\;l\in J\right),
\end{equation}
where $k$ denotes the row index of $S$ and $l$ the column index of $S$ and recall
\be \label{2018-3-19-5} \lambda^\infty=\left( \sqrt{\lambda_j}=\sqrt{\frac{||j||^2}{1+||j||^2}}=1-O(\frac{1}{|j|^2}):\; j\in\mathbb Z_+^d\setminus\, J\right)\ee
 Write $B^{-1}=(b_{kl}:k,l\in J)$. By \eqref{2018-3-19-2}, \eqref{L71} and \eqref{2018-3-19-5},
 we have
 \[\Omega_j^0=\varpi+O(|j|^{-\kappa}),\; \kappa=2,\, \varpi=1+\frac{b}{2}\sum_{k,l\in J}b_{kl}\xi_l.\] This verifies that \eqref{3-08-morning+} and \eqref{170512-1+} of Assumption ${\bf B^\star}$ hold true.

For a matrix $X,$ by $X(k)$ denote the $k-$th row of $X$. Then by \eqref{yuanbu2},\eqref{L70} and\eqref{L71}, we get
\begin{eqnarray}\label{L72}
&&(SB^{-1})(i)=(SB_{0}^{-1})(i)+O(\frac{N}{L^{2}})\nonumber\\
&=&\frac{b(\lambda_{i}+1)}{2\sqrt{\lambda_i}}\left(\left(
                                      \begin{array}{ccc}
                                        1 & \cdots & 1 \\
                                        1 & \cdots & 1 \\
                                        \vdots & \ddots & \vdots \\
                                        1 & \cdots & 1 \\
                                        \vdots & \vdots & \vdots \\
                                      \end{array}
                                    \right)B_{0}^{-1}\right)(i)+O(\frac{N}{L^{2}})\nonumber\\
&=&\frac{b(1+\lambda_{i})}{2\sqrt{\lambda_i}\,(a-b)(a+(N-1)b)}(-b(N-1)+a+(N-2)b, \cdots, -b(N-1)+a+(N-2)b)+O(\frac{N}{L^{2}})\nonumber\\
&=& \frac{b(1+\lambda_{i})}{2\sqrt{\lambda_i}\,(a+(N-1)b)}(1, \cdots, 1)+O(\frac{N}{L^{2}}),
\end{eqnarray}
where we have used $\sqrt{\lambda_{l}}=1+O(\frac{1}{L^{2}})$ with $l\in J.$ For any $k\in\mathbb{Z}^{N}\setminus \{0\},$ we
assume $k_{1}\neq 0$ without loss of generality.
Arbitrarily take an infinite dimensional integer vector $l=(l_{j}\in \mathbb{Z}: j\in \mathbb{Z}_{+}^{d}\setminus J)$ with
$|l|=\sum_{j\in\mathbb{Z}_{+}^{d}\setminus J}|l_{j}|\leq 2.$
Let $\langle l, \Omega^{0}(\xi)\rangle=\sum_{j\in\mathbb{Z}_{+}^{d}\setminus J}l_{j}\Omega_{j}^{0}(\xi)$.
Then
$$\left | \frac{d}{d\xi_{1}}\left(\langle k, \omega^{0}(\xi)\rangle+\langle l, \Omega^{0}(\xi)\rangle \right)\right |
=\left | k_{1}+\sum_{j\in\mathbb{Z}^{d}_{+}\setminus J}l_{j}\frac{b^{2}(1+\lambda_{j})}{2\sqrt{\lambda_j}\,(a+(N-1)b)}\right |+O(\frac{N}{L^{2}}).$$
Note $||j||^2=\tau_1^2\, j_1^2+\cdots+\tau_d\, j_d^2 \ge 1.$
So
 $$\left|\sum_{j\in\mathbb{Z}^{d}_{+}\setminus J}l_{j}\frac{b(1+\lambda_{j})}{2\sqrt{\lambda_j}\,(a+(N-1)b)}\right|\leq \frac12\sqrt{\frac{7}{2}}\,\frac{b}{a+(N-1)b}<\frac{b}{a+(N-1)b}.$$
Assume $N\geq 2.$ We have
\begin{eqnarray}\label{L73}
\left| \frac{d}{d\xi_{1}}\left(\langle k, \omega^{0}(\xi)\rangle+\langle l, \Omega^{0}(\xi)\rangle \right)\right |
&\geq & 1-\frac{b}{a+(N-1)b}+O(\frac{N}{L^{2}})\nonumber\\
& \geq &\frac{a+(N-2)b}{a+(N-1)b}+O(\frac{N}{L^{2}})\nonumber\\&\geq & C(N)>0,\quad \text{for}\quad L\gg N.
\end{eqnarray}
This verifies condition \eqref{170510-2} of Assumption ${\bf B^\star}$. Let $q=p>d/2$ and
\[D_p:=D_p(\e_0)=\{(x,y,\hat z,\bar {\hat{z}})\in \, \mathbb C^N/(2\pi\mathbb Z)^N\times \mathbb Z^N\times h_p\times h_p:\; |\Im\, x|\le s_0,\, |y|\le \e_0,\, ||\hat z||_p\le \e_0^{1/3},\,||\overline{\hat z}||_p\le \e_0^{1/3}\}.\]
By \eqref{2018-3-18-1}, \eqref{sym2}, \eqref{sym3} and \eqref{2018-3-18-6}  we verify Assumption  {\bf C} and \be \W \lf X_{R^0}\rc\W_{q,D_p\times\Pi} \le C\,\epsilon_0,\; \W \lf \p_\xi\,X_{ R^0}\rc\W_{q,D_p\times\Pi} \le C\, \sqrt{\epsilon_0}.\ee
  It follows from \eqref{L19}, \eqref{Real} and \eqref{sym2}  that Assumption {\bf  D} holds true. Finally,  Assumption {\bf E} holds true clearly, since $B=0$. Using Theorem \ref{theorem2+2} we have

\begin{thm} \label{on-BBM}\label{on-PC}Assume $(\frac{2\pi}{T_j}:j\in\mathbb Z_+^d)$ is in $\tilde\Theta$. Around the neighborhood of  $u=0$,  gPC equation \eqref{L1} has many (the initial value set of $N$-dimensional positive Lebesgue measure) smooth solutions  which are  quasi-periodic in time, linear stable and of zero Lyapunov exponent. More exactly, there exists $\e_0^*=\e_0^*(N,\tau, J)>0$ depending on $N,\tau, J$ such that for any $0<\e_0<\e_0^*$ there is a subset $\breve{\Pi}$ of the initial value set $\Pi_0:=[\sqrt{\e_0},2\sqrt{\e_0}]^N$ with
\[\text{Leb}\;\breve{\Pi}=(\text{Leb}\, \Pi_0)\; \left(1-C \frac{1}{|\log\,\e_0|}\right) \] and for any $\xi=(\xi_l:\,l=1,...,N)\in\breve{\Pi}$, gPC equation has a quasi-periodic solution $u(t,x)$ of frequency $\omega\in\mathbb R^N$ in time $t$
\[u(t,x)=\sum_{k\in\mathbb Z^N,j\in\mathbb Z\setminus\{0\}}\;\hat{u}(k,j)\; e^{{\bf i}(k,\omega)}\, \phi_j(x)\] with\[\phi_n(x)=
\sin n_{1}\tau_{1}x_{1}\cdots \sin n_{d}\tau_{d}x_{d}, \; \;\forall\;\;n=(n_1,...,n_d)\in\mathbb Z_+^d, \] satisfying
\[|\omega-\omega_0|\le C\sqrt{\e_0},\; \omega\in\mathbb R^N, \; \omega_0=\left(\sqrt{\frac{||j_l||^2}{1+|| j_l||^2}}:\; j_l\in J\right)\in\mathbb R^N,
\]
\[ \left| \hat{u}(e_l,j_l)-\xi_l\right|<C\, \e_0^{1/3},\; e_l-l^{\text{th}}\, \text{unit vector of}\; \mathbb Z^N,\;
j_l\in J,\,l=1,...,N, \] and
\[\sum_{(k,j)\notin \mathcal{S}}\left|\hat{u}(k,j \right|^2e^{|k|\, s_0+2a\,|j|}|j|^{2p}<C\,\e_0^{1/3},\quad \mathcal{S}={(e_l,j_l):\; l=1,...,N},\]
where some constants $s_0>0,a>0$ and $p>d/2$.
\end{thm}
\begin{rem} Theorem  \ref{theorem2+2} applies to more general PC equation:
\begin{equation}\label{LL1}
\left\{
  \begin{array}{ll}
    u_{tt}-\Delta u-\Delta u_{tt} +F(u,\Delta u)=0,\;\;x\in \Omega\subset\mathbb{R}^{d},\\
    u|_{\partial \Omega}=0,
  \end{array}
\right.
\end{equation}
where $F(u)=\sum_{j+j\ge 3} c_{ij}\, u^i\,(\Delta u)^j$ is an analytic function of $u$ and $\Delta\, u$ with $c_{ij}\in\mathbb R$.
\end{rem}
\section{Final Remark on global solutions to BBM and gPC}

 Up to now there have been a lot of works on the existence and long-time asymptotic behavior as well as traveling solutions for BBM and gPC where the spatial variable is in the whole space $\mathbb R^d$. See \cite{Naher, PBILER, mwang, JALBERT, Shaomei-Fang} for BBM and \cite{Triki, Shawagfeh,Zhao-Zhang,Xu-Liu} for gPC and the references therein. According to our knowledge, there is not any results on the existence of solutions, let alone long-time  behavior
of solutions, for BBM and gPC when the spatial variable $x$ is in some compact space, $\mathbb T^d$, say.

By Theorems 9.5 and 10.5, $1$-dimensional BBM and $d$-dimensional gPC equations subject to typical periodic boundary conditions (the spatial variable in compact space)  have many quasi-periodic solutions with initial values of positive finite dimensional Lebesgue measure. These quasi-periodic solutions are of course global and of recurrent property. As done in \cite{cong-liu-yuan}, Those solutions whose initial date close to any quasi-periodic solution are almost global, that is, assuming $u_0(t,x)$ with initial datum $u_0(0,x)$ is a quasi-periodic solution for BBM or gPC equation subject the typical boundary conditions, then any solution $u(t,x)$ with initial value satisfying $||u(0,x)-u_0(0,x)||_p<\delta$ with any $0<\delta\ll 1$ obeys that the solution $u(t,x)$ exists for time $|t|<L\delta^{-1}$ and
\[||u(t,\cdot)-u_0(t,\cdot)||_p\le C \delta,\quad \forall\;\; |t|<\delta^{-1}.\]

\section{ Appendices}

Let $\ti p, \ti q>d/2$. For a linear operator $L:\; h_{\ti p}\to h_{\ti q}$, denote by $L_{ij}$'s   the matrix elements of $L$. Given an index set $I\subset\mathbb Z^d$.
Partition $L$ as follows
\[L=\begin{pmatrix} L^{(11)}& L^{(12)}\\ L^{(21)}& L^{(22)}\end{pmatrix},\]
where
\[L^{(11)}=(L_{ij}:\; i\in I,j\in I),\; L^{(12)}=(L_{ij}:\; i\in I,j\in \mathbb Z^d\setminus I),\] \[ L^{(21)}=(L_{ij}:\; i\in \mathbb Z^d\setminus I,j\in I),\; L^{(22)}=(L_{ij}:\; i\in \mathbb Z^d\setminus I,j\in \mathbb Z^d\setminus I).\]
Expand $L^{(ij)}$ ($i,j\in\{1,2\})$ to $\tilde L^{(ij)}$ as follows
\[\tilde L^{(11)}=\begin{pmatrix} L^{(11)}& 0\\ 0& 0\end{pmatrix},\;
\tilde L^{(12)}=\begin{pmatrix}0 & L^{(12)}\\ 0& 0\end{pmatrix},\;
\tilde L^{(21)}=\begin{pmatrix}L^{(21)} & 0\\ 0& 0\end{pmatrix},\; \tilde L^{(22)}=\begin{pmatrix}0 & 0\\ 0& L^{(22)}\end{pmatrix}. \]
Define
\[||L^{(ij)}||_{h_{\ti p}^i\to h_{\ti q}^j}=||\ti L^{(ij)}||_{h_{\ti p}\to h_{\ti q}}.\]
According to the partition of $L$, split the space $h_{\ti p}$:
\[h_{\tilde p}=h_{\ti p}^1\oplus h_{\ti p}^2,\]
where $h_{\ti p}^1=\{(z_j\in\mathbb C:\; j\in I)\}$ and $h_{\ti p}^2=\{(z_j\in\mathbb C:\; j\in \mathbb Z^d\setminus I)\}$. And define
\[||z||_{h_{\ti p}^1}^2=\sum_{j\in I} |j|^{2\ti p} |z_j|^2,\;\; z\in h_{\ti p}^1,\]
and
\[||z||_{h_{\ti p}^2}^2=\sum_{j\in \mathbb Z^d\setminus I} |j|^{2\ti p} |z_j|^2,\;\; z\in h_{\ti p}^2.\]

\begin{lem} \label{splitlemma} For any $\ti p,\ti q\in\{p,q=p+\kappa\}$ and any $i,j\in\{1,2\}$, we have
\[||L^{(ij)}||_{h_{\ti p}^i\to h_{\ti q}^j}\le ||L||_{h_{\ti p}\to h_{\ti q}}.\]
\end{lem}
\begin{proof} The proof can be found in P. 104, \cite{Y}.
\end{proof}
\begin{lem}  Assuming that $X$ is self-adjoint in $\ell_2$ and assuming that $\operatorname{Dim}\, X<\infty$ , we have
\be \label{17-15-3}||X||_{\ell_2\to\ell_2}\le ||X||_{h_p\to h_p},\; \forall\,\, p>0.\ee

\end{lem}
\begin{proof} 
 Let $\lambda$ be any eigenvalue of $X$ and $x_0$ be the eigenvector with the eigenvalue $\lambda$. Since $X$ is self-adjoint, $||X||_{\ell_2\to\ell_2}=\sup\,\{|\lambda|\}$. Write $X=(X_{ij}:\, i,j\in\Xi)$. Let $I=\operatorname{diag}\, (|j|^p:\, j\in\Xi)$. Then
\[||X||_{h_p\to h_p}=\sup_{x\neq 0}\frac{||X\,x||_{h_p}}{||x||_{h_p}}=
\sup_{x\neq 0}\frac{||I\,X\, x||_{\ell_2}}{||I\, x||_{\ell_2}}\ge \frac{||I\,X \, x_0||_{\ell_2}}{||I\, x_0||_{\ell_2}}=\frac{||I\,\lambda \, x_0||_{\ell_2}}{||I\, x_0||_{\ell_2}}=|\lambda|.\]
Thus
\[||X||_{h_p\to h_p}\ge \sup\,\{|\lambda|\}=||X||_{\ell_2\to\ell_2}. \]

\color{black}
\end{proof}

\begin{lem} \label{absolute-summation}Assume $\lf X\rc$ and $\lf Y\rc$ are two bounded operator from $h_{\ti p}$ to $h_{\ti q}$ where $\ti p,\ti q\in\{p,q,0\}$. Then
\[||\lf X+Y\rc||_{h_{\ti p}\to h_{\ti q}}\le ||\lf X\rc ||_{h_{\ti p}\to h_{\ti q}}+||\lf Y\rc||_{h_{\ti p}\to h_{\ti q}}\] and
\[||\lf X\, Y\rc||_{h_{\ti p}\to h_{\ti p}}\le ||\lf X\rc ||_{h_{\ti p}\to h_{\ti p}}\, ||\lf Y\rc ||_{h_{\ti p}\to h_{\ti p}}  \]

\end{lem}
\begin{proof} The proof is easily verified by the definitions of $||\cdot||_{h_{\ti p}\to h_{\ti q}}$ and $\lf \cdot \rc$. We omit it. \end{proof}

\begin{lem}\label{finite-norm} For a finite dimensional matrix $X=(X_{ij}:\; i,j\in\mathbb Z^d,\; |i|, |j|\le \Gamma),$ where $\Gamma$ is a fixed constant, then
\[||\lf X\rc ||_{h_{\ti p}\to h_{\ti q}} \le \Gamma^{d/2}\, ||X||_{h_{\ti p}\to h_{\ti q}},\]
where $h_{\ti p}$ is a space of finite dimensional vectors:
\[h_{\ti p}=\{z=(z_j\in\mathbb C:\; j\in\mathbb Z^d, |j|\le \Gamma) \}\]
with
\[||z||^2_{h_{\ti p}}=\sum_{|j|\le \Gamma} |j|^{2\ti p}|z_j|^2.\]
\end{lem}
\begin{proof}
Introducing a weight $w_{ij}=|i|^{\ti p}|j|^{-\ti q}$. Let $\tilde X_{ij}=w_{ij}\, X_{ij}$. Then
\[||\tilde X||_{h_0\to h_0}=||X||_{h_{\ti p}\to h_{\ti q}},\; ||\lf \tilde X\rc||_{h_0\to h_0}=||\lf X\rc||_{h_{\ti p}\to h_{\ti q}}.\]
Let $\delta=||\tilde X||_{h_0\to h_0}$. Then
\[\max_{|i|\le \Gamma} \sum_{|j|\le \Gamma}|\tilde X_{ij}|^2\le \delta^2,\; \max_{|j|\le \Gamma} \sum_{|i|\le \Gamma}|\tilde X_{ij}|^2\le \delta^2.\]
For any $u=(u_j:\; |j|\le \Gamma)$ with $||u||_{h_0}=1$,
\[\begin{array}{lll} ||\lf X\, u\rc ||_{h_0\to h_0}^2&=& \sum_{|i|\le \Gamma}\left|\sum_{|j|\le \Gamma} |\tilde X_{ik}|\, u_j\right|^2\\
&\le & \sum_{|i|\le \Gamma} \left( \sum_{|j|\le \Gamma}|\tilde X_{ij}|^2\right)\left( \sum_{|j|\le \Gamma}|u_j|^2\right)\\
&\le &\sum_{|i|\le \Gamma}\delta^2 \\
&\leq &\Gamma^d\, \delta^2.\end{array}\]
This completes the proof.

\end{proof}



\section*{Acknowledgements} I am very grateful to the referee for the invaluable suggestions.
This article is a revised and updated version, replacing a $d$-dimensional generalized BBM equation by a $d$-dimensional generalized PC equation, of the 2017-August version ``KAM theorem with normal frequencies clustering at zero for some shallow water equations".
The author began to conceive this article when visiting the Mittag-Leffler Institute in 2010
invited by Professor H. Eliasson and Professor J.-C. Yoccoz. In 2015, the author was invited by Professor S. Kuksin to
 report the early version of this article at the Euler Institute. In 2016,
 the author was invited by Professor F. Meng to visit Qufu Normal University where the author
 made some revisions. Then the author had a beneficial discussion with Professor J. Liu.
 In this long writing process, the author was also encouraged by Professor W. Craig, Professor D. Bambusi,
 and Professor B. Grebert. In addition, Professor M. Gao and Professor J. Li helped the author edit part
 of the manuscript with latex. The author would like to express his sincere gratitude to
 all the professors mentioned above.  The author is also thankful to Dr. K. Zhang  for her help in computing
  the normal form of BBM and to Professor H. Cong and Dr. Y. Shi for their reading the manuscript.

 \end{document}